\documentclass[12pt,
  oneside]{amsart}

\usepackage{amscd, amsmath, colonequals}

\usepackage[usenames]{xcolor}

\title[]{Real submanifolds of maximum complex tangent space at a CR singular point, II}
\author[]{Xianghong Gong}

\address{Department of Mathematics,
 University of Wisconsin-Madison, Madison, WI 53706, U.S.A.}
 \email{gong@math.wisc.edu}

\author{Laurent Stolovitch}
\address{CNRS and Laboratoire J.-A. Dieudonn\'e
U.M.R. 6621, Universit\'e de Nice - Sophia Antipolis, Parc Valrose
06108 Nice Cedex 02, France.}
\email{stolo@unice.fr}
\thanks{Research of L. Stolovitch was supported by ANR grant ``ANR-10-BLAN 0102'' for the project DynPDE and "ANR-14-CE34-0002-01" for the project "Dynamics and CR geometry"}


\newtheorem{thm}{Theorem}[section]
\newtheorem{cor}[thm]{Corollary}
\newtheorem{prop}[thm]{Proposition}
\newtheorem{lemma}[thm]{Lemma}

\newcommand{\sbt}{\,\begin{picture}(-1,1)(-1,-3)\circle*{3}\end{picture}\ }

\newcommand{\diag}{\operatorname{diag}}

\theoremstyle{definition}

\newtheorem{defn}[thm]{Definition}
\newtheorem{exmp}[thm]{Example}

\newtheorem{rem}[thm]{Remark}

\renewcommand{\th}[1]{\begin{thm}\label{#1}}
\newcommand{\eth}{\end{thm}}
\newcommand{\co}[1]{\begin{cor}\label{#1}}
\newcommand{\eco}{\end{cor}}
\renewcommand{\le}[1]{\begin{lemma}\label{#1}}
\newcommand{\ele}{\end{lemma}}
\newcommand{\pr}[1]{\begin{prop}\label{#1}}
\newcommand{\epr}{\end{prop}}

\newcommand{\ga}{\begin{gather}}
\newcommand{\ega}{\end{gather}}
\newcommand{\gan}{\begin{gather*}}
\newcommand{\egan}{\end{gather*}}
\newcommand{\al}{\begin{align}}
\newcommand{\eal}{\end{align}}
\newcommand{\aln}{\begin{align*}}
\newcommand{\ealn}{\end{align*}}
\newcommand{\eq}[1]{\begin{equation}\label{#1}}
\newcommand{\eeq}{\end{equation}}

\newcommand{\ci}{~\cite}

\newcommand{\f}[2]{\frac{#1}{#2}}

\newcommand{\fix}{\operatorname{Fix}}

\newcommand{\cc}{{\bf C}}
\newcommand{\nn}{{\bf N}}
\newcommand{\zz}{{\bf Z}}
\newcommand{\rr}{{\bf R}}
\newcommand{\qq}{{\bf Q}}

\newcommand{\ov}{\overline}
\newcommand{\ord}{\operatorname{ord}}

\newcommand{\id}{\operatorname{I}}

\newcommand{\RE}{\operatorname{Re}}
\newcommand{\IM}{\operatorname{Im}}

\newcommand{\cL}{\mathcal}

\newcommand{\I}{\operatorname{I}}

\newcommand{\gaa}{\gamma}

\newcommand{\del}{\delta}
\newcommand{\Del}{\Delta}
\newcommand{\var}{\varphi}
\newcommand{\e}{\epsilon}
\newcommand{\om}{\omega}
\newcommand{\Om}{\Omega}

\newcommand{\la}{\lambda}

\newcommand{\pd}{\partial}

\newcommand{\yt}{\frac{1}{2}}

\newcommand{\re}[1]{(\ref{#1})}
\newcommand{\rea}[1]{$(\ref{#1})$}
\newcommand{\rl}[1]{Lemma~\ref{#1}}
\newcommand{\nrc}[1]{Corollary~\ref{#1}}
\newcommand{\rp}[1]{Proposition~\ref{#1}}
\newcommand{\rt}[1]{Theorem~\ref{#1}}

\newcommand{\rla}[1]{Lemma~$\ref{#1}$}

\newcommand{\rpa}[1]{Proposition~$\ref{#1}$}
\newcommand{\rta}[1]{Theorem~$\ref{#1}$}

\newcounter{pp}
\newcommand{\bpp}{\begin{list}{$\hspace{-1em}\alph{pp})$}{\usecounter{pp}}}
\newcommand{\epp}{\end{list}}

\newcounter{ppp}
\newcommand{\bppp}{\begin{list}{$\hspace{-1em}(\roman{ppp})$}{\usecounter{ppp}}}
\newcommand{\eppp}{\end{list}}

\def\beq{\begin{equation}}
\def\eeq{\end{equation}}

\begin{document}


\begin{abstract} We study a germ of  real analytic $n$-dimensional submanifold  of ${\mathbf C}^n$ that has a complex tangent space of maximal dimension at a CR singularity.
Under the condition that its complexification admits the maximum number of deck transformations, we first  classify holomorphically its quadratic CR singularity.
We then study its transformation to a normal form under the action of local (possibly formal) biholomorphisms at the singularity. 
We first conjugate formally its associated reversible map $\sigma$ to suitable normal forms and show that all these normal forms can be divergent. We then construct a unique formal normal form under a non degeneracy condition.
\end{abstract}

\date{\today}
 \maketitle

\tableofcontents

\addtocontents{toc}{\protect\setcounter{tocdepth}{1}}

\setcounter{section}{0}
\setcounter{thm}{0}\setcounter{equation}{0}
\section{Introduction and main results}

\subsection{Introduction}
We say that a point $x_0$ in
 a  real submanifold $M$ in $\cc^n$ is  a CR singularity,
 if the complex tangent spaces $T_xM\cap J_xT_xM$  do not
have a constant dimension in any neighborhood of $x_0$.  The study of real submanifolds with CR singularities was initiated by E.~Bishop in his pioneering work~\cite{Bi65},
when the complex tangent space of $M$ at a CR singularity is minimal, that is exactly one-dimensional.
 The very elementary models of this kind of manifolds are   classified as    the Bishop quadrics in $\cc^2$,  given by
\eq{bishopq}
Q \colon z_2=|z_1|^2+\gaa(z_1^2+\ov z_1^2), \quad 0\leq\gaa<\infty; \quad
Q\colon z_2=z_1^2+\ov z_1^2, \quad \gaa=\infty
\eeq
with Bishop invariant $\gaa$. The origin is a complex tangent which is said to be {\it elliptic} if $0\leq\gaa<1/2$,  {\it parabolic} if $\gaa=1/2$, or {\it hyperbolic} if $\gaa>1/2$.

In ~\cite{MW83}, Moser and Webster studied the normal form problem of  a real analytic surface
$M$ in $\cc^2$
which is the higher order perturbation of $Q$. They showed that
 when $0<\gaa<1/2$, $M$ is holomorphically equivalent to a normal form which
is  an algebraic surface that depends only on $\gaa$ and two discrete invariants. They also constructed a
formal normal form of $M$ when the origin is a    non-exceptional hyperbolic complex tangent point;
 although the normal form is still convergent, they showed that the normalization is divergent in general for the hyperbolic case.
In fact, Moser-Webster  dealt  with an $n$-dimensional
real submanifold $M$ in $\cc^n$,  of which
the complex tangent space has   (minimum) dimension $1$
at a CR singularity.  When $n>2$, they also found normal forms under suitable non-degeneracy condition.

In this paper we continue our previous investigation on an $n$-dimensional real analytic  submanifold $M$ in $\cc^n$ of which the complex tangent space has the
{\it largest}
possible
dimension at a  given CR singularity~\ci{GS15}. The dimension must be $p=n/2$. Therefore,  $n=2p$ is even. As shown in~\ci{St07} and~\ci{GS15}, there is yet another basic quadratic model
 \eq{Qgam}
 Q_{\gamma_s}\subset\cc^4\colon z_{3}= (z_1+2\gamma_s\ov z_{2})^2,
\quad z_{4}=( z_{2}+2
(1-\ov\gamma_{s})  \ov z_{1})^{2}
\eeq
with $\gaa_s$ an    invariant satisfying   $\RE\gaa_s\leq1/2$,  $\IM\gaa_s\geq0$, and $\gaa_s\neq0$.  The complex tangent at the origin is said of complex type.
In~\ci{GS15}, we obtained convergence of normalization for {\rm abelian}   CR singularity.
In this paper, we study systematically the normal forms of the manifolds $M$ under the condition that $M$ admit the maximum number of deck transformations,  condition~D, introduced in~\ci{GS15}.

In  suitable holomorphic coordinates, a $2p$-dimensional real analytic submanifold 
in $\cc^{2p}$
 that has a complex tangent space of maximum dimension at the origin  is given by
\ga\label{mzpjintr}\nonumber
M\colon z_{p+j}=E_j(z',\ov z'),
\quad 1\leq j\leq p,
\\
\label{ejzp}\nonumber
E_j(z',\ov z')=h_j(z',\ov z')+q_j(\ov z') + O(|(z',\ov z')|^3), 
\end{gather}
where $z'=(z_1,\ldots, z_p)$,   each  $h_j(z',\ov z')$ is a homogeneous quadratic polynomial in $z',\ov z'$ without holomorphic or anti-holomorphic terms,
  and each $q_j(\ov z')$ is a homogeneous quadratic polynomial in $\ov z'$.
  We call $M$ a  {\em quadratic manifold} in $\cc^{2p}$ if $E_j$ are homogeneous quadratic polynomials. If $M$  is a product of Bishop quadrics \re{bishopq} and quadrics of the form \re{Qgam}, it  is called a {\em product quadric}.

 \subsection{Basic invariants}  We first describe some basic invariants of real analytic submanifolds, which are essential to the normal forms.
To study $M$, we consider its complexification in $\cc^{2p}\times\cc^{2p}$ defined by
\begin{equation}
{\mathcal M}\colon
\begin{cases}
z_{p+i} = E_{i}(z',w'), & i=1,\ldots, p,
\\
w_{p+i} = \bar E_i(w',z'),& i=1,\ldots, p.\\
\end{cases}
\nonumber
\end{equation}
 It is a complex submanifold of complex
dimension $2p$ with coordinates $(z',w')\in\cc^{2p}$.  Let $\pi_1,\pi_2$ be the restrictions of the projections $(z,w)\to z$
and   $(z,w)\to w$ to $\cL M$, respectively.
Note that $\pi_2=C\pi_1\rho_0$, where  $\rho_0$ is the restriction to $\cL M$ of the anti-holomorphic involution $(z,w)\to(\ov w,\ov z)$ and $C$ is the complex conjugate. 
 It is proved in ~\cite{GS15} that when $M$ satisfies {\it condition B},  i.e. $q^{-1}(0)=0$,
 the deck transformations of $\pi_1$ are involutions that commute pairwise, while the number of deck transformations can be $2^\ell$ for $1\leq\ell\leq p$.  As in~\ci{GS15}, our basic hypothesis on $M$ is {\it condition $D$} that $\pi_1$ admits the maximum number, $2^p$, deck transformations. Then it is proved in~\ci{GS15} that the group of deck transformations of $\pi_1$ is generated uniquely by $p$ involutions $\tau_{11},\dots, \tau_{1p}$ such that each $\tau_{1j}$ fixes a
hypersurface in $\cL M$. Furthermore,
$$
\tau_1:=\tau_{11}\dots\tau_{1p}
$$
is the unique deck transformation of which the set of the fixed-points has the smallest dimension $p$.   We call $\{\tau_{11},\ldots, \tau_{1p},\rho_0\}$ the set of {\it Moser-Webster involutions}.   Let  $\tau_2=\rho_0\tau_1\rho_0$ and
  $$
 \sigma=\tau_1\tau_2.
 $$
 Then $\sigma$ is {\it reversible} by $\tau_j$ and $\rho_0$, i.e.
 $\sigma^{-1}=\tau_j\sigma\tau_j^{-1}$ and $\sigma^{-1}=\rho_0\sigma\rho_0$.

In this paper  for classification purposes, we will impose the following condition:

 \medskip
 \noindent
{\bf Condition E.} {\it   $M$ has distinct eigenvalues, i.e.  $\sigma$ has
  $2p$ distinct eigenvalues.}

\medskip

 We now introduce our main results.

Our first step is to normalize   $\{\tau_1,\tau_2, 
\rho_0\}$.  When $p=1$, this normalization is   the main step in order to obtain the Moser-Webster normal form; in fact a simple further normalization allows Moser and Webster to achieve a convergent normal form under a suitable non-resonance condition even for the  non-exceptional  hyperbolic complex tangent.

When $p>1$, we need to carry out a further normalization for $\{\tau_{11},\ldots, \tau_{1p},\rho_0\}$; this is our second step. Here the normalization has a large degree of freedom  as shown by  our  formal and convergence results.


\subsection{A normal form of quadrics}
In section~\ref{secquad}, we study all quadrics which admit the maximum number of deck transformations. For such quadrics,
 all deck transformations are linear.
Under condition E,  we will first normalize 
$\sigma, \tau_1,\tau_2$ and $\rho_0$
into $\hat S, \hat T_1,\hat T_2$ and  $\rho$ where
\begin{equation}
\begin{array}{rrclrcl}
  \hat T_1\colon  &  \xi_j'   &\!\! = \!\!&  \la_j^{-1}\eta_j,  &\quad \eta_j'     &\!\!\! = \!\!\!& \la_j\xi_j, \\ 
  \hat T_2\colon  & \xi_j'    &\!\! = \!\!&  \la_j\eta_j,          &\quad  \eta_j'        &\!\!\! = \!\!\!& \la_j^{-1}\xi_j, \\ 
\hat S\colon       &\xi_j'     &\!\! = \!\!&  \mu_j\xi_j,          &\quad  \eta_j'       &\!\!\! = \!\!\!& \mu_j^{-1}\eta_j 
\end{array}
\nonumber
\end{equation}
with
\eq{muss}\nonumber
 \la_e>1,   \quad |\la_h|=1, \quad |\la_s|>1, \quad\la_{s+s_*}=\ov\la_s^{-1}, \quad \mu_j=\la_j^2.
\eeq
Here $1\leq j\leq p$.  Throughout the paper,   the indices $e,h,s$ have the ranges:  $1\leq e\leq e_*$,
$ e_*<h\leq e_*+h_*$,  $ e_*+h_*<s\leq p-s_*$. Thus $e_*+h_*+2s_*=p$.
We will call $e_*,h_*, s_*$
the numbers
of {\it elliptic, hyperbolic} and {\it complex} components of a product quadric, respectively.
As in the Moser-Webster theory, at the complex tangent (the origin)  an {\it elliptic} component of a product quadric corresponds
a  {\it hyperbolic} component of $\hat S$, while a {\it hyperbolic} component of the quadric corresponds an {\it elliptic}
component of $\hat S$.
 On the other hand, a {\it complex} component of the quadric    behaves like an elliptic component when the CR singularity is abelian, and it also behaves like a hyperbolic components for the existence of attached complex manifolds; see~\ci{GS15} for details.

For the above normal form of $\hat T_1,\hat T_2$ and $\hat S$,  we always normalize the anti-holomorphic involution
$\rho_0$ as
\begin{gather}\label{57rhoz}
\rho\colon\left\{ \begin{array}{lcllcl}
\xi_e' &=&\ov \eta_e, \qquad&\eta_e'&=&\ov\xi_e,\\
\xi_h' &=&\ov \xi_h, & \eta_h'&=&\ov\eta_h,\\
\xi_s' &=&\ov\xi_{s+s_*}, & \eta_s'&=&\ov\eta_{s+s_*},\\
\xi_{s+s_*}'   &=&\ov \xi_s, & \eta_{s+s_*}'&=&\ov\eta_s.
\end{array} \right.
\end{gather}
 With the above normal forms $\hat T_1,\hat T_2,\hat S, \rho$ with $\hat S=\hat T_1\hat T_2$, we will
then normalize the  $\tau_{11},\ldots, \tau_{1p}$ under linear transformations that
commute with $\hat T_1, \hat T_2$, and $\rho$, i.e. the
linear transformations belonging to the {\it centralizer}
of  $\hat T_1,\hat T_2$ and $\rho$. This is a subtle  step.
Instead of normalizing the involutions directly,
we will use the  pairwise commutativity of  $\tau_{11}, \ldots, \tau_{1p}$ to associate to
these $p$ involutions 
a non-singular $p\times p$ matrix $\mathbf B$. The normalization
of $\{\tau_{11}, \ldots,\tau_{1p},
\rho\}$
 is then identified with the normalization of the matrices  $\mathbf B$ under a suitable
equivalence relation.   The latter  is easy to solve. Our normal form of $\{\tau_{11}, \ldots,\tau_{1p},
\rho\}$ is then constructed from the normal forms of $T_1,T_2,\rho$, and the matrix $\mathbf B$.
Following Moser-Webster~\cite{MW83}, we will construct the normal form of the quadrics 
 from the normal form of involutions.   Let us first state a Bishop type holomorphic classification for quadratic real manifolds.

\begin{thm}\label{thm1intr}Let $M$ be a quadratic submanifold defined by
$$
z_{p+j}=h_j(z',\ov z')+q_j(\ov z'),
\quad 1\leq j\leq p.
$$
Suppose that $M$ satisfies condition  $E$,
i.e.
  the branched covering
of  $\pi_1$ of complexification $\cL M$ has   $2^p$  deck transformations
and  
 $2p$ distinct eigenvalues. 
Then $M$ is holomorphically equivalent to
  \begin{gather}\label{qbgaintro}
  \nonumber
Q_{\mathbf B,\boldsymbol{\gamma}}\colon
z_{p+j}=L_j^2(z',\ov z'), \quad 1\leq j\leq p
\end{gather}
where $(L_1(z',\ov z'), \ldots, L_p(z',\ov z'))^t=\mathbf B(z'-2\boldsymbol{\gamma}\ov z')$,
 $\mathbf B\in GL_p(\cc)$ and
\begin{gather}\nonumber
\boldsymbol{\gamma}:=\begin{pmatrix}
  \boldsymbol{\gaa}_{e_*}    &\mathbf{0} &\mathbf{0} &\mathbf{0}\\
    \mathbf{0}&\boldsymbol\gaa_{h_*}&\mathbf{0}& \mathbf{0}  \\
   \mathbf{0}&\mathbf{0}&\mathbf{0}&\boldsymbol{\gaa}_{s_*} \\
   \mathbf{0}&\mathbf{0}&{\mathbf I_{s_*}}-\ov{\boldsymbol{\gaa}}_{s_*}&\mathbf{0}
\end{pmatrix}.
\end{gather}
Here $p=e_*+h_*+2s_*$, $\mathbf I_{s_*}$ denotes the $s_*\times s_*$ identity matrix, and
\gan \boldsymbol{\gaa}_{e_*}=\diag(\gamma_1,\ldots, \gamma_{e_*}), \quad
\boldsymbol{\gaa}_{h_*}=\diag(\gamma_{e_*+1},\ldots, \gamma_{e_*+h_*}),\\
\boldsymbol{\gaa}_{s_*}=\diag(\gamma_{e_*+h_*+1},\ldots, \gamma_{p-s_*})
\end{gather*}
with  $\gaa_e,\gaa_h,$ and $\gaa_s$ satisfying
\eq{0ge1}\nonumber
0<\gamma_e<1/2, \quad
1/2<\gamma_h<\infty, \quad \RE
\gaa_s<1/2, \quad  \IM\gaa_s>0.
\eeq
 Moreover, $\mathbf B$  is uniquely determined by   an equivalence
  relation
 $
 \mathbf B\sim \mathbf C\mathbf B\mathbf R
 $ for suitable non-singular matrices $\mathbf C,\mathbf R$ which have exactly $p$ non-zero entries.
 %
 %
\end{thm}

  When $\mathbf B$ is the identity matrix, we get a product quadric or its equivalent form.
See~\rt{quadclass} for detail of the equivalence relation. The scheme of finding quadratic
normal forms turns out to be useful.
It 
will be applied to the study of  normal forms of  the general real submanifolds.


\subsection{Formal submanifolds, formal involutions, and formal centralizers}
%
The normal forms of $\sigma$ turn out to be in the centralizer of $\hat S$, the  normal form of the
linear part of $\sigma$.
 The family is subject to a second step of normalization  
under mappings which again turn out to be in the centralizer  of $\hat S$. Thus, before we introduce normalization, we will first study various centralizers.
We will discuss the centralizer of $\hat S$ as well as the centralizer of $\{\hat T_1,\hat T_2\}$ in section~\ref{fsubm}.

\subsection{Normalization of $\sigma$}

As mentioned earlier, we will divide the normalization for the families of non-linear involutions  into two steps.
This division will serve  two purposes: first, it helps us to find the formal normal forms of the family of involutions $\{\tau_{11},\ldots, \tau_{1p},\rho\}$; second, it  helps us  understand
the convergence of normalization of the original normal form problem for the real submanifolds.  For purpose of normalization, we will assume that $M$
is {\it non-resonant}, i.e. $\sigma$ is {\it non-resonant}, if its eigenvalues $\mu_1, \ldots,\mu_p,  \mu_1^{-1}, \ldots, \mu_p^{-1}$ satisfy 
\eq{muqn1}
\mu^Q\neq1, \qquad \forall Q\in\zz^p, \quad |Q|\neq0.
\eeq

In section~\ref{secnfs}, we obtain the normalization of $\sigma$ by proving the following.
\begin{thm}\label{ideal0intr} Let $\sigma$ be a holomorphic map with linear part $\hat S$.
Assume that  $\hat S$ has eigenvalues $\mu_1, \dots, \mu_p, \mu_1^{-1}, \dots, \mu_p^{-1}$  satisfying the non-resonant condition \rea{muqn1}. Suppose that $\sigma=\tau_1\tau_2$ where $\tau_1$ is a holomorphic involution, $\rho$ is an anti-holomorphic involution, and $\tau_2=\rho\tau_1\rho$.
Then there exists a formal map $\Psi$
such that $\rho:=\Psi^{-1}\rho\Psi$ is given by \rea{57rhoz},  $\sigma^*=\Psi^{-1}\sigma\Psi$   and $\tau_{i}^*=\Psi^{-1}\tau_i\Psi$ have the form
 \begin{gather}\label{sigma0}
\sigma^*\colon\xi_j'=M_j(\xi\eta)\xi_j, \quad \eta_j'=M_j^{-1}(\xi\eta)\eta_j, \quad M_j(0)=\mu_j, \quad 1\leq j\leq p,\\
\tau_i^*=\Lambda_{ij}(\xi\eta)\eta_j, \quad \eta_j'=\Lambda_{ij}^{-1}(\xi\eta)\xi_j.
\nonumber
\end{gather}
 Here, $\xi\eta=(\xi_1\eta_1,\ldots,\xi_p\eta_p)$. Assume further that  $\log M$ $($see
 \rea{logm} for definition$)$  is tangent to the identity. 
  Under a further change of coordinates that preserves $\rho$,    $\sigma^*$  and
$\tau_i^*$ are  transformed into
\begin{gather} \label{hmjw}
\hat\sigma\colon \xi_j'=\hat M_j(\xi\eta)\xi_j, \quad \eta_j'=\hat M_j^{-1}(\xi\eta)\eta_j, \quad\hat M_j(0)=\mu_j,\quad 1\leq j\leq p,\\
\hat\tau_i=\hat\Lambda_{ij}(\xi\eta)\eta_j, \quad \eta_j'=\hat\Lambda_{ij}^{-1}(\xi\eta)\xi_j,
\quad\hat\Lambda_{2j}=\hat\Lambda_{1j}^{-1}.
\nonumber
\end{gather}
 Here the $j$th component of  $\log \hat M(\zeta)- \zeta= O(|\zeta|^2)$ 
  is independent of $\zeta_ j$.
Moreover, $\hat M$ is unique.
\end{thm}

\begin{rem}
The condition that $\log M$ is tangent to identity at the origin has to be understood as a non-degeneracy condition of   the simplest form. When there is no ambiguity, ``tangent to identity'' stands for ``tangent to identity at the origin''.
\end{rem}

We will conclude section~\ref{secnfs} with an example showing that although $\sigma,\tau_1,\tau_2$ are  linear, $\{\tau_{11},\ldots,\tau_{1p},\rho\}$ are not necessarily linearizable, provided $p>1$. 


Section \ref{div-sect} is devoted to the proof of the following divergence result. 
\begin{thm}\label{divsig} There exists a non-resonant real analytic submanifold $M$ with
pure elliptic complex tangent in $\cc^6$ 
such that if its associated $\sigma$  is transformed into a map $\sigma^*$
that commutes with the linear part of $\sigma$ at the origin,
then $\sigma^*$ must diverge.
\end{thm}

Note that the  theorem says that all normal forms of $\sigma$ (by definition, they belong to the centralizer of its linear part,
i.e. they are in the Poincar\/e-Dulac normal forms) are divergent.   It implies  that any transformation for $M$ that  transforms  $\sigma$ into a Poincar\'e-Dulac normal
form  must diverge. This is in contrast with the Moser-Webster theory:
For  $p=1$, a convergent  normal form can always be achieved even if the  associated transformation is divergent (in the case of hyperbolic complex tangent), and
furthermore in case of $p=1$ and elliptic complex tangent with a non-varnishing Bishop invariant, the normal form can   be achieved by a convergent
transformation.
A divergent Birkhoff normal form for the classical Hamiltonian systems was
obtained in \cite{Go12}. See Yin~\ci{Yi15} for the existence of divergent Birkhoff normal forms for real analytic area-preserving mappings.

 We do not know if there exists a non-resonant real analytic submanifold with pure elliptic eigenvalues in $\cc^4$ of which all Poincar\'e-Dulac normal forms are divergent.


\subsection{A unique normalization for the family $\{\tau_{ij},\rho\}$}
In section~\ref{nfin}, we will follow the normalization scheme developed for the quadric
normal forms in order to normalize $\{\tau_{11},\ldots,\tau_{1p},\rho\}$. 
Let $\hat\sigma$ be given by \re{hmjw}. We define
$$
\hat\tau_{1j}\colon\xi_j'=\hat\Lambda_{1j}(\xi\eta)\eta_j, \quad \eta_j'=\hat\Lambda_{1j}^{-1}(\xi\eta)\xi_j, \quad\xi_k'=\xi_k, \quad \eta_k'=\eta_k,\quad k\neq j,
$$
where 
 $\hat\Lambda_{1j}(0)=\la_j$  and
 $\hat M_j=\hat\Lambda_{1j}^2$.  We have the following formal normal form.
\begin{thm}\label{nfofM}Let $M$ be a real analytic submanifold that is a
 higher order perturbation of a non-resonant product
quadric. Suppose that
its associated  $\sigma$ is formally equivalent to $\hat\sigma$ given by \rea{hmjw}. 
 Suppose that  the formal mapping  $\log\hat M$ is as   in \rta{ideal0intr}. 
Then the formal normal form of $M$ is completely determined by
\eq{Phi1}\nonumber
\hat M(\zeta), \quad  \Phi(\xi,\eta).
\eeq
Here 
the  formal
mapping $\Phi$ is  in $\cL C^c(\hat\tau_{11},\ldots, \hat\tau_{1p})\cap\cL C(\hat\tau_1)$ and tangent to the identity.
 Moreover, $\Phi$ is uniquely determined up to the equivalence relation $\Phi\sim R_\e\Phi R_\e^{-1}$ with  $R_\e\colon\xi_j=\e_j\xi,\eta_j'=\e_j\eta_j$ $(1\leq j\leq p)$, $\e_j^2=1$   and $\e_{s+s_*}=\e_s$. Furthermore, if the normal form \rea{sigma0} of $\sigma$ can be achieved by a convergent transformation, so does the normal form of $M$.
\end{thm}
The   set $\cL C(\hat\tau_{1})\cap\cL C^c(\hat\tau_{11},\ldots, \hat\tau_{1p})$ is given in   \rl{cnnl} with $\mathbf B_1$ being the identity matrix. 

 We now mention related normal form problems.
The normal form problem, that is the equivalence to a model manifold, of analytic real hypersurfaces in $\cc^n$ with a non-degenerate Levi-form has a complete theory achieved through the works of E.~Cartan~\cite{Ca32}, \cite{Ca33},  Tanaka~\cite{Ta62}, and Chern-Moser~\cite{chern-moser}. 
In another direction, the  relations between formal and holomorphic equivalences of real analytic hypersurfaces (thus there is no CR singularity) have been investigated by  Baouendi-Ebenfelt-Rothschild ~\cite{BER97}, ~\cite{BER00}, Baouendi-Mir-Rothschild~\cite{BMR02},  and Juhlin-Lamel~\cite{JL13}, where  positive (i.e. convergent) results were obtained. In a recent paper, Kossovskiy and  Shafikov~\cite{KS13} showed that there are real analytic real hypersurfaces which are formally but not holomorphically equivalent. In the presence of CR singularity, the problems and techniques required are however different from those used in the CR case.   See~\cite{GS15} for further references and therein.

\subsection{Notation}
We briefly introduce  notation used in the paper. The identity map is denoted by $I$.  The matrix of a linear map $y=Ax$ is denoted by a bold-faced $\mathbf A$. We
denote by $LF$ the linear part at the origin of a mapping $F\colon\cc^m\to\cc^n$
with $F(0)=0$.
 Let $F'(0)$ or $DF(0)$ denote the Jacobian matrix of the $F$ 
 at the origin. Then
$LF(z)=F'(0)z$.
We also denote by 
$DF(z)$ or simply $DF$, the Jacobian matrix of $F$
at 
$z$, when there is no ambiguity. 
If $\mathcal F$ is a family of mappings fixing the origin, let $L\mathcal F$ denote the family of linear parts of  mappings in $\cL F$.
 By an analytic (or holomorphic) function, we shall mean a {\it germ} of analytic function at a point (which will be defined by the context) otherwise stated.
 We shall denote by ${\cL O}_n$ (resp. $\widehat {\cL O}_n$, $\mathfrak M_n$, $\widehat{\mathfrak M}_n$) the space of germs of holomorphic functions of $\cc^n$ at the origin (resp. of formal power series in $\cc^n$,  holomorphic
germs, and formal germs vanishing at the origin).

\setcounter{thm}{0}\setcounter{equation}{0}

\section{
Moser-Webster involutions and product quadrics}
\label{secinv}
  In this section we will first recall a formal and convergent result from~\ci{GS15}  that will be used to classify real submanifolds admitting the maximum number of
 deck transformations. We will then derive the family of deck transformations for the product quadrics.

 We  consider a   formal real submanifold of dimension $2p$ in $\cc^{2p}$
defined by
\eq{fmzp}
M\colon
z_{p+j} = E_{j}(z',\bar z'), \quad 1\leq j\leq p.
\eeq
Here $E_j$ are   formal power series in $z',\ov z'$.  We  assume that
\eq{fmzp+}
E_j(z',\bar z')=h_j(z',\ov z')+q_j(\ov z')+O(|(z',\ov z')|^3)
\eeq
and $h_j, q_j$ are homogeneous quadratic
polynomials. The formal complexification of $M$ is  defined by
\begin{equation}\nonumber
\label{variete-complex}
\cL M\colon
\begin{cases}
z_{p+i} = E_{i}(z',w'),\quad i=1,\ldots, p,\\
w_{p+i} = \bar E_i(w',z'),\quad i=1,\ldots, p.\\
\end{cases}
\end{equation}
 We define a {\it formal deck transformation}
of $\pi_1$ to be a formal biholomorphic map
$$
\tau\colon (z',w')\to (z',f(z',w')), \quad \tau(0)=0
$$
such that $\pi_1\tau=\pi_1$, i.e. $E\circ\tau=E$. Assume that $q^{-1}(0)=0$ and that the formal manifold defined
by \rea{fmzp}-\rea{fmzp+} satisfies condition D that its  formal branched  covering $\pi_1$ admits
$2^p$ formal deck transformations.  Then $\pi$ admits a unique set of $p$ deck transformations $\{\tau_{11}, \dots, \tau_{1p}\}$ such that each $\tau_{1j}$ fixes a hypersurface in $\cL M$.

As in the Moser-Webster theory,  the significance of the two sets of
involutions is the following proposition that transforms the normalization of the real manifolds
into that 
of two families $\{\tau_{i1},\dots, \tau_{ip}\}$ ($i=1,2$) of commuting involutions satisfying $\tau_{2j}=\rho\tau_{1j}\rho$
for an antiholomorphic involution $\rho$.
Let us recall the anti-holomorphic involution
 \eq{rho0}
 \rho_0\colon(z',w')\to(\ov{w'},\ov{z'}).
 \eeq

\begin{prop} \label{mmtp}
Let $M,\tilde M$ be formal $(${\rm resp. real analytic}$)$ real submanifolds of dimension  $2p$
in $\cc^n$ of the form \rea{fmzp}-\rea{fmzp+}.
Suppose that   $M,\tilde M$ satisfy condition ~{\rm D}. Then the following hold~$:$
\bppp
\item $M$ and $\tilde M$ are
formally $(${\rm resp.~holomorphically}$)$ equivalent
if and only if  their associated families of involutions
$\{\tau_{11}, \ldots, \tau_{1p},
 \rho_0\}$ and $\{\tilde\tau_{11}, \ldots, \tilde\tau_{1p},\rho_0\}$
are   formally $(${\rm resp.~holomorphically}$)$ equivalent.
\item
Let  $\cL T_1=\{\tau_{11}, \ldots,\tau_{1p}\}$ be a family of
 formal holomorphic $(${\rm resp.~holomorphic}$)$
commuting involutions such that
 the tangent spaces of  $\fix(\tau_{11}),\ldots,
 \fix(\tau_{1p})$  are hyperplanes
   intersecting  transversally at the origin. Let $\rho$
 be an anti-holomorphic formal $(${\rm resp.~holomorphic}$)$  involution
 and let $\cL T_2=\{\tau_{21},\ldots,\tau_{2p}\}$   with $\tau_{2j}=\rho\tau_{1j}\rho$. Let  $[\mathfrak M_n]_1^{L\cL T_i}$ be the set of linear functions without constant terms that are invariant by $L\cL T_i$. 
Suppose that
\eq{conreal}
 [\mathfrak M_n]_1^{L\cL T_1}\cap[\mathfrak M_n]_1^{L\cL T_2}
 =\{0\}.
\eeq
There exists a formal $(${\rm resp. real analytic}$)$   submanifold  defined by
\begin{equation}\label{masym}
 z''=(B_1^2,\ldots, B_p^2)(z',\ov z')
\end{equation}
for some formal $(${\rm resp. convergent}$)$  power series $B_1,\ldots, B_p$
such that $M$ satisfies condition  {\rm D}.  The set
of involutions  $\{ \tilde\tau_{11},\ldots, \tilde\tau_{1p},   \rho_0\}$  of $M$ is formally
$(${\rm resp. holomorphically}$)$  equivalent to
 $\{\tau_{11}, \ldots, \tau_{1p},\rho\}$.
 \eppp
\end{prop}
  The above proposition is proved in~\cite[Propositions 2.8 and  3.2]{GS15}. Since we need to apply the   realization several times, let us recall  how \re{masym} is constructed. Using the fact that $\tau_{11},\dots, \tau_{1p}$
are commuting involutions of which the sets of fixed points are hypersurfaces intersecting transversally, we ignore $\rho$ and linearize them simultaneously as
$$
Z_j\colon z_{p+j}\to-z_{p+i}, \quad z_i\to z_i, \quad i\neq j
$$
for $1\leq j\leq p$. Thus in $z$ coordinates, invariant functions
of $\tau_{11}, \dots, \tau_{1p}$ are generated by $z_1,\dots, z_p$ and $z_{p+1}^2,\dots, z_{2p}^2$. In the original coordinates, $z_j=A_j(\xi,\eta), 1\leq j\leq p,$ are invariant by the involutions, while $z_{p+j}=\tilde B_j(\xi,\eta)$ is skew-invariant by  $\tau_{1j}$. Then $\ov{A_j(\xi,\eta)}$ are invariant by the second family $\{\tau_{2i}\}$. Condition \re{conreal} ensures that $\var\colon (z',w') = (A(\xi,\eta),  \ov{A\circ\rho(\xi,\eta)})$
is a germ of formal  (biholomorphic) mapping at the origin. Then
$$
M\colon z_{p+j}=\tilde B^2_j\circ\var^{-1}(z',\ov z'), \quad 1\leq j\leq p
$$
is a realization for $\{\tau_{11}, \dots, \tau_{1p}, \rho\}$ in the sense stated in the above proposition.

Next we recall the deck transformations for a product quadric from \ci{GS15}.

Let us first recall involutions in \cite{MW83} where the complex tangents
are elliptic (with non-vanishing Bishop invariant) or hyperbolic.
When $\gaa_1\neq0$, the non-trivial
 deck transformations of
\eq{z2z12}
\nonumber
Q_{\gaa_1}\colon
z_2=|z_1|^2+\gaa_1(z_1^2+\ov z_1^2)
\eeq
 for $\pi_1,\pi_2$ are $\tau_1$ and $\tau_2$,  respectively. They are
\eq{tau10e}\nonumber
\tau_1\colon  z_1' = z_1, \quad w_1'=-w_1 -\gamma_1^{-1}z_1; \quad \tau_2=\rho\tau_1\rho
\eeq
with $\rho$ being defined by  (\ref{rho0}). Here the formula is valid for $\gaa_1=\infty$
(i.e. $\gaa_1^{-1}=0$).  Note that $\tau_1$ and $\tau_2$
do not commute and $\sigma=\tau_1\tau_2$ satisfies
 \eq{} \nonumber
 \sigma^{-1}=\tau_i\sigma\tau_i=\rho\sigma\rho, \quad \tau_i^2= I,\quad\rho^2= I.
 \eeq
When the complex tangent is not parabolic, the eigenvalues of $\sigma$ are   $\mu,\mu^{-1}$ with
$\mu=\la^2$ and
$
\gaa\la^2 -\la +\gaa=0.
$
For the elliptic complex tangent, we can choose a solution  $\la>1$, and in suitable coordinates we obtain
\begin{gather}
\label{tau1e}\nonumber
\tau_1\colon\xi'=\la\eta+O(|(\xi,\eta)|^2), \quad \eta'=\la^{-1}\xi+O(|(\xi,\eta)|^2),\\
\tau_2=\rho\tau_1\rho, \nonumber
\quad 
\nonumber
\rho(\xi,\eta)=(\ov\eta,\ov\xi),\\
\sigma\colon\xi'=\mu\xi+O(|(\xi,\eta)|^2), \quad \eta'=\mu^{-1}\eta+O(|(\xi,\eta)|^2),\quad \mu=\la^2.\nonumber
\end{gather}
When the complex tangent is hyperbolic, i.e.
$1/2<\gaa\leq\infty$,    $\tau_i$ and  $\sigma$ still have the above form, while $|\mu|=1=|\la|$ and
\eq{rhohy}
\nonumber
\rho(\xi,\eta)
=(\ov\xi,\ov\eta).
\eeq
 We recall from ~\cite{MW83} that
\eq{gaala1}\nonumber
\gaa_1=\f{1}{\la+\la^{-1}}.
\eeq
Note that for a parabolic Bishop surface,  the linear part of $\sigma$ is not diagonalizable.

 Consider a quadric of the complex type of CR singularity
 \eq{z3z42} Q_{\gaa_s}\colon  z_3=z_1\ov z_2+\gaa_s\ov z_2^2+(1- \gaa_s)z_1^2, \quad z_4=\ov z_3.
\eeq
Here $\gaa_s$ is a complex number.

By condition B, we know that $\gaa_s\neq0,1$.     Recall from ~\ci{GS15} that  the  deck transformations for $\pi_1$ are generated by two involutions
\ga\tau_{11}
\label{tau1112}
\nonumber
\colon \begin{cases}
z_1'=z_1,\\
z_2'=z_2,\\
w_1'=-w_1-(1-\ov\gaa_s)^{-1} z_2,\\
w_2'=w_2;
\end{cases}
\quad
\tau_{12}\colon\begin{cases}
z_1'=z_1,\\
z_2'=z_2,\\
w_1'=w_1,\\
w_2'=-w_2- \gaa_s^{-1}z_1.
\end{cases}
 \end{gather}
We still have $\rho$ defined by (\ref{rho0}).  Then $\tau_{2j}=\rho\tau_{1j}\rho$, $j=1,2$, are given by
\ga
\nonumber
\tau_{21}\colon\begin{cases}
z_1'=-z_1- (1-\gaa_s)^{-1} w_2,\\
z_2'=z_2,\\
w_1'=w_1,\\
w_2'=w_2;\end{cases}
\quad
\tau_{22}\colon  \begin{cases}
z_1'=z_1,\\
z_2'=-z_2-\ov\gaa_s^{-1} w_1,\\
w_1'=w_1,\\
w_2'=w_2.
\end{cases}
\end{gather}
Thus $\tau_i=\tau_{i1}\tau_{i2}$ is the unique deck transformation of $\pi_i$ that has the smallest dimension of the fixed-point set among all deck transformations. They are
\gan\tau_{1}\colon  \begin{cases}
z_1'=-z_1- (1-\gaa_s)^{-1} w_2,\\
z_2'=-z_2-\ov\gaa_s^{-1} w_1,\\
w_1'=w_1,\\
w_2'=w_2;
\end{cases}
\tau_{2}\colon  \label{tau12s}
\begin{cases}
z_1'=z_1,\\
z_2'=z_2,\\
w_1'=-w_1-(1-\ov\gaa_s)^{-1} z_2,\\
w_2'=-w_2-\gaa_s^{-1} z_1.
\end{cases}\end{gather*}
Also $\sigma_{s1}:=\tau_{11}\tau_{22}$ and $\sigma_{s2}:=\tau_{12}\tau_{21}$ are given by
\ga  \nonumber
\sigma_{s1}\colon \begin{cases}
z_1'=z_1,\\
z_2'=-z_2- \ov\gaa_s^{-1} w_1,\\
w_1'=(1-\ov\gaa_s)^{-1} z_2+((\ov\gaa_s-\ov\gaa_s^2)^{-1}-1)w_1,\\
w_2'=w_2;
\end{cases}\\
\sigma_{s2}\colon \begin{cases}
z_1'=-z_1-(1-\gaa_s)^{-1}w_2,\\
z_2'=z_2,\\
w_1'=w_1,\\
w_2'=\gaa_s^{-1} z_1+((\gaa_s-\gaa_s^2)^{-1}-1)w_2.
\end{cases}
\nonumber
\end{gather}
  And $\tau_1\tau_2=\sigma_{s1}\sigma_{s2}$ is given by
\gan\sigma_s\colon \begin{cases}
z_1'=-z_1- (1-\gaa_s)^{-1}w_2,\\
z_2'=-z_2- \ov\gaa_s^{-1}w_1,\\
w_1'= (1-\ov\gaa_s)^{-1} z_2+((\ov\gaa_s-\ov\gaa_s^2)^{-1}-1)w_1,\\
w_2'=\gaa_s^{-1}z_1+((\gaa_s-\gaa_s^2)^{-1}-1)w_2.
\end{cases}
\label{sigsn}
\end{gather*}
  Suppose that $\gaa_s\neq1/2$. The eigenvalues of $\sigma_s$ are
\begin{gather}\label{msms-}
\mu_s, 
\quad  \mu_s^{-1}, 
\quad
  \ov \mu_s^{-1}, \quad 
  \ov\mu_s, \\  
\mu_s= \ov\gaa_s^{-1}-1.
\label{msms-+}
\end{gather}
  Here if $\mu_s=\ov\mu_s$ and $\mu_s^{-1}=\ov\mu_s^{-1}$ then each eigenspace  has dimension $2$.
 Under suitable linear coordinates,  the involution $\rho$, defined by (\ref{rho0}),
takes the form
 \eq{rhos}
\rho(\xi_1,\xi_2,\eta_1,\eta_2)=(\ov\xi_2,\ov\xi_1,\ov\eta_2,\ov\eta_1).\eeq
  Moreover, for $j=1,2$, we have $\tau_{2j}=\rho\tau_{1j}\rho$ and
\gan
\tau_{1j}\colon\xi_j'=\la_j\eta_j, \quad \eta_j'=\la_j^{-1}\xi_j; \quad\xi_i'=\xi_i, \quad\eta_i'=\eta_i, \quad i\neq j;\label{taus}\\
\la_1=\la_s, \quad \la_2=\ov\la_s^{-1}, \quad\mu_s=\la_s^2.\label{taus+}
\end{gather*}
By a permutation of coordinates that preserves $\rho$, we obtain a unique holomorphic invariant $\mu_s$ satisfying
\eq{mu1im}
|\mu_s|\geq1, \quad \IM\mu_s\geq0, \quad 0\leq\arg\la_s\leq \pi/2, \quad \mu_s\neq -1.
\eeq
By condition E, we have $|\mu_s|\neq1$.

 Although the case
 $\gaa_s=1/2$ is not studied in this paper,  we remark that when $\gaa_s=1/2$   the only eigenvalue of $\sigma_{s1}$ is $1$.     We can choose  suitable linear coordinates such that   $\rho$
 is given by \re{rhos},  while
\begin{equation}\label{jsig}
\begin{array}{rrclrclrclrcl}
\sigma_{s1}\colon &\xi_1'&\!\!=\!\!&\xi_1, \quad &\eta_1'&\!\!=\!\!&\eta_1+\xi_1, \quad &\xi_2'&\!\!=\!\!&\xi_2,\quad &\eta_2'&\!\!=\!\!&\eta_2\\
\sigma_{s2}\colon &\xi_1'&\!\!=\!\!&\xi_1,\quad &\eta_1'&\!\!=\!\!&\eta_1,\quad  &\xi_2'&\!\!=\!\!&\xi_2,\quad&\eta_2'&\!\!=\!\!&-\xi_2+\eta_2,\\
\sigma_s\colon      &\xi_1'&\!\!=\!\!&\xi_1, \quad &\eta_1'&\!\!=\!\!&\xi_1+\eta_1, \quad &\xi_2'&\!\!=\!\!& \xi_2,\quad&\eta_2'&\!\!=\!\!&-\xi_2+\eta_2.
\end{array}
\end{equation}
Note that  eigenvalue formulae \re{msms-} and  the Jordan normal form \re{jsig} tell us that $\tau_1$ and $\tau_2$ do not commute,   while
$\sigma_{s1}$ and $ \sigma_{s2}$
commute and they are diagonalizable if and only if $\gaa_s\neq 1/2$.  We further remark that when $\mu_s$ satisfies \re{mu1im}, we have
\ga\label{regs}
\RE\gaa_s \leq1/2, \ \IM\gaa_s \geq0, \quad \text{if $|\mu_s|\geq1, \ \IM\mu_s\geq 0$};\\
\RE\gaa_s=1/2, \ \IM\gaa_s\geq0, \quad \gaa_s\neq1/2, \quad\text{if $|\mu_s|=1,\  \IM\mu_s\geq0,\  \mu_s\neq1$};  \label{regsb}\\
\gaa_s<1/2, \ \gaa_s\neq0, \quad \text{if $\mu_s^2>1$}; \qquad\gaa_s=1/2, \quad \text{if $\mu_s=1$}.\label{regsc}
\end{gather}
We have therefore proved the following.
\begin{prop}Quadratic surfaces in $\cc^4$ of complex type CR singularity at the origin is classified by \rea{z3z42} with $\gaa_s$
uniquely determined  by \rea{regs}-\rea{regsc}.
\end{prop}

  The region of eigenvalue $\mu$,  restricted to $E:=\{|\mu|\geq1, \IM\mu\geq1\}$,  can be described as follows: For a Bishop quadric, $\mu$
is precisely located  in $\om:=\{\mu\in\cc\colon |\mu|=1\}\cup [1,\infty)$.  The value of $\mu$ of a quadric of complex type, is precisely located in $\Om:=E\setminus\{-1\}$, while $\om=\pd E\setminus(-\infty,-1)$.

 In summary,  under the condition that no component is a Bishop parabolic quadric or a complex quadric with $\gaa_s=1/2$,  we have found linear coordinates for the product quadrics such that the normal forms of
$S$, $T_{ij}$, $\rho$ of the corresponding $\sigma,\sigma_j,\tau_{ij},\rho_0$ are given by
\begin{align*}
&S\colon\xi_j'=\mu_j\xi_j,\quad \eta_j'=\mu_j^{-1}\eta_j;
\\
&T_{ij}\colon\xi_j'=\la_{i j}\eta_j,\quad\eta_j'=\la_{i j}^{-1}\xi_j,\quad\xi_k'=\xi_k,\quad
\eta_k'=\eta_k, \quad k\neq j; 
\\ 
& \rho\colon \left\{\begin{array}{ll}
(\xi_e',\eta_e',\xi_h',\eta_h')=
(\ov\eta_e,\ov\xi_e,\ov\xi_h,\ov\eta_h),\vspace{.75ex}
\\
(\xi_{s}', \xi_{s+s_*}',\eta_{s}',\eta_{s+s_*}')=(\ov\xi_{s+s_*},
\ov\xi_{s},\ov\eta_{s+s_*}, \ov\eta_{s}).
\end{array}\right.
\end{align*}
Notice that we can always normalize $\rho_0$ into the above normal form $\rho$.

%
%

For various reversible mappings and their relations with general mappings, the reader is referred  to \cite{OZ11} for recent results and references therein.

To derive our normal forms, we shall transform  $\{\tau_1,\tau_2,\rho\}$ into
a
normal form first.
We will  further normalize $\{\tau_{1j},\rho\}$ by using the group of biholomorphic maps
that preserve the normal form of $\{\tau_1,\tau_2,\rho\}$, i.e. the centralizer of the normal form  of $\{\tau_1,\tau_2,\rho\}$.

\setcounter{thm}{0}\setcounter{equation}{0}

\section{
Quadrics with the maximum number of deck transformations}
\label{secquad}
 In~ \rp{mmtp},
we describe   the basic relation between the classification of real manifolds and that of  two families of involutions intertwined
by an antiholomorphic involution, which is established in~\ci{GS15}.
As an application, we obtain  in this section a normal form for  two families of linear involutions
and use it to construct the normal form for their associated quadrics. This section
also serves an introduction to our approach to find the normal forms of the real submanifolds   at least at
the formal level.   At the end of the section, we will  also introduce examples of quadrics of which $S$ is given by Jordan matrices. The perturbation of such quadrics will not be studied in this paper.

\subsection{Normal form of two families of linear involutions}

To formulate our results, we first discuss the normal forms  which we are seeking for the involutions.
We are given two families of commuting linear involutions
$ \cL T_1
=\{T_{11}, \ldots, T_{1p}\}$ and $\cL T_2=\{T_{21},\ldots, T_{2p}\}$ with $T_{2j}=\rho T_{1j}\rho$. Here $\rho$ is a linear anti-holomorphic involution. We
 set
$$  T_1=T_{11}\cdots T_{1p}, \quad T_2=\rho T_1\rho. $$
 We also assume that each $\fix T_{1j}$ is a hyperplane and
$\cap \fix T_{1j}$  has dimension $p$.  By \cite[Lemma 2.4]{GS15},  in suitable linear coordinates,  each $T_{1j}$ has the form
\eq{zjxp}\nonumber
Z_j\colon \xi'=\xi, \quad \eta_i'=\eta_i \ (i\neq j), \quad \eta_j'=-\eta_j.
\eeq
Thus by \re{conreal},
\ga\label{ilii++}
  \dim [\mathfrak M_n]_1^{\cL T_i}=p,\quad  [\mathfrak M_n]_1^{\cL T_i}= [\mathfrak M_n]_1^{T_i},
 \\
  \dim [\mathfrak M_n]_1^{T_i}=p,\quad
 [\mathfrak M_n]_1^{ T_1}\cap[\mathfrak M_n]_1^{T_2}
 =\{0\}.\label{bilii++}
 \end{gather}
  Recall that  $[\mathfrak M_n]_1$ denotes the linear holomorphic functions   without constant terms.   We would like to find
a change of coordinates $\var$ such that $\var^{-1} T_{1j}\var$ 
and
$\var^{-1}\rho\var$ have a simpler form. We would like
to show that two such families of involutions $\{\cL T_1,\rho\}$ and
$\{\widetilde{\cL T_1},\tilde\rho\}$ are holomorphically equivalent, if there are
normal forms are equivalent under a much smaller set of changes of coordinates,
or if they are identical in the ideal situation.

Next, we describe our scheme to derive the normal forms for linear involutions.
The scheme  to derive the linear normal forms
turns out to be essential to derive normal forms for non-linear involutions and the perturbed quadrics. We define
$$
 S=T_1T_2.
$$
Besides   conditions   \re{ilii++}-\re{bilii++}, we will soon impose condition E   that $S$ has $2p$ distinct eigenvalues.

We  first use a linear map $\psi$ to
 diagonalize $S$ to its normal form
$$
\hat S\colon \xi_j'=\mu_j\xi, \quad \eta_j'=\mu_j^{-1}\eta_j,\quad 1\leq j\leq p.
$$
The choice of $\psi$ is not unique.
We   further normalize $T_1,T_2, \rho$ under linear transformations commuting with $\hat S$, i.e.
the   invertible mappings in the
{\it linear centralizer} of $\hat S$. We   use a linear map that commutes with $\hat S$ to
 transform $\rho$ into a
normal form too, which is still
 denoted by $\rho$.
 We then use a transformation $\psi_0$ in the linear centralizer of $\hat S$ and $\rho$
  to normalize the   $T_1,T_2$ into the normal form
$$
\hat T_i\colon \xi_j'=\lambda_{ij}\eta_j, \quad \eta_j'=\lambda_{ij}^{-1}\xi_j, \quad 1\leq j\leq p.
$$
Here we require $\lambda_{2j}=\lambda_{1j}^{-1}$. Thus $\mu_j=\lambda_{1j}^2$ for $1\leq j\leq p$, and
 $\lambda_{11}, \ldots, \lambda_{1p}$ form a complete set of invariants
of $T_1,T_2,\rho$, provided the normalization satisfies
$$
\la_{1e}>1, \quad \IM\la_{1h}>0, \quad \arg\la_{1s}\in(0,\pi/2), \quad |\la_s|>1.
$$
  This normalization will be verified  under condition E.

 Next we   normalize the family  $\cL T_1$ of linear involutions
 under mappings in the linear centralizer of $\hat T_1, \rho$. Let us assume that   $T_1,\rho$ 
 are in the normal forms $\hat T_1,\rho$.
To further normalize the family $\{\cL T_1,\rho\}$,   we use the crucial property that $T_{11},\ldots, T_{1p}$
commute pairwise and each $T_{1j}$  fixes a hyperplane.
 This allows us to express the family of involutions via a single linear mapping $\phi_1$:
$$
T_{1j}=\var_1\phi_1Z_j\phi_1^{-1}\var_1^{-1}.
$$
Here the linear mapping $\var_1$ depends only on $\la_1,\ldots, \la_p$.
Expressing  $\phi_1$
 in  a  non-singular $p\times p$  constant
matrix $\mathbf B$, the normal form for $\{T_{11}, \ldots, T_{1p},\rho\}$ consists of invariants $\la_1,\ldots,
\la_p$ and a normal form  of $\mathbf B$.  After we obtain the normal form for  $\mathbf B$,
we will construct the normal form of the quadrics by using the realization procedure  in the  \rp{mmtp} (see the proof in ~\ci{GS15})

\medskip

We now carry out the details.

Let  $T_1=T_{11}\cdots T_{1p}$,
   $T_2=\rho T_1\rho$ and
$
S=T_1T_2.
$
   Since $T_i$ and $\rho$ are involutions, then $S$ is reversible with respect to $T_i$
   and $\rho$, i.e. $$S^{-1}=T_i^{-1}ST_i, \quad S^{-1}=\rho^{-1} S\rho, \quad T_i^2= I, \quad\rho^2= I.
   $$
   Therefore,
  if $\kappa$ is an eigenvalue of $S$ with a  (non-zero)   eigenvector $u$, then
\eq{}
Su=\kappa u, \quad S(  T_iu)=\kappa^{-1}T_iu, \quad S(\rho u)=\ov\kappa^{-1}\rho u,\quad
S(\rho T_iu )=\ov\kappa\rho T_iu. \nonumber
\eeq
  Following \cite{MW83} and  [St07], we will divide eigenvalues of   product quadrics that satisfy condition $E$ into $3$ types: $\mu$ is {\it elliptic} if $\mu\neq\pm1$
and $\mu$ is real, $\mu$ is {\it hyperbolic} if $|\mu|=1$ and $\mu\neq1$,
and $\mu$ is {\it complex} otherwise.   The classification of $\sigma$ into the types
corresponds to the classification of the types of complex tangents described in section~\ref{secinv}; namely,  an elliptic (resp. hyperbolic)  complex tangent is tied to
 a hyperbolic (resp. elliptic) mapping $\sigma$.
%
%
%
%
%

We first characterize the linear family $\{T_1,T_2,\rho\}$ that can be realized by a product quadric with $S$ being diagonal.
\begin{lemma}\label{incomp}
Let $\{T_1,T_2\}$ be a pair  of linear involutions on $\cc^{2p}$ satisfying  \rea{bilii++}. Suppose that $T_2=\rho T_1\rho$ for a linear anti-holomorphic involution and $S=T_1T_2$ is diagonalizable. Then $\{T_1,T_1,\rho\}$ is realized by the product of quadrics of type elliptic, hyperbolic, or complex.   In particular, if $S$  has $2p$ distinct eigenvalues, then $1$ and $-1$ are not eigenvalues of $S$.
\end{lemma}
\begin{proof}
The last assertion follows from the first part of the lemma immediately. Thus the following argument does not  assume that $S$ has distinct eigenvalues.  Let $E_i(\nu_i)$ with $i=1,\dots, 2p$ be  eigenspaces of $S=T_1T_2$ with eigenvalues $\nu_i$.  Thus
$$
\cc^{2p}=\bigoplus_{i=1}^{2p} E_i(\nu_i), \quad \cc^{2p}{\ominus}E_i(\nu_i):=\bigoplus_{j\neq i}E_j(\nu_j).
$$
Fix an $i$ and denote the corresponding space by $E(\nu)$. Since $\sigma^{-1}=T_1\sigma T_1$, then $T_1E(\nu)=T_2E(\nu)$, which is equal to   some invariant space $E(\nu^{-1})$.
Take an eigenvector $e\in E(\nu)$ and set $e'=T_1e$.

Let us first show that $1$ is not an eigenvalue. Assume for the sake of contradiction that $E(1)$  is spanned by a (non-zero) eigenvector $e$.
Then $T_1$ preserves $E(1)$. Otherwise, $e'$ and $e$ are independent.  Now $T_2e=T_1e =e '$ and  $
T_i(e +e ')=e '+e ,
$
which contradicts $\fix T_1\cap\fix T_2=\{0\}$. With $E(1)$ being preserved by $T_i$, we have $T_ie=\e e$ and $\e=\pm1$, since $T_i$ are involutions.  We have $\e\neq1$ since $\fix T_1\cap\fix T_2=\{0\}$. Thus $T_1e=-e=T_2e$.
Then $\fix T_1$ and $\fix T_2$ are subspaces of $ \cc^{2p}\ominus E(1)$ and both are of dimension $p$. Hence $\fix T_1\cap\fix T_2\neq\{0\}$, a contradiction.

Since $S^{-1}=\rho^{-1}S\rho$ and $S^{-1}=T_i^{-1}ST_i$ then $T_1$ sends $E(\nu)$ to some $E(\nu^{-1})$ as mentioned earlier,  while $\rho$ sends $E(\nu)$ to some $E(\ov\nu^{-1})$. In such a way, each of $T_i,\rho$ yields an involution on the set $\{E(\nu_1), \dots, E(\nu_{2p})\}$.

Let $E_1(-1), \dots, E_k(-1)$ be all spaces invariant by $T_1$.
Since $T_2=T_1S$,  they are also invariant by $T_2$.
Then none of the $k$ spaces is invariant by $\rho$. Indeed, if one of   them, say $E_j$ generated by $e_j$, is invariant by $\rho$, we have $T_1e_j=\e e_j$ and $\rho e_j=be_j$ with $\e^2=1=|b|$. We get $T_2e_j=(\rho T_1\rho)e_j=\e e_j$ and $\sigma e_j=e_j$, which contracts that $\sigma$ has eigenvalue $-1$ on $E_j$.  Furthermore, if $E(-1)$ is invariant by $T_1$, then $\rho E(-1)$ is also invariant by  $T_1$ 
 as $T_1=\rho T_2\rho$.
Thus we may assume that   $\rho E_j=E_{\ell+j}$ for $1\leq j\leq \ell:=k/2$.  For each $j$ with $1\leq j\leq \ell$,  either $T_1=I=-T_2$ on $E_j$ and $T_1=\rho T_2\rho=-I$ on $E_{\ell+j}$,
or $T_1=-I$ on $E_j$ and $T_1=I$ on $E_{j+\ell}$. Interchanging $E_j, E_{\ell+j}$ if necessary, we may assume that $T_1=I=-T_2$ on $E_j$ and $T_1=-I=-T_2$ on $E_{\ell+j}$.  We can restrict the involutions $T_1,T_2,\rho$ on $\cc^2:=E_j\oplus E_{\ell+j}$ as it is invariant by the three involutions. By the realization in~\cite{MW83},  $\{T_1,T_2,\rho\}$ is realized by a Bishop quadric; in fact, it is   $Q_\infty$.
Assume now that $E(-1)$ is not invariant by $T_1$. Thus
$T_i$ sends $E(-1)$ into a different $\tilde E(-1). $
Assume first that $E(-1)$ is  invariant by $\rho$.  Then $\tilde E(-1)$ is also invariant by $\rho$ as $\rho=T_2\rho T_1$.
 Thus as the previous case $\{T_1,T_2,\rho\}$, restricted to $E(-1)\oplus \tilde E(-1)$ is realized by $Q_\infty$.

 Suppose now that  $\rho$ does not preserve $E(-1)$. Recall that we already assume that $T_1(E(-1))=\tilde E(-1)$ is different from $E(-1)$.   Let us  show that $\tilde E(-1)\neq \rho E(-1)$. Otherwise, we let $\tilde e=\rho e$ with
$e$ being an eigenvector in $E(-1)$. Then $T_1e=a  \tilde e$. So $T_2e=\rho T_1\rho e=\ov a^{-1}\tilde e$ and $T_1T_2e=|a|^{-2}e$.  This contracts   $Se=-e$.  We now realize
$E(-1)\oplus \rho E(-1)\oplus \tilde E(-1)\oplus \rho\tilde E(-1)$ by a product of two copies of $Q_\infty$ as follows. Take a non-zero  vector $e\in E(-1)$. Define $e_1=e+T_1e$. So $T_1e_1=e_1$, $Se_1=-e_1$, and $T_2e_1=T_1Se_1=-e_1$.
Define $\tilde e_1=\rho e_1$; then $T_1\tilde e_1=\rho T_2\rho \tilde e_1=-\tilde e_1$. Define  $ \tilde e_2=e_1-T_1e_1$; then $T_1\tilde e_2=-\tilde e_2$ and $T_2\tilde e_2=\tilde e_2$.  Define $e_2=\rho \tilde e_2$; then $T_1e_2=\rho T_2\rho e_2=e_2$.
  In coordinates
$
z_1e_1+w_1\tilde e_1+z_2e_2+w_2\tilde e_2,
$
we have $T_1(z_j)=z_j$ and $T_1(w_j)=-w_j$ and $\rho (z_j)=\ov{w_j}$.  
Therefore, $\{T_1,T_2,\rho\}$ is realized by the product of two copies of $Q_\infty$.

Consider now the case $\nu$ is positive and $\nu\neq1$.  We have
\eq{tiem}
T_i\colon E(\nu)\to E(\nu^{-1}), \quad i=1,2.
\eeq
There are two cases: $\rho E(\nu)=E(\nu^{-1})$ or $\rho E(\nu):=\tilde E(\nu^{-1})\neq E(\nu^{-1})$. For the first case,
the family $\{T_1,T_2,\rho\}$, restricted to $E(\nu)\oplus E(\nu^{-1})$, is realized by an elliptic Bishop quadric $Q_\gaa$
with $\gaa\neq0$. For the second case, we want to verify that $\{T_1,T_2,\rho\}$, restricted to $E(\nu)\oplus \rho E(\nu)\oplus E(\nu^{-1})\oplus \rho E(\nu^{-1})$,  is  realized by a quadric of complex type singularity. Write $\nu_1:=\nu=\la_1^2$ with $\la_1>0$, $\la_2:=\la_1^{-1}$, and $\nu_2:=\nu_1^{-1}$.  Let $u_1$
be an eigenvector in $E(\nu)$. Define $v_1=\la_1T_1u_1\in E(\nu^{-1})$. Then $T_ju_1=\la_j^{-1}v_1$. Define $u_2=\rho u_1$ and $v_2=\rho v_1$. Then $T_1u_2=\rho T_2\rho u_2=\rho T_2u_1=\la_2^{-1}v_2$.   Thus $\sigma u_j=\nu_j u_j$ and $\sigma v_j=\nu_j ^{-1}v_j$. We now realize the family of involutions by a quadratic submanifold. For the convenience of the reader, we repeat   part of argument in~\cite{GS15}; see the paragraph after \rp{mmtp}.
In coordinates
$
\xi_1u_1+\xi_2u_2+\eta_1v_1+\eta_2v_2,
$
we have $T_i(\xi,\eta)=(\la_i\eta,\la_i^{-1}\xi)$ and $\rho(\xi,\eta)=(\ov\xi_2,\ov\xi_1,\ov\eta_2,\ov\eta_1)$. Let
\gan
z_j=\xi_j+\la_j\eta_j,  \quad w_j=\ov{z_j\circ\rho},\quad j=1,2;\\
z_3=(\eta_1-\la_1^{-1}\xi_1)^2, \quad z_4=(\eta_2-\la_2^{-1}\xi_2)^2.
\end{gather*}
Expressing $\xi_j,\eta_j$ via $(z_1,z_2,w_1,w_2)$, we obtain
$$
z_3=L_1^2(z_1,z_2,w_1,w_2), \quad z_4=L_2^2(z_1,z_2,w_1,w_2).
$$
Setting $w_1=\ov z_1$ and $w_2=\ov z_2$, we obtain the defining equations of $M\subset \cc^4$ that is a realization of $\{T_1,T_2,\rho\}$.

Assume now that $\nu<0$ and $\nu\neq-1$. We still have \re{tiem}.  We want to show that  $\rho(E(\nu))\neq E(\nu^{-1})$ where $E(\nu^{-1})$ is in \re{tiem}, i.e. the above second case in $\nu>0$ occurs and the above argument shows that $\{T_1,T_2,\rho\}$, restricted to $E(\nu)\oplus \rho E(\nu)\oplus E(\nu^{-1})\oplus \rho E(\nu^{-1})$,  is  realized by a quadric of complex type singularity. Suppose that $\rho E(\nu)=E(\nu^{-1})$.
Take $e\in E(\nu)$. We can write $\tilde e=\rho e\in E(\nu^{-1})$. Then $T_1e=a\tilde e$. We have $T_2e=T_1Se=\nu a\tilde e$ and $T_2e=\rho T_1\rho e=\rho(a^{-1}e)=\ov a^{-1}\tilde e$. We obtain $\nu=|a|^{-2}>0$, a contradiction.

Analogously, if $\nu$ has modulus $1$ and is different from $\pm1$, we have two cases: $\rho E(\nu)= E(\nu^{-1})$   or $\rho E(\nu):=\tilde E(\nu^{-1})\neq E(\nu^{-1})$. In the first case, $\{T_1,T_2,\rho\}$ restricted to the two dimensional subspace is realized by a hyperbolic quadric $Q_\gaa$ with $\gaa\neq\infty$. In the second case its restriction to the $4$-dimensional subspace is realized by a quadric of complex CR singularity with $|\nu|=1$. In fact the same argument is valid.  Namely,  let $\la_1^2=\nu=\nu_1$. Let $\la_2=\la_1^{-1}$ and $\nu_2=\nu_1^{-1}$. Take an eigenvector $e_1\in E(\nu)$. Define $\tilde e_1=\la_1 T_1e_1, e_2=\rho e_1$ and $\tilde e_2=\rho\tilde e_1$. Then define $z_j,w_j$ and $L_j$ as above, which gives us a realization. We leave the details to the reader.
Finally, if $\nu,\ov\nu^{-1},\nu^{-1},\ov\nu$ are distinct, then we have a realization proved in \rt{quadclass} for a general case where all eigenvalues are distinct. \end{proof}


  Of course, there are non-product quadrics that realize $\{T_1,T_2,\rho\}$ in \rl{incomp} and the main purpose of this section is to classify  them under condition E.
We now assume conditions E and     \re{ilii++}-\re{bilii++} for the rest of the section to derive a normal form for $T_{1j}$ and $\rho$.

We need to choose the eigenvectors of $S$  and their
eigenvalues in such a way that $T_1,T_2$ and $\rho$ are in a normal form. We will first
choose eigenvectors to put
$\rho$ into a normal form. After normalizing $\rho$, we will then choose eigenvectors to normalize $T_1$
and $T_2$.

First, let us consider  an elliptic eigenvalue $\mu_e$. Let $u$ be an eigenvector of $\mu_e$. Then $u$ and $v=\rho (u)$ satisfy
\begin{equation}\label{svmN}
S(v)=\mu_e^{-1}v, \quad T_j(u)=\la_j^{-1}v, \quad \mu_e=\la_1\la_2^{-1}.
\end{equation}
Now $T_2(u)=\rho T_1\rho(u)$ implies that
$$\lambda_2=\ov\la_1^{-1},\quad \mu_e=|\la_1|^2.
$$
 Replacing $(u,v)$ by $(cu,\ov cv)$, we may assume that $\la_1>0$
 and $\la_2=\la_1^{-1}$. Replacing $(u,v)$ by $(v,u)$ if necessary,
  we may further  achieve
 $$
 \rho(u)=v, \quad \la_1=\la_e>1, \quad\mu_e=\la_e^2>  1.
 $$
We still have the freedom to replace $(u,v)$ by $(ru,rv)$ for $r\in\rr^*$, while
preserving the above conditions.

Next,  let $\mu_h$ be a
 hyperbolic eigenvalue of $S$ and $S(u)=\mu_h u$. Then $u$ and
$v=T_1(u)$ satisfy
$$
\rho (u)=au, \quad \rho (v)=bv, \quad |a|=|b|=1.
$$
Replacing $(u,v)$ by $(cu,v)$, we may assume that $a=1$.
Now $T_2(v)=\rho T_1\rho(v)=\ov bu$.   To obtain $b=1$, we replace $(u,v)$ by $(u, b^{-1/2}v)$.
This give us \re{svmN} with $|\la_j|=1$. Replacing $(u,v)$ by $(v,u)$ if necessary,
we may further achieve
\eq{}
\nonumber
\rho (u)=u, \quad \rho( v)=v, \quad \la_1=\la_h, \quad \mu_h=\la_h^2, \quad  \arg\la_h\in(0,\pi/2).
\eeq
Again, we  have the freedom to replace $(u,v)$ by $(ru,rv)$ for $r\in\rr^*$, while
preserving the above conditions.

 Finally, we consider a complex eigenvalue $\mu_s$. Let $S(u)=\mu_s u$. Then $\tilde u=\rho (u)$
satisfies $S(\tilde u)=\ov\mu_s^{-1}\tilde u$.  Let $u^*=T_1(u)$
and $\tilde u^*=\rho(u^*)$. Then $S(u^*)=\mu_s^{-1}u^*$ and $S(\tilde u^*)=\ov\mu_s \tilde u^*$. We change eigenvectors by
$$(u,\tilde u,u^*,\tilde u^*)\to (u,\tilde u, cu^*,\ov c\tilde u^*)$$
so that
\gan
\rho(u)=\tilde u, \quad\rho(u^*)=\tilde u^*,
\\ T_j(u)=\la_j^{-1}u^*,\quad
T_j(\tilde u)=\ov\la_j\tilde u^*,\quad \la_2=\la_1^{-1}.
\end{gather*}
Note that $S(u)=\la_1^2 u$, $S(u^*)=\la_1^{-2}u^*$, $S(\tilde u)=\ov\la_1^{-2}\tilde u$, and $S(\tilde u^*)=\ov\la_1^2\tilde u^*$.
Replacing $(u,\tilde u,u^*,\tilde u^*)$ by $(u^*,\tilde u^*,u,\tilde u)$ changes the argument and the modulus of $\la_1$ as  $\la_1^{-1}$
becomes $\la_1$. Replacing them by $(\tilde u,u,\tilde u^*,u^*)$ changes only the modulus as
 $\la_1$ becomes $ \bar \la_1^{-1}$ and then replacing them by $(u^*,\tilde u^*, -u,-\tilde u)$ changes the sign of $\la_1$. Therefore, we
may achieve
\eq{las}\nonumber
\mu_s=\la_s^2, \quad\la_1=\la_s, \quad  \arg\gaa_s\in(0,\pi/2), \quad |\la_s|>1.
\eeq
 We still have the freedom to replace $(u,u^*,\tilde u,\tilde u^*)$ by $(cu,cu^*,\ov c\tilde u,\ov c\tilde u^*)$.

We  summarize  the above choice of eigenvectors and their corresponding coordinates.
First, $S$
has distinct eigenvalues
$$
\lambda_e^2=\ov\la_e^2, \quad\lambda_e^{-2};\qquad
\lambda_h^2,\quad\ov\lambda_h^2=\lambda_h^{-2}; \qquad \lambda_s^2,\quad \lambda_s^{-2},
\quad\overline\lambda_s^{-2},\quad\overline\lambda_s^2.
$$
Also, $S$ has linearly independent  eigenvectors satisfying
\gan
Su_e=\lambda_e^2u_e, \quad Su_e^*=\lambda_e^{-2}u^*_e,\\
 Sv_h=\la_h^2v_h,\quad Sv_h^*=\la_h^{-2}v_h^*,\\
  Sw_s=\lambda_s^2w_s,
\quad   Sw_s^*=\lambda_s^{-2}w_s^*,
\quad S\tilde w_s=\ov\lambda_s^{-2}\tilde w_s,\quad S\tilde w_s^*=\ov\lambda_s^2 \tilde w_s^*.
\end{gather*}
Furthermore,  the $\rho$, $T_1$, and the chosen eigenvectors of $S$ satisfy
\begin{gather*}
 \rho u_e= u^*_e,\quad  T_1u_e=\lambda_e^{-1}u_e^*;\\
 \rho v_h= v_h,\quad \rho v_h^*=v_h^*,
 \quad T_1v_h=\lambda_h^{-1}v_h^*; \\
 \rho w_s= \tilde w_s,\quad \rho w_s^*=\tilde w_s^*,
 \quad T_1w_s=\lambda_s^{-1}  w_s^*,\quad T_1\tilde w_s=\ov\lambda_s\tilde w_s^*.
 \end{gather*}

 For normalization, we collect elliptic eigenvalues
$\mu_e$ and $ \mu_e^{-1}$,    hyperbolic eigenvalues
$\mu_h$ and $\mu_h^{-1}$, and   complex eigenvalues
in $\mu_s,\mu_s^{-1}, \ov\mu_s^{-1}$ and $\ov\mu_{s}$.
We put them in the order
\gan
\mu_e=\ov\mu_e,\quad \mu_{p+e}=\mu_e^{-1},  \\
\mu_h,\quad\mu_{p+h_*+h}=\ov\mu_h,   \\
 \mu_s,\quad\mu_{s+s_*}=
\overline\mu_s^{-1},\quad \mu_{p+s}= \mu_s^{-1},
\quad\mu_{p+s_*+s}=\overline\mu_s .
\end{gather*}
Here and throughout the paper the ranges of subscripts $e,h, s$ are restricted to
$$
1\leq e\leq e_*, \quad e_*< h\leq e_*+h_*, \quad e_*+h_*<s\leq p-s_*.
$$
 Thus
$
e_*+h_*+2s_*=p.
$
Using the new coordinates
$$
\sum(\xi_eu_e+\eta_eu_e^*)+\sum(\xi_hv_h+\eta_hv_h^*)+
\sum (\xi_{s} w_s+\xi_{s+s_*}\tilde w_s+\eta_{s} w^*_s+ \eta_{s+s_*}\tilde w^*_s),
$$
we have normalized $\sigma, T_1, T_2$ and $\rho$.  In summary, we have the following normal form.
\begin{lemma}\label{t1t2sigrho}Let $T_1,T_2$ be linear holomorphic involutions on $\cc^{n}$ that satisfy \rea{bilii++}.
Then $n=2p$ and $\dim [\mathfrak M_n]_1^{T_i}=p$. Suppose that $T_2= \rho_0 T_1\rho_0$ for some anti-holomorphic linear involution $\rho_0$.
Assume that $S=T_1T_2$ has $n$ distinct eigenvalues.
There exists a linear change of holomorphic coordinates that transforms
 $T_1,T_2, S,\rho_0$  simultaneously   into the normal forms $\hat T_1,\hat T_2,\hat S,\rho:$
\ga
\label{hatt1}
\hat T_1\colon \xi_j'=\la_j\eta_j,\quad \eta_j'=\la_j^{-1}\xi, \quad 1\leq j\leq p;
\\
\hat T_2\colon \xi_j'=\la_j^{-1}\eta_j,\quad \eta_j'=\la_j\xi_j, \quad 1\leq j\leq p;\\
\label{hsxie}\hat S\colon\xi_j'=\mu_j\xi_j, \quad\eta_j'=\mu_j^{-1}\eta_j, \quad 1\leq j\leq p;\\
\label{eqrh}
\rho\colon \left\{\begin{array}{lr}
\xi_e'=\ov\eta_e, &  \qquad \eta_e'= \ov\xi_e,\vspace{.75ex}\\
\xi_h'=\ov\xi_h, & \qquad \eta_h'=\ov\eta_h,\\
\xi_{s}'=\ov\xi_{s+s_*}, &\xi_{s+s_*}'=\ov\xi_{s},\\
\eta_{s}'=\ov\eta_{s+s_*}, &\eta_{s+s_*}'=\ov\eta_{s}.
\end{array}\right.
\end{gather}
Moreover, the eigenvalues $\mu_1,\ldots, \mu_p$ satisfy
\ga\label{mula}
\mu_j=\la_j^2, \quad 1\leq j\leq  p;\\
\label{muor}
\la_e>1,\quad |\la_h|=1,\quad |\la_s|>1,\quad\la_{s+s_*}=\ov\la_{s}^{-1};\\
\arg\la_{h}\in(0,\pi/2), \quad\arg\la_s\in(0,\pi/2);\\
\label{laee} \la_{e'}<\la_{e'+1}, \quad 0<\arg\la_{h'}<\arg\la_{h'+1}<\pi/2;\\
 \label{lass}\text{$\arg\la_{s'}<\arg\la_{s'+1}$, or $\arg\la_{s'}=\arg\la_{s'+1}$
 and $|\la_{s'}|<|\la_{s'+1}|$.}
\end{gather}
Here $1\leq e'<e_*$, $e_*<h'<e_*+h_*$, and $e_*+h_*<s'<p-s_*$. And $1\leq e\leq e_*$, $e_*<h\leq e_*+h_*$, and $e_*+h_*<s\leq p-s_*$.
 If $\tilde S$
is also in the normal form \rea{hsxie} for possible different eigenvalues $\tilde \mu_1,\dots, \tilde \mu_p$
satisfying \rea{mula}-\rea{lass},  then $S$ and $\tilde S$
are equivalent if and only if their eigenvalues are identical.
\end{lemma}
The above normal form of $\rho$
 will be fixed
for the rest of paper.
 Note that in case of non-linear involutions
$\{\tau_{11},\ldots,\tau_{1p},\rho\}$ of which the
linear part are given by $\{T_{11},\dots,T_{1p},\rho\}$
we can always linearize $\rho$ first
under a holomorphic  map of which the
linear part at the origin is described  in above
normalization
for the linear part of $\{\tau_{11},\ldots, \tau_{1p},\rho\}$.
Indeed, we may assume that the linear part of the latter family is already
in the normal form. Then $\psi=\f{1}{2}(I+(L\rho)\circ\rho)$ is tangent
to the identity and $(L\rho)\circ\psi \circ\rho=\psi$, i.e. $\psi$
transforms $\rho$ into $L\rho$ while preserving the linear parts
of $\tau_{11},\dots, \tau_{1p}$.
  Therefore in the non-linear case, we can
assume that $\rho$ is given by the above normal form.
 The above lemma   tells us  the ranges
of eigenvalues $\mu_e,\mu_h$ and $\mu_s$ that can be realized by  quadrics that satisfy conditions E and     \re{ilii++}-\re{bilii++}.

Having normalized $T_1$ and $\rho$, we want to  further
 normalize $\{T_{11},\ldots, T_{1p}\}$ under linear maps that preserve the normal forms   of
$   \hat T_1$ and $\rho$.
 We know that the composition of $T_{1j}$ is in the normal form, i.e.
\eq{t11h}
T_{11}\cdots T_{1p}= \hat T_1
\eeq
   is given in \rl{t1t2sigrho}.
We  first   find an expression for  all   $T_{1j}$ that  commute pairwise and satisfy \re{t11h}, by using
 invariant and skew-invariant functions of $\hat T_1$.   Let
\eq{var1z}\nonumber
(\xi,\eta)=\var_1(z^+,z^-)\eeq
 be defined  by
\begin{alignat}{3}\label{zep1}
&z_e^+ &&=\xi_e+\la_e\eta_e, \quad &&
z_e^-=\eta_e-\la^{-1}_e\xi_e, \\
&z_h^+ &&=\xi_h+\la_h\eta_h, \quad
&&z_h^-=\eta_h-\ov\la_h\xi_h, \\
&z_{s}^+&&=\xi_{s}+\la_s\eta_{s}, \quad
&&z_{s}^-=\eta_{s}-\la^{-1}_s\xi_{s}, \\
 &z_{s+s_*}^+&&=\xi_{s+s_*}+\ov\la^{-1}_s\eta_{s+s_*}, \quad
 &&z_{s+s_*}^-=\eta_{s+s_*}-\ov\la_s\xi_{s+s_*}.
 \label{zep4}
\end{alignat}
 In $(z^+,z^-)$ coordinates, $\var_1^{-1}\hat T_1\var_1$ becomes
$$
Z\colon z^+\to z^+, \quad z^-\to -z^-.
$$
We decompose $Z=Z_{1}\cdots Z_{p}$ by using
$$
Z_{j}\colon (z^+,z^-)\to (z^+, z_1^-,\ldots, z^-_{j-1},-z_j^-,z_{j+1}^-, \ldots,
z_{p}^-).
$$

To keep simple notation, let us use the same notions $x,y$ for a linear transformation $y=A(x)$
and its matrix representation:
$$
A\colon x\to \mathbf Ax.
$$
The following lemma, which can be verified immediately, shows the advantages of coordinates $z^+, z^-$.
\begin{lemma}\label{zcent}The   linear centralizer of $Z$ is the set of  mappings of the form
\eq{vaph}
\phi\colon (z^+,z^-)\to (\mathbf{A}z^+,\mathbf{B}z^-),
\eeq
where $\mathbf{A},\mathbf{B}$ are constant and possibly singular 
matrices.
 Let $\nu$ be a permutation of $\{1,\ldots, p\}$.
Then $Z_j\phi=\phi Z_{\nu(j)}$ for all $j$ if and only  if $\phi$ has the above form with
$\mathbf{B}=\diag_\nu \mathbf{d}$. Here
\eq{dfndiag}
\diag_\nu(d_1,\ldots, d_p):= (b_{ij})_{p\times p}, \quad b_{j\nu(j)}=d_j, \quad b_{jk}=0 \ \text{if $k\neq\nu(j)$}.
\eeq
In particular, the linear centralizer of
$\{Z_{1},\ldots, Z_{p}\}$ is   the set of  mappings \rea{vaph} in  which  $\mathbf{B}$ are diagonal.
 \end{lemma}

 To continue our normalization for the family $\{T_{1j}\}$, we note that
 $\var_1^{-1}T_{11}\var_1,\ldots$, $\var_1^{-1}T_{1p}\var_1$  generate an abelian group of $2^p$  involutions
 and each of these $p$ generators fixes
a hyperplane.   By \cite[Lemma 2.4]{GS15},
 there is a linear transformation $\phi_1$ such that
\eq{phi1-}\nonumber
\var_1^{-1}T_{1j}\var_1= \phi_1Z_{j}\phi_1^{-1}, \quad 1\leq j\leq p.
\eeq
 Computing two compositions on both sides, we see that   $\phi_1$  must be in the linear centralizer of $Z$. Thus, it is in the form \re{vaph}.
 Of course,
$\phi_1$ is not unique;
$\tilde\phi_1$ is another such linear map for the same $T_{1j}$ if and only if  $\tilde\phi_1=\phi_1\psi_1$
with $\psi_1\in\cL C(Z_{1},\ldots, Z_{p})$.
By \re{vaph}, we may restrict ourselves to $\phi_1$  given by
\eq{phi1z}
\phi_1\colon (z^+,z^-)\to (z^+,\mathbf{B}z^-).
\eeq
Then $\tilde \phi_1$ yields  the same family
$\{T_{1j}\}$ if and only if its corresponding matrix
 $\widetilde{\mathbf{B}}=\mathbf{B}\mathbf{D}$ for a diagonal matrix $\mathbf{D}$.

In the above we have expressed all $T_{11},\ldots, T_{1p}$ via equivalence classes of matrices. It will be convenient to
restate them via matrices.

For simplicity,  $T_i$ and $S$ denote $\hat T_i,\hat S$, respectively.  In matrices, we write
$$
 T_1\colon
\left(\begin{array}{c}\xi \\ \eta \end{array}\right)
 \to \mathbf{T}_1\left(\begin{array}{c}\xi \\ \eta \end{array}\right),\quad
 \rho\colon \left(\begin{array}{c}\xi \\ \eta \end{array}\right)\to \boldsymbol{\rho}\left(\begin{array}{c}
 \ov\xi \\
 \ov\eta \end{array}\right), \quad
S\colon \left(\begin{array}{c}\xi \\ \eta \end{array}\right)\to \mathbf{S}\left(\begin{array}{c}\xi \\ \eta \end{array}\right).
$$
Recall that  the bold faced  $\mathbf{ A}$  represents a linear map $A$.  Then
\gan
\mathbf{T_1}=\begin{pmatrix}
\mathbf{ 0} &\mathbf{\Lambda}_1  \\
\mathbf{\Lambda}_1^{-1}  &\mathbf{ 0}
\end{pmatrix}_{2p\times2p}, \quad \mathbf{S}=\begin{pmatrix}
\mathbf{\Lambda}_1^2 & \mathbf{ 0} \\
\mathbf{ 0} & \mathbf{\Lambda}_1^{-2}
\end{pmatrix}_{2p\times2p}.
\end{gather*}
We will abbreviate
$$
{\boldsymbol\xi}_{e_*}=(\xi_1,\ldots,\xi_{e_*}), \quad
{\boldsymbol\xi}_{h_*}=(\xi_{e_*+1},\ldots, \xi_{e_*+h_*}), \quad
{\boldsymbol\xi}_{2s_*}=(\xi_{e_*+h_*+1},\ldots,\xi_p).
$$
We use the same abbreviation for $\eta$. Then $({\boldsymbol\xi}_{e_*},{\boldsymbol\eta}_{e_*})$, $({\boldsymbol\xi}_{h_*},{\boldsymbol\eta}_{h_*})$, and
$({\boldsymbol\xi}_{2s_*},{\boldsymbol\eta}_{2s_*})$ subspaces are invariant under $T_{1j}$,
$T_1$, and $\rho$. We also denote by
$T^{e_*}_1, T^{h_*}_1,T^{s_*}_1$ the restrictions of $T_1$ to these
subspaces. Define analogously for the restrictions of $\rho, S$ to these subspaces.
Define diagonal matrices $\mathbf{\Lambda}_{1e_*}, \boldsymbol{\Lambda}_{1h_*}, \mathbf{\Lambda}_{1s_*}$, of size $e_*\times e_*,h_*\times h_*$ and $ s_*\times s_*$ respectively, by
$$
\boldsymbol{\Lambda}_{1}=\begin{pmatrix}
\mathbf{\Lambda}_{1e_*} &\mathbf{ 0 } &\mathbf{ 0} &\mathbf{ 0} \\
\mathbf{ 0} & \mathbf{\Lambda}_{1h_*}&  \mathbf{ 0}&\mathbf{ 0}\\
\mathbf{ 0} & \mathbf{ 0} &\mathbf{\Lambda}_{1s_*} &\mathbf{ 0}\\
\mathbf{ 0}&\mathbf{ 0}&\mathbf{ 0}&\ov {\mathbf{\Lambda}}_{1s_*}^{-1}
\end{pmatrix} ,\quad
\ov{\boldsymbol{\Lambda}_{1}}=\begin{pmatrix}
\mathbf{\Lambda}_{1e_*} &\mathbf{ 0}  &\mathbf{ 0} &\mathbf{ 0} \\
\mathbf{ 0} & \mathbf{\Lambda}_{1h_*}^{-1}& \mathbf{  0}&\mathbf{0}\\
\mathbf{0} & \mathbf{0} &\ov \Lambda_{1s_*} &\mathbf{0}\\
\mathbf{0}&\mathbf{\mathbf{0}}&\mathbf{0}&\mathbf{\Lambda}_{1s_*}^{-1}
\end{pmatrix}.
$$
Thus, we can
express  $T_1^{s_*}$ and $S^{s_*}$ in $(2s_*)\times(2s_*)$
matrices
$$
\mathbf{T}_1^{s_*}
=\begin{pmatrix}
\mathbf{0}& \mathbf{0} &\mathbf{\Lambda}_{1s_*} &  \mathbf{0} \\
 \mathbf{0}& \mathbf{0} &\mathbf{0}&\ov{\mathbf{\Lambda}}_{1s_*}^{-1}\\
 \mathbf{\Lambda}_{1s_*}^{-1}&\mathbf{0}&\mathbf{0} &\mathbf{0}\\
\mathbf{0} &\ov{\mathbf{\Lambda}}_{1s_*}&\mathbf{0} & \mathbf{0}
\end{pmatrix},\quad
\mathbf{S}^{s_*}=\begin{pmatrix}
\mathbf{\Lambda}_{1s_*}^2 & \mathbf{0}& \mathbf{0} & \mathbf{0} \\
\mathbf{0} &\ov{\mathbf{\Lambda}}_{1s_*}^{-2} & \mathbf{0} & \mathbf{0}\\
\mathbf{0}&\mathbf{0} & \mathbf{\Lambda}_{1s_*}^{-2}&\mathbf{0}\\
\mathbf{0}&\mathbf{0} & \mathbf{0} &\ov{\mathbf{\Lambda}}_{1s_*}^{2}
\end{pmatrix}.
$$
Let $\mathbf{ I}_k$ denote the $k\times k$
identity matrix.
With the abbreviation, we can express $\rho$ as
\begin{gather*}
\boldsymbol{\rho}^{e_*}=\begin{pmatrix}
\mathbf{0} & \mathbf{I}_{e_* }\\
\mathbf{I} _{e_* }& \mathbf{0}
\end{pmatrix},\quad
\boldsymbol{\rho}^{h_*}=\mathbf{I}_{2h_*},\quad
\\
\boldsymbol{\rho}^{s_*}
=\begin{pmatrix}
\mathbf{0} & \mathbf{I}_{s_*}
& \mathbf{0} & \mathbf{0} \\
\mathbf{I}_{s_*} & \mathbf{0} & \mathbf{0} & \mathbf{0}\\
\mathbf{0}&\mathbf{0} & \mathbf{0}& \mathbf{I}_{s_*}\\
\mathbf{0}&\mathbf{0} & \mathbf{I} _{s_*}&\mathbf{0}
\end{pmatrix}.
\end{gather*}
Note that   $\rho$ is anti-holomorphic linear transformation.  
If $A$ is a complex linear transformation,  in $(\xi,\eta)$ coordinates the matrix of $\rho A$ is 
$\boldsymbol{\rho}\ov{\mathbf{A}}$, i.e.
$$
\rho A\colon
\begin{pmatrix}\xi\\ \eta\end{pmatrix}\to \boldsymbol{\rho}\ov{\mathbf A}\begin{pmatrix}\ov \xi\\ \ov\eta\end{pmatrix}
$$
 with
\ga\label{lrho}
\nonumber
\boldsymbol{\rho}=
\begin{pmatrix}
\mathbf{0} &\mathbf{0}  &\mathbf{0}   &\mathbf{0}&\mathbf{I}_{e_*}  & \mathbf{0} & \mathbf{0} &\mathbf{0}  \\
\mathbf{0}&\mathbf{I}_{h_*}   & \mathbf{0} & 0&\mathbf{0}&\mathbf{ 0}  &\mathbf{ 0}  &\mathbf{ 0}  \\
\mathbf{ 0} & \mathbf{ 0} &  \mathbf{  0} & \mathbf{I}_{s_*} &\mathbf{ 0 } & \mathbf{ 0} & \mathbf{ 0} &\mathbf{0}\\
\mathbf{0} & \mathbf{0} &  \mathbf{I}_{s_*} & \mathbf{0}&\mathbf{0} & \mathbf{0} & \mathbf{0} & \mathbf{0} \\
\mathbf{I}_{e_*} & \mathbf{0} &\mathbf{0} &\mathbf{0} & \mathbf{0} & \mathbf{0} & \mathbf{0} & \mathbf{0} \\
\mathbf{0} &\mathbf{0}  & \mathbf{0}  & \mathbf{0}&\mathbf{0} &\mathbf{I}_{h_*}  &\mathbf{0}  & \mathbf{0} \\
\mathbf{0} & \mathbf{0} & \mathbf{0}  &\mathbf{0}&\mathbf{0}  &\mathbf{0}  & \mathbf{0} &  \mathbf{I}_{s_*} \\
 \mathbf{0}&\mathbf{0}  & \mathbf{0}  &\mathbf{0}& \mathbf{0} & \mathbf{0} & \mathbf{I}_{s_*}  &\mathbf{0}
\end{pmatrix}.
\end{gather}

For an invertible $p\times p$ matrix $\mathbf A$,  let us define an $n\times n$ matrix $\mathbf E_{\mathbf A}$ by
\eq{Elam}
\mathbf{E}_{\mathbf A}:=\f{1}{2}\begin{pmatrix}
\mathbf{I}_p & -\mathbf{A} \\
\mathbf{A}^{-1} & \mathbf{I}_p
\end{pmatrix},\quad
\mathbf{E}_{\mathbf A}^{-1}=\begin{pmatrix}
\mathbf{I}_p & \mathbf{A} \\
-\mathbf{A}^{-1} & \mathbf{I}_p
\end{pmatrix}.
\eeq
For a $p\times p$ matrix $\mathbf{B}$,  we define
\eq{bstar}\nonumber
\mathbf{B}_*:=\begin{pmatrix}
\mathbf{I} _p&\mathbf{0} \\
\mathbf{0} & \mathbf{B}
\end{pmatrix}.
\eeq
Therefore,   we can express
\ga\label{t1jd}
\mathbf{T}_{1j}=\mathbf{E}_{\mathbf \Lambda_1} \mathbf{B}_*\mathbf{Z}_j\mathbf{B}_*^{-1}\mathbf{E}_{\mathbf \Lambda_1}^{-1}, \quad \mathbf{T}_{2j}=\boldsymbol{\rho} \ov{\mathbf{T}_{1j}}\boldsymbol{\rho}, \\
\mathbf{Z}_j=\diag(1,\ldots,1,-1,1,\ldots, 1).\label{t1jd+}
\end{gather}
Here $-1$ is at the $(p+j)$-th place.   By \rl{zcent},  $\mathbf{B}$ is uniquely determined up to equivalence relation via diagonal matrices $\mathbf{D}$:
\eq{bsdb}
\mathbf{B}\sim \mathbf{B}\mathbf{D}.
\eeq
We have  expressed all $\{T_{11},\ldots, T_{1p},\rho\}$    for which
  $\hat T_1=T_{11}\cdots T_{1p}$ and $\rho$ are in the normal forms  in \rl{t1t2sigrho} and we have found an equivalence relation
 to classify the involutions. Let us summarize the results in a lemma.
\begin{lemma}\label{sett}
 Let $\{T_{11},\ldots, T_{1p},\rho\}$ be the involutions of a quadric manifold $M$. Assume that $
 S =T_1\rho T_1\rho$ has
distinct eigenvalues. Then in suitable linear $(\xi,\eta)$
 coordinates, $T_{11},\ldots, T_{1p}$ are given by   \rea{t1jd}, while $T_{11}\cdots T_{1p}=\hat T_1$ and $ \rho$ are given by
\rea{hatt1} and \rea{eqrh}, respectively.
Moreover, $\mathbf{B}$ in \rea{t1jd} is uniquely determined by the equivalence relation \rea{bsdb} for diagonal matrices $\mathbf{D}$.
\end{lemma}
  We remind the reader that we divide the classification for $\{T_{11},\ldots, T_{1p},\rho\}$ into two steps.
We have obtained the classification for the composition $T_{11}\cdots T_{1p}=\hat T_1$ and $\rho$ in \rl{t1t2sigrho}.
Having found   all $\{T_{11},\ldots, T_{1p},\rho\}$ and an equivalence relation,
we are ready to reduce their classification  
to an equivalence problem that involves two dilatations and a coordinate permutation. 
\begin{lemma}\label{unsol}    Let $\{T_{i1},\ldots, T_{ip},\rho\}$ be given by  \rea{t1jd}.  Suppose that $\hat T_1=T_{11}\cdots T_{1p}$,  $\rho$,   $\hat T_2=\rho\hat T_1\rho$, and
$\hat S=\hat T_1\hat T_2$
have the forms in \rla{t1t2sigrho}.  Suppose that $\hat S$ has distinct eigenvalues.
 Let $\{\hat{T}_{11},\ldots,\hat{ T}_{1p},\rho\}$ be given by \rea{t1jd}
where $\la_j$ are unchanged  and $\mathbf B$ is replaced by   $\hat{\mathbf B}$. Suppose that
 $R^{-1}T_{1j}R=\widehat T_{1\nu(j)}$ for all $j$ and $R\rho=\rho R$.
  Then the matrix of $R$ is
$\mathbf R=\diag( \mathbf{a}, \mathbf{a})$ with $\mathbf{a}=(\mathbf{ a}_{e_*},\mathbf{ a}_{h_*},\mathbf{ a}_{s_*},\mathbf{ a}_{s_*}')$,
while $\mathbf a$ satisfies  the reality condition
\ga
\label{aeae} \mathbf{ a}_{e_*}  \in(\rr^*)^{e_*}, \quad \mathbf{ a}_{h_*}\in(\rr^*)^{h_*}, \quad
\ov{\mathbf{ a}_{s_*}}=\mathbf{ a}_{s_*}'\in(\cc^*)^{s_*}.
\end{gather}
Moreover, there exists $\mathbf{d}\in(\cc^*)^p$  such that
\ga\label{bcbd-}
\hat {\mathbf{B}}=(\diag \mathbf{a} )^{-1}\mathbf{B}(\diag_{ \nu} \mathbf{d}), \quad i.e.,
\quad a_i^{-1}b_{i\nu^{-1}(j)}d_{\nu^{-1}(j)}=\hat b_{ij},
 \quad
1\leq i,j\leq p.
\end{gather}
Conversely, if $\mathbf{ a},\mathbf{ d}$ satisfy \rea{aeae} and \rea{bcbd-}, then $R^{-1}T_{1j}R=\hat T_{1\nu(j)}$
and $R\rho=\rho R$.
\end{lemma}
\begin{proof}
Suppose that $R ^{-1}  T_{1j} R=\widehat T_{1\nu(j)}$ and $ R \rho=\rho R $. Then $ R^{-1}\hat T_1 R= \hat T_1$ and $R^{-1}\hat SR=\hat S$.  The latter implies that the matrix of
 $R$ is diagonal. The former   implies that
\eq{phi0}
\nonumber
 R \colon \xi_j'=a_j\xi_j,\quad \eta_j'=a_j\eta_j
\eeq
with $  a_j\in\cc^*$.  Now $ R \rho=\rho R $ implies \re{aeae}.
We express $R ^{-1}  T_{1j} R=\widehat T_{1\nu(j)}$ via matrices:
\eq{ela1t}
\mathbf{E}_{\mathbf \Lambda_1}  \widehat{\mathbf{B}}_*\mathbf{Z}_{\nu(j)}\widehat{\mathbf{B}}_*^{-1} \mathbf{E}_{\mathbf \Lambda_1}^{-1}
=\mathbf{R}^{-1}\mathbf{E}_{\mathbf \Lambda_1} \mathbf{B}_*\mathbf{Z}_j\mathbf{B}_*^{-1}\mathbf{E}_{\mathbf \Lambda_1}^{-1}\mathbf{R}.
\eeq
In view of formula \re{Elam}, we see that $\mathbf{E}_{\mathbf \Lambda_1}$ commutes with $\mathbf{R}=\diag(\mathbf{a},\mathbf{a})$.
 The above is equivalent to that $\boldsymbol{\psi}:=\mathbf{B}_*^{-1}\mathbf{R}\widehat{\mathbf{B}}_*$ satisfies $ \mathbf{Z}_{\nu(j)}={\boldsymbol{\psi}}^{-1}\mathbf{Z}_j\boldsymbol{\psi}$.
   By \rl{zcent}  we obtain $\boldsymbol{\psi}=\diag(\mathbf{A},\diag_{\nu}\mathbf{d})$. This shows that
$$
\begin{pmatrix}
\mathbf{A}&\mathbf{ 0}\\
 \mathbf{0}  &  \diag_{\nu}\mathbf{d}
\end{pmatrix} =\begin{pmatrix}
\mathbf{I } &\mathbf{0}\\
 \mathbf{0}  &\mathbf{B}
\end{pmatrix}^{-1}\begin{pmatrix}
\diag \mathbf{a}  &\mathbf{0}\\
 \mathbf{0}  &\diag \mathbf{a}
\end{pmatrix}\begin{pmatrix}
\mathbf{ I} &\mathbf{0}\\
 \mathbf{0}  &  \widehat{\mathbf{B}}
\end{pmatrix}.
$$
The matrices on diagonal yield $\mathbf{A}=\diag \mathbf{a}$  and \re{bcbd-}.
The lemma is proved.
\end{proof}

\rl{unsol}  does not give us an explicit description of the normal form  for the families of
involutions $\{T_{11},\ldots, T_{1p},\rho\}$.  Nevertheless by the lemma, we can always choose a $\nu$ and $\diag {\mathbf d}$ such that the diagonal elements of $\mathbf{ \tilde B}$,
corresponding to $\{\tilde T_{1\nu(1)},\ldots, \tilde T_{1\nu(p)}, \rho\}$,  are $1$.

\begin{rem}
In what follows, we will fix a $\mathbf{B}$ and its associated  $\{\cL T_1,\rho\}$ to further
study our normal form problems.
\end{rem}

\subsection{Normal form of the  quadrics}

  We now use  the matrices $\mathbf{B}$  to express
the normal form for the quadratic submanifolds. Here we follow the realization procedure in    \rp{mmtp}.
We will use the coordinates $z^+, z^-$ again to express invariant functions of $T_{1j}$ and
use them to construct the corresponding quadric.
We will then pull back  the quadric to the $(\xi,\eta)$ coordinates and then to the $z,\ov z$ coordinates
to achieve the final normal form of the quadrics.

We return to the construction of invariant and skew-invariant functions $z^+,z^-$ in \re{zep1}-\re{zep4}
when $\mathbf{B}$ is the identity matrix. For a general $\mathbf{B}$, we define $\Phi_1$ and the matrix
$\mathbf{\Phi}_1^{-1}$ by
$$
\Phi_1(Z^+,Z^-)=(\xi,\eta), \quad
\mathbf{\Phi}_1^{-1}:=
 \mathbf{B}_*^{-1}\mathbf{E}_{\mathbf \Lambda_1}^{-1}= \begin{pmatrix}
\mathbf{I} &\mathbf{\Lambda}_1 \\
 -\mathbf{B}^{-1}\mathbf{\Lambda}_1^{-1} & \mathbf{B}^{-1}
\end{pmatrix}.
$$
Note that $Z^+=z^+$ and $\Phi_1^{-1}T_{1j}\Phi_1=Z_j$.  The  $Z^+, Z_i^-$ with $i\neq j$
are invariant functions of $T_{1j}$, while $Z_j^-$
is a skew-invariant function
of $T_{1j}$. They can be written as
\eq{Z+xi}
Z^+=\xi+\mathbf{\Lambda}_1\eta, \quad Z^-=\mathbf{B}^{-1}(-\mathbf{\Lambda}_1^{-1}\xi+\eta).
\eeq
Therefore, the invariant functions of $\cL T_1$ are generated
by
$$
Z_j^+=\xi_j+\la_{j}\eta_j, \quad (Z_j^-)^2=(\tilde {\mathbf{ B}}_j(-\mathbf{\Lambda}_1^{-1}\xi+\eta))^2, \quad 1\leq j\leq p.
$$
Here 
$\tilde {\mathbf{ B}}_j$ is the $j$th row of $\mathbf{B}^{-1}$.
The invariant (holomorphic)
 functions of $\cL T_2$ are generated
by
\eq{wjpm}
W_j^+=\ov{Z_j^+\circ\rho}, \quad( W_j^-)^2=(\ov{Z^-_j\circ\rho})^2,
\quad 1\leq j\leq p.
\eeq
  Here $W_j^-=\ov{Z^-_j\circ\rho}$.
We will soon verify that
$$
m\colon (\xi,\eta)\to (z',w')=(Z^+(\xi,\eta), W^+(\xi,\eta))
$$
is biholomorphic.  A straightforward computation shows that
$m\rho m^{-1}$ equals 
$$
\rho_0\colon (z',w')\to(\ov {w'},\ov {z'}).
$$
We define
\ga\label{m0zp}
\nonumber
M\colon z_{p+j}''=(Z_j^-\circ m^{-1}(z',\ov{z'}))^2.\end{gather}
We want to find a simpler expression for $M$.
We first separate $B$ from $Z^-$ by writing
\ga\label{Zhmi}
\mathbf{  \hat Z}^-:= (-{\mathbf \Lambda}_1^{-1}\  \mathbf{ I}),
 \quad \mathbf{ Z}^-=\mathbf{ B}^{-1}\mathbf{ \hat Z}^-.
\end{gather}
Note that $m$ does not depend on $\mathbf B$.  To compute $\hat Z^-\circ m^{-1}$, we will use
matrix expressions for  $({\boldsymbol\xi}_{e_*},{\boldsymbol\eta}_{e_*})$, $({\boldsymbol\xi}_{h_*},{\boldsymbol\eta}_{h_*})$
and $({\boldsymbol\xi}_{2s_*}, {\boldsymbol\eta}_{2s_*})$ subspaces.
Let $m_{e_*},m_{h_*},m_{s_*}$ be the restrictions  $m$ to these
subspaces. In the matrix form,
we have   by \re{wjpm}
$$
\mathbf{ W}^+=\ov {\mathbf{ Z}^+\boldsymbol{\rho}}, \quad \mathbf{ W}^-=\ov{\mathbf{  Z}^-
\boldsymbol{\rho}}.
$$
Recall that $\mathbf{\Lambda}_1=\diag({\mathbf \Lambda}_{e_*},{\mathbf \Lambda}_{h_*},{\mathbf \Lambda}_{1s_*}, \ov{\mathbf \Lambda}_{1s_*}^{-1})$. Thus
\aln
\mathbf{ m}_{e_*}&=\begin{bmatrix}
\mathbf{ I} & {\mathbf \Lambda}_{1{e_*}} \\
{\mathbf \Lambda}_{1{e_*}} & \mathbf{ I}
\end{bmatrix},\quad
\mathbf{ m}_{e_*}^{-1}=
\begin{bmatrix}
\mathbf{ I} &- {\mathbf \Lambda}_{1{e_*}} \\
-{\mathbf \Lambda}_{1{e_*}} &\mathbf{  I}
\end{bmatrix}\begin{bmatrix}
(\mathbf{ I}-{\mathbf \Lambda}_{1{e_*}}^2)^{-1} &\mathbf{0} \\
\mathbf{0} & (\mathbf{ I}-{\mathbf \Lambda}_{1{e_*}}^2)^{-1}
\end{bmatrix},\\
\mathbf{ m}_{h_*}&=\begin{bmatrix}
\mathbf{ I} & {\mathbf \Lambda}_{1{h_*}} \\
\mathbf{ I} & {\mathbf \Lambda}_{1{h_*}}^{-1}
\end{bmatrix},\hspace{5ex}
\mathbf{ m}_{h_*}^{-1}=
\begin{bmatrix}
\mathbf{  I }& -{\mathbf \Lambda}_{1{h_*}}^2 \\
-{\mathbf \Lambda}_{1{h_*}} &{\mathbf \Lambda}_{1{h_*}}
\end{bmatrix}\begin{bmatrix}
(\mathbf{ I}-{\mathbf \Lambda}_{1{h_*}}^2)^{-1}&\mathbf{0} \\
\mathbf{0} &(\mathbf{ I}-{\mathbf \Lambda}_{1{h_*}}^2)^{-1}
\end{bmatrix},
\\
\mathbf{ m}_{s_*}&=
\begin{bmatrix}
{\mathbf  I} & {\mathbf  0}&{\mathbf \Lambda}_{1{s_*}} &{\mathbf  0}\\
{\mathbf  0}&{\mathbf  I} & {\mathbf  0}&\ov{\mathbf \Lambda}_{1{s_*}}^{-1}\\
{\mathbf  0}&{\mathbf  I}&{\mathbf  0}&\ov{\mathbf \Lambda}_{1{s_*}}\\
{\mathbf  I}&{\mathbf  0}&{\mathbf \Lambda}_{1{s_*}}^{-1}&{\mathbf  0}
\end{bmatrix},\\
\mathbf{ m}_{s_*}^{-1}&=
\begin{bmatrix}
{\mathbf \Lambda}_{1{s_*}}^{-1}&{\mathbf  0}&{\mathbf  0}&-{\mathbf \Lambda}_{1{s_*}}\\
{\mathbf  0}&\ov{\mathbf \Lambda}_{1{s_*}} &-\ov{\mathbf \Lambda}_{1{s_*}}^{-1}&{\mathbf  0}\\
-{\mathbf  I}&{\mathbf  0}&{\mathbf  0}&{\mathbf  I}\\
{\mathbf  0}&-{\mathbf  I}&{\mathbf  I}&{\mathbf  0}
\end{bmatrix}
\begin{bmatrix}
\mathbf{ L}_{s_*}&{\mathbf  0}\\
{\mathbf  0}&-\ov{\mathbf L}_{s_*}
\end{bmatrix},\\
\mathbf{ L}_{s_*}&=\begin{bmatrix}
({\mathbf \Lambda}_{1{s_*}}^{-1}-{\mathbf \Lambda}_{1{s_*}})^{-1}&{\mathbf  0}\\
{\mathbf  0}& (\ov{\mathbf \Lambda}_{1{s_*}}-\ov{\mathbf \Lambda}_{1{s_*}}^{-1})^{-1}
\end{bmatrix}.
\end{align*}
Note that $\mathbf{ I}-{\mathbf \Lambda}_{1}^2$ is diagonal.  Using \re{Zhmi} and
the above formulae, the matrices of $\hat Z_{e_*}^{-1}\circ m^{-1}$,
$\hat Z_{h_*}^-\circ m^{-1}$, and $\hat Z_{s_*}^{-1}\circ m^{-1}$ are respectively given by
\begin{align*}
\mathbf{\hat Z}^-_{e_*}\mathbf{ m}_{e_*}^{-1}
&=\mathbf{ L}_{e_*}
\begin{bmatrix}
\mathbf{ I}  &-2({\mathbf \Lambda}_{1{e_*}}+{\mathbf \Lambda}_{1{e_*}}^{-1})^{-1}
\end{bmatrix},\\
\mathbf{ L}_{e_*}&=
(\mathbf{ I}-{\mathbf \Lambda}_{1{e_*}}^2)^{-1}(-{\mathbf \Lambda}_{1{e_*}}-{\mathbf \Lambda}_{1{e_*}}^{-1}),\\
\mathbf{ \hat Z}_{h_*}^-\mathbf{ m}_{h_*}^{-1}
&=\mathbf{ L}_{h_*}
\begin{bmatrix}
\mathbf{ I} & -2{\mathbf \Lambda}_{1{h_*}}({\mathbf \Lambda}_{1{h_*}}+{\mathbf \Lambda}_{1{h_*}}^{-1})^{-1}
\end{bmatrix},\\
\mathbf{ L}_{h_*}&=(\mathbf{ I}-{\mathbf \Lambda}_{1{h_*}}^2)^{-1}
(-{\mathbf \Lambda}_{1{h_*}}-{\mathbf \Lambda}_{1{h_*}}^{-1}),
\\ 
\mathbf{ \hat Z}_{s_*}^- \mathbf{ m}_{s_*}^{-1}&=
\begin{bmatrix}
-\mathbf{ I}-{\mathbf \Lambda}_{1{s_*}}^{-2}&{\mathbf  0}&{\mathbf  0}&2\mathbf{ I}\\
{\mathbf  0}&-\mathbf{ I}-\ov{\mathbf \Lambda}_{1{s_*}}^{2} &2\mathbf{ I} & {\mathbf  0}
\end{bmatrix}
\begin{bmatrix}
\mathbf{ L}_{s_*}&{\mathbf  0}\\
{\mathbf  0}&-\ov{ \mathbf{ L}}_{s_*}
\end{bmatrix}
\\
&=\mathbf{ \tilde L}_{s_*}
\begin{bmatrix}
\mathbf{ I}&{\mathbf  0}&{\mathbf  0}  &-  2(\mathbf{ I}+{\mathbf \Lambda}_{1{s_*}}^{-2})^{-1}\\
{\mathbf  0}&\mathbf{ I}&-  2 (\mathbf{ I}+\ov{\mathbf \Lambda}_{1{s_*}}^{2})^{-1} &{\mathbf  0}
\end{bmatrix},\\
\mathbf{ \tilde L}_{s_*}&=
\begin{bmatrix}
(\mathbf{ I}+{\mathbf \Lambda}_{1{s_*}}^{-2})({\mathbf \Lambda}_{1{s_*}}-{\mathbf \Lambda}_{1{s_*}}^{-1})^{-1}&{\mathbf  0}\\
{\mathbf  0}&(\mathbf{ I}+\ov{\mathbf \Lambda}_{1{s_*}}^{2})(\ov{\mathbf \Lambda}_{1{s_*}}^{-1}-\ov{\mathbf \Lambda}_{1{s_*}})^{-1}
\end{bmatrix}.
\end{align*}
Combining the above identities, we obtain
$$
\mathbf{ \hat Z}^{-1}\mathbf{ m}^{-1}=
\diag(\mathbf{ L}_{e_*}, \mathbf{ L}_{h_*},\mathbf{ \tilde L}_{s_*})\left(\mathbf{ I}_p,
  -2\diag\left(\boldsymbol{ \Gamma}_{e_*},{\mathbf \Lambda}_{1{h_*}}\boldsymbol{ \Gamma}_{h_*},
   \begin{bmatrix}0& \tilde{\boldsymbol{ \Gamma}}_{s_*}\\
  {\boldsymbol{ \Gamma}}_{s_*} &0\end{bmatrix}\right)\right)
$$
with  $\tilde{\mathbf{ \Gamma}}_{s_*}=\mathbf{ I}-\ov{\mathbf{ \Gamma}}_{1{s_*}}$  and  \ga\label{Gehs}
\mathbf{ \Gamma}_{e_*}=({\mathbf \Lambda}_{1{e_*}}+{\mathbf \Lambda}_{1{e_*}}^{-1})^{-1},
\quad\mathbf{  \Gamma}_{h_*}=({\mathbf \Lambda}_{1{h_*}}+{\mathbf \Lambda}_{1{h_*}}^{-1})^{-1},\quad
\mathbf{ \Gamma}_{s_*}= (\mathbf{ I}+\ov{\mathbf \Lambda}_{1{ s_*}}^{2})^{-1}.
\end{gather}
  We define $\tilde{\mathbf B}_j$ to be the $j$-th row of
\eq{deftb}
\tilde {\mathbf B}:=\mathbf{ B}^{-1}\diag(\mathbf{ L}_{e_*}, \mathbf{ L}_{h_*},\mathbf{ \tilde L}_{s_*}).
\eeq

With $\mathbf  z_{s_*}'=(z_{p-s_*+1}, \ldots, z_p)$,  the defining equations of $M$ are given by
$$
z_{p+j}''=\bigl\{\mathbf{ \tilde B}_j\cdot 
(\mathbf z_{{e_*}}-2\mathbf{ \Gamma}_{e_*}\ov {\mathbf z}_{e_*},
\mathbf z_{h_*}
-
2\mathbf{ \Gamma}_{h_*}\mathbf{\Lambda}_{1h_*}\ov{ \mathbf  z}_{h_*}, \mathbf z_{s_*}-2\mathbf{ \Gamma}_{s_*}\ov{ \mathbf  z}_{s_*}',
{\mathbf  z}_{s_*}'-2(\mathbf{ I}-\ov{\mathbf{ \Gamma}_{s_*}})\ov {\mathbf  z}_{s_*})\bigr\}^2.
$$
 Let us replace   $z_j$ with $j\neq h$,  $z_h$ by $iz_j$ and  $i \sqrt{\la_h}z_h$, respectively for $1\leq j\leq p$. Replace $z_{p+j}$ by $-z_{p+j}$.
 In the new coordinates, $M$ is given
by
$$
z_{p+j}''=\bigl\{\hat{\mathbf{B}}_j
\cdot(\mathbf z_{{e_*}}+2\mathbf{ \Gamma}_{e_*}\ov {\mathbf z}_{e_*},
\mathbf z_{h_*}
+  2\mathbf{ \Gamma}_{h_*}\ov{ \mathbf  z}_{h_*}, \mathbf z_{s_*}+2\tilde{\mathbf{ \Gamma}}_{s_*}\ov{ \mathbf  z}_{s_*}',
{\mathbf  z}_{s_*}'+2 {\mathbf{ \Gamma}}_{s_*}
\ov {\mathbf  z}_{s_*})\bigr\}^2.
$$
 Explicitly, we have
\begin{align}\label{qbga}
Q_{\mathbf B,\gamma}\colon
z_{p+j}&=\Bigl(
\sum_{\ell=1}^{e_*+h_*} \hat b_{j\ell}(z_\ell+2\gamma_\ell\ov z_\ell)
\\ &\qquad  
   +\sum_{s=e_*+h_*+1}^{p-s_*}
 \hat b_{js}(z_s+2\gamma_{s+s_*}\ov z_{s+s_*})  +
 \hat b_{j(s+s_*)}(z_{s+s_*}+2
 \gamma_{s}\ov z_{s})\Bigr)^2  \nonumber
\end{align}
for $ 1\leq j\leq p$. Here
\eq{gsss}\nonumber
\gaa_{s+s_*}=1-\ov\gaa_s.
\eeq
By \re{deftb}, we also obtain the following identity
\eq{defhb}
\nonumber
\hat{\mathbf B}=  \mathbf B^{-1} \diag(\mathbf{ L}_{e_*}, \mathbf{ L}_{h_*},\mathbf{ \tilde L}_{s_*})
\diag(\mathbf I_{e_*}, \mathbf{\Lambda}_{1h_*}^{1/2}, \mathbf I_{2s_*}).
\eeq
The equivalence relation \re{bcbd-} on the set of non-singular matrices $\mathbf B$ now takes the form
\eq{hbsim}
 \widehat{\tilde{\mathbf B}}= (\diag_{ \nu}\mathbf d)^{-1}\hat{\mathbf B}\diag\mathbf a,
\eeq
where $\mathbf a$ satisfies \re{aeae} and   $\diag_{ \nu}\mathbf d$  is defined in \re{dfndiag}.

Therefore, by \rp{mmtp} we obtain the following classification for
the quadrics.

\begin{thm}\label{quadclass}
Let $M$ be a quadratic submanifold defined by \rea{fmzp}-\rea{fmzp+} with $q^{-1}(0)=0$. Assume that the branched covering
 $\pi_1$ has   $2^p$  deck transformations.
 Let $T_1,T_2$ be the pair of Moser-Webster involutions
 of $M$. Suppose that  $S=T_1T_2$
has $2p$ distinct eigenvalues.
Then $M$ is holomorphically equivalent to
\rea{qbga} with  $\hat{\mathbf{ B}}\in GL(p,\cc)$ being uniquely determined by   the equivalence
  relation \rea{hbsim}.
\end{thm}

 When $\hat{\mathbf{ B}}$  is the identity, we obtain
 the product of 3 types of quadrics
\ga
\nonumber
\cL Q_{\gamma_e}\colon z_{p+e}= (z_e+2\gaa_e\ov z_e)^2;\\
\nonumber
\cL Q_{\gamma_h}\colon z_{p+h}= (z_{h}+2\gaa_{h}\ov z_h)^2;\\
\cL Q_{\gamma_s}\colon z_{p+s}= (z_s+2(1-\ov\gamma_{s}) \ov z_{s+s_*})^2,
\quad z_{p+s+s_*}=( z_{s+s_*}+2
 \gamma_{s} \ov z_{s})^{2}
 \label{skewgaas}
\end{gather}
with
\eq{gens}\nonumber
\gamma_e=\frac{1}{\la_e+\la_e^{-1}}, \quad
\gamma_h=\frac{ 1}{\la_h+\ov\la_h},\quad
 \gamma_s=\frac{1}{1+\ov{\la_s}^2}.
\eeq
Note that $\arg\la_s
\in(0,\pi/2)$ and $|\la_s|>1$. Thus
\eq{gehs}
 0<\gaa_e<1/2, \quad
 \gaa_h>1/2, \quad
\gaa_s\in \{z\in\cc\colon \RE z>1/2, \IM z>0\}.
\eeq

\begin{rem} By seeking simple     formulae \re{Z+xi} for invariant functions $Z^+$ of $\{T_{1j}\}$ and \re{wjpm} for invariant functions
$W^+$ of
$\{T_{2j}\}=\{\rho T_{1j}\rho\}$, we have mismatched  the indices so that $W_{s+s_*}^+(\xi,\eta)$,  instead of $W_{s}^+$, is invariant by $T_{2s}$.
In \re{skewgaas} for $p=2$ and $h_*=e_*=0$,  by interchanging $(z_{s},z_{p+s})$ with $(z_{s+s_*},z_{p+s+s_*})$ we get the quadric \re{Qgam}, an equivalent form of \re{skewgaas}.
\end{rem}
We define the following invariants.
\begin{defn}We call $\mathbf{ \Gamma}=\diag(\mathbf{ \Gamma}_{e_*},\mathbf{ \Gamma}_{h_*},\mathbf{ \Gamma}_{s_*}, \mathbf I_{s_*}-\bar{\boldsymbol \Gamma}_{s_*})$,  given by formulae \re{Gehs},
 the {\em Bishop invariants} of the quadrics. The  equivalence classes $\hat{\mathbf B}$ of non-singular matrices
 $\mathbf{ B}$ under the equivalence relation \re{bcbd-} are called the {\it extended   Bishop invariants}
  for the quadrics.
\end{defn}
Note that $\mathbf{ \Gamma}_{e_*}$ has
diagonal elements in $(0,1/2)$, and $\mathbf{ \Gamma}_{h_*}$
has   diagonal elements in $(1/2,\infty)$,
and $\mathbf{ \Gamma}_{s_*}$ has diagonal elements
in $  (-\infty,1/2)+ i (0,\infty)$.

We remark that $Z_j^-$ is
skew-invariant by $T_{1i}$ for $i\neq j$ and invariant by
$\tau_{1j}$. Therefore, the square of a linear combination
of $Z_1^-, \ldots, Z_p^-$ might not be invariant by all
$T_{1j}$. This explains the presence of $\mathbf B$ as invariants
in the normal form.

It is worthy stating the following normal form
for two families of linear holomorphic involutions
which may not satisfy the reality
condition.

\begin{prop}\label{2tnorm}
Let $\cL T_i=\{T_{i1},\ldots, T_{ip}\},i=1,2$ be two families
of distinct and commuting
linear holomorphic involutions on $\cc^n$.
Let $T_i=T_{i1}\cdots T_{ip}$. Suppose that for each $i$, $\fix(T_{i1}), \ldots$, $\fix(T_{ip})$
are hyperplanes intersecting transversally. Suppose that  $T_1,T_2$ satisfy \rea{bilii++} and
$S=T_1T_2$ has $2p$ distinct eigenvalues.
In suitable linear coordinates, the matrices of $T_i,S$ are
\gan
\mathbf{ T}_i=\begin{pmatrix}
                 \mathbf{  0 }& {\mathbf \Lambda}_i \\
                  {\mathbf \Lambda}_i^{-1} & \mathbf{ 0} \\
                \end{pmatrix},\quad
\mathbf{ S}=\begin{pmatrix}
               {\mathbf \Lambda}_1^2   & \mathbf{ 0} \\
               \mathbf{  0} & {\mathbf \Lambda}_1^{-2} \\
                \end{pmatrix}
\end{gather*}
with ${\mathbf \Lambda}_2={\mathbf \Lambda}_1^{-1}$  being diagonal matrix whose entries
do not contain $\pm1, \pm i$. The  ${\mathbf \Lambda}_1^2$ is uniquely determined
up to a permutation in diagonal entries.  Moreover, the matrices of $T_{ij}$ are
\eq{btij}
\mathbf{ T}_{ij}=\mathbf{ E}_{{\mathbf \Lambda}_i}(\mathbf{ B}_i)_*\mathbf{ Z}_j(\mathbf{ B}_i)_*^{-1}\mathbf{ E}_{{\mathbf \Lambda}_i}^{-1}
\eeq
for some non-singular complex matrices $\mathbf{ B}_1, \mathbf{ B}_2$
uniquely determined by   the equivalence
relation
\eq{tb1tb2}
(\mathbf{ B}_1,\mathbf{ B}_2)\sim (\mathbf{ \tilde B}_1,\mathbf{ \tilde B}_2):=
( (\diag\mathbf a) ^{-1}\mathbf{ B}_1\diag_{  \nu_1}\mathbf d_1,(\diag\mathbf a)^{-1}\mathbf{ B}_2\diag_{ \nu_2}\mathbf d_2),
\end{equation}
where $\diag_{\nu_1}\mathbf{ d}_1,\diag_{ \nu_2}\mathbf{ d}_2$ are defined as in \rea{dfndiag},
and $\mathbf{ R}=\diag(\mathbf a,\mathbf a)$ is a non-singular diagonal  complex matrix   representing  the linear transformation $\var$ such that
\eq{}
\nonumber
\var^{-1}T_{ij}\var=\tilde T_{i\nu_i(j)}, \quad i=1,2, j=1, \ldots, p.
\eeq
 Here $\tilde {\cL T}_i$ is the family of  the involutions  associated to the matrices $\mathbf {\tilde B}_i$, and
 $\mathbf E_{\mathbf \Lambda_i}$ and $\mathbf B_*$ are defined by  \rea{Elam}-\rea{bstar}.
\end{prop}
\begin{proof}  Let $\kappa$ be an eigenvalue of $S$ with (non-zero) eigenvector $u$. Since $T_iST_i=S^{-1}$.
Then $S(T_i(u))=\kappa^{-1}T_i(u)$. This shows that $\kappa^{-1}$ is also an eigenvalue of $S$. By \rl{incomp}, $1$ and $ -1$
are not eigenvalues of $S$. Thus, we can list the eigenvalues of $S$ as $\mu_1, \ldots, \mu_p, \mu_1^{-1}, \ldots, \mu_p^{-1}$.
Let $u_j$ be an eigenvector of $S$ with eigenvalue $\mu_j$. Fix $\la_j$ such that $\la_j^2=\mu_j$.
Then $v_j:=\la_j T_1(u_j)$ is an eigenvector of $S$ with eigenvalue $\mu_j^{-1}$. The $\sum\xi_ju_j+\eta_jv_j$ defines
a coordinate system on $\cc^n$ such that $T_i,S$ have  the above matrices $\mathbf{\Lambda}_i$ and $\mathbf S$,
respectively.  By \re{phi1z} and  \re{t1jd},
$T_{ij}$ can be expressed in \re{btij}, where each $\mathbf B_i$ is uniquely determined up to
$\mathbf B_i\diag\mathbf d_i$. Suppose that $  \{\tilde T_{1j}\}, \{\tilde T_{2j}\}$
are another pair of
families of linear involutions of which the corresponding matrices are $ \mathbf{\tilde B}_1, \mathbf {\tilde B}_2$.
  If there is a linear change of coordinates $\var$ such
that $\var^{-1} T_{ij}\var=\tilde T_{i\nu_i(j)}${,
then  in}
the matrix $\mathbf R$ of $\var$, we obtain \re{tb1tb2}; see a similar computation
for \re{bcbd-} by using \re{ela1t}.
Conversely, \re{ela1t} implies that the corresponding pairs of families
of involutions are equivalent.
\end{proof}

Finally, we conclude the section with   examples of quadratic manifolds  of maximum deck transformations for which the corresponding $\sigma$ is not diagonalizable.
\begin{exmp}\label{jordanblocks}
Let $\mathbf K$ be a $p\times p$ invertible matrix.
 Let $T_1, \rho, T_2=\rho T_1\rho, S$ have matrices
$$
\mathbf T_1=\begin{pmatrix}
                 \mathbf{  0 }& \mathbf K \\
                  \mathbf K^{-1} & \mathbf{ 0} \\
                \end{pmatrix},
                \
 \boldsymbol{\rho}=\begin{pmatrix}
                 \mathbf{  0 }&  {\mathbf I}_{p} \\
                 {\mathbf I}_{p} & \mathbf{ 0} \\
                \end{pmatrix},  \
        \mathbf T_2= \begin{pmatrix}
                 \mathbf{  0 }& \ov {\mathbf K}^{-1} \\
                  \ov{\mathbf K} & \mathbf{ 0} \\
                \end{pmatrix},
                  \
 \mathbf S=\begin{pmatrix}
                {\mathbf{  K}} \ov{\mathbf{  K}}& {\mathbf 0} \\
                  {\mathbf 0} &  {\mathbf{  K}}^{-1} \ov{\mathbf{  K}}^{-1} \\
                \end{pmatrix}.
                $$
 One can verify that the sets of fixed points of $T_1,T_2$ intersect transversally if
 \eq{detkk}
 \det({\mathbf K}-\ov{\mathbf K}^{-1})\neq0.
 \eeq
We can decompose $T_1=T_{11}\cdots T_{1p}$ where $T_{11},\dots, T_{1p}$ are commuting involutions  and each of them fixes a hyperplane by using
$$
\begin{pmatrix}
                 \mathbf{  0 }& \mathbf K \\
                  \mathbf K^{-1} & \mathbf{ 0} \\
                \end{pmatrix}
                =
  \begin{pmatrix}
          \mathbf K        &\ \mathbf{  0 }\\
                  \mathbf 0& \mathbf{ I} \\
                \end{pmatrix}
     \begin{pmatrix}
                 \mathbf{  0 }& \mathbf I \\
                  \mathbf I & \mathbf{ 0} \\
                \end{pmatrix}
    \begin{pmatrix}
                 \mathbf{ K }& \mathbf 0 \\
                  \mathbf 0& \mathbf{ I} \\
                \end{pmatrix}^{-1}.
               $$
 In coordinates, we have $T_1\colon (\xi,\eta)^t\to \mathbf T_1(\xi,\eta)^t$. Thus the linear invariant functions of $\{T_{11},\dots, T_{1p}\}$ are precisely generated by linear invariant functions of $T_1$, and they are linear combinations of the entries of the column vector
 $
 \xi^t+\mathbf K\eta^t.
 $
On the other hand, the linear invariant functions of $\{T_{21},\dots, T_{2p}\}$ are linear combinations of the entries of the vector
$\xi^t+\ov{\mathbf K}^{-1}\eta^t.
$
The two sets of entries are linearly independent functions; indeed if there
are row vectors $\mathbf a,\mathbf b$ such that
$$
\mathbf a( \xi^t+\mathbf K\eta^t)+\mathbf b(\xi^t+\ov{\mathbf K}^{-1}\eta^t)=0
$$
then $\mathbf a=\mathbf b$ and $\mathbf a(\mathbf K-\ov{\mathbf K}^{-1})=\mathbf 0$. Thus $\mathbf a=\mathbf 0$ if \re{detkk} holds.
 Thus condition \re{detkk} also implies \re{ilii++}-\re{bilii++}. By \rp{mmtp}, the family of $\{T_{11}, \dots, T_{1p}, \rho\}$, in particular the   matrix $S$, can be realized by a quadratic manifold.

For a more explicit example,  let ${\mathbf J}_p$ be the  $p\times p$ Jordan matrix with  entries $1$ or $0$. Then $\mathbf K=\la\mathbf J_p$   satisfies \re{detkk} if $\la$ is positive and $\la\neq1$, as $\ov {\mathbf K}^{-1}=\ov\la^{-1}\mathbf{J}_p^{-1}$. For another example, set
$$
\mathbf K_\la=\begin{pmatrix}
                 \mathbf{  0 }& \la{\mathbf J}_{q} \\
                 \ov \la{\mathbf J}_{q}& \mathbf{ 0} \\
                \end{pmatrix}, \quad
   \mathbf K_\la\ov{ \mathbf K_\la}=\begin{pmatrix}
               \la^2{\mathbf J}_{q} ^2  &\mathbf{  0 } \\
            \mathbf{ 0}    &   \ov \la^2{\mathbf J}^2_{q}\\
                \end{pmatrix}
                $$
with $q=p/2$ and $p$ even.
If  $\la\in\cc$ and $\la\neq0,\pm1$ then $K$ satisfies \re{detkk} as
$$
\ov{\mathbf K_\la}^{-1}=\begin{pmatrix}
                 \mathbf{  0 }& \la^{-1}{\mathbf J}_{q}^{-1} \\
                 {\ov \la}^{-1}{\mathbf J}_{q}^{-1}& \mathbf{ 0} \\
                \end{pmatrix}.
                $$
 When $\la=1$ and $q=2$, we obtain $\mathbf S$ in \re{jsig}
 if the above  $\mathbf J_2$ is replaced by $\mathbf J_2^{1/2}$, the Jordan matrix with eigenvalue $1$ and off-diagonal entries $1/2$.
\end{exmp}
\setcounter{thm}{0}\setcounter{equation}{0}

\setcounter{thm}{0}\setcounter{equation}{0}
\section{
Formal deck
transformations  and centralizers}\label{fsubm}

In section~\ref{secinv} we  describe 
 the equivalence of   the classification of real analytic submanifolds $M$ that admit the maximum number of deck transformations and the classification of the families of
involutions $\{\tau_{11}, \ldots, \tau_{1p},\rho\}$ that satisfy some mild conditions  (see \rp{mmtp}). To classify the families of involutions and to find their normal forms, we
will also study the centralizers of various linear maps to deal with  resonance.
This is relevant as the normal form of $\sigma$ will belong to the centralizer of its linear part
and any further normalization will also be performed by transformations that are
in the centralizer.

In this subsection, we   describe   centralizers regarding $\hat S, \hat T_1$ and $\hat{\mathcal T}_1$.
We will also describe the complement sets of the centralizers, i.e. the sets
of mappings which satisfy suitable normalizing conditions. Roughly speaking, our normal forms are   in
the centralizers and coordinate transformations that achieve the normal forms are normalized, while
an arbitrary formal transformation admits a unique  decomposition of a mapping in a centralizer and a mapping in the complement of
the centralizer.

Recall that
\ga\label{sxij}
\hat S\colon\xi_j'=\mu_j\xi_j, \quad\eta_j'=\mu_j^{-1}\eta_j, \quad 1\leq j\leq p,\\
\hat T_i\colon \xi_j'=\la_{ij}\eta_j,\quad \eta_j'=\la_{ij}^{-1}\xi, \quad 1\leq j\leq p
\label{tixi}
\end{gather}
with $\mu_j=\la_{1j}^2$ and $\la_{2j}^{-1}=\la_{1j}=\la_j$.

\begin{defn}Let $\cL  F$ be a family of formal mappings  on $\cc^n$  fixing
the origin.
Let $\cL C(\cL  F)$
 be the  {\it centralizer} of $ \cL   F$, i.e.
 the set of formal  holomorphic mappings $g$ that fix  the origin and
  commute with each element $f$ of $\cL  F$, i.e., $f\circ g=g\circ f$.
\end{defn}

 Note that we do not require that elements in $\cL C(\cL  F)$   be invertible or convergent.

We first compute the centralizers.
\begin{lemma}\label{cents}
Let $\hat S$  be given by \rea{sxij} with  $\mu_1,\ldots, \mu_p$
being non-resonant.
Then ${\cL C}(\hat S)$ consists of mappings of the form
\eq{pajb}
\psi\colon \xi_j'= a_j(\xi\eta)\xi_j,\quad \eta_j'= b_j(\xi\eta)\eta_j, \quad 1\leq j\leq p.
\eeq
Let $\tau_1,\tau_2$ be  formal  holomorphic
involutions  such that $\hat S=\tau_1\tau_2$. Then
$$
\tau_{i}\colon\xi_j'=\Lambda_{ij}(\xi\eta)\eta_j, \quad \eta_j'=\Lambda_{ij}^{-1}(\xi\eta)\xi_j, \quad 1\leq j\leq p
$$
with $\Lambda_{1j}\Lambda_{2j}^{-1}=\mu_j$.
The   centralizer of   $\{\hat T_1,\hat T_2\}$ consists of the above transformations satisfying
 \eq{bjaj}
 b_j=a_j, \quad 1\leq j\leq p.
 \eeq
\end{lemma}
\begin{proof} Let $e_j=(0,\ldots, 1,\ldots, 0)\in \nn^p$, where $1$ is at the $j$th place.  Let $\psi$ be given by $$\xi_j'=\sum a_{j,PQ}\xi^P\eta^Q,\quad\eta_j'=\sum b_{j,PQ}\xi^P\eta^Q.
$$
By the non-resonance condition, it is straightforward that if $\psi \hat S=\hat S\psi$, then $a_{j,PQ}=b_{j,QP}=0$
if $P-Q\neq e_j$.
Note that $\hat S^{-1}=T_0\hat ST_0$ for $T_0\colon(\xi,\eta)\to(\eta,\xi)$. Thus $\tau_1T_0$ commutes with
$\hat S$. So $\tau_1T_0$ has the form \re{pajb} in which we rename $a_j,b_j$ by $\Lambda_{1j},\tilde\Lambda_{1j}$, respectively. Now $\tau_1^2=\id$ implies that
$$
\Lambda_{1j}((\Lambda_{11}\tilde\Lambda_{11})(\zeta)\zeta_1,\ldots, (\Lambda_{1p}\tilde\Lambda_{1p})(\zeta)\zeta_p)\tilde\Lambda_{1j}(\zeta)=1, \quad 1\leq j\leq p.
$$
Then $\Lambda_{1j}(0)\tilde\Lambda_{1j}(0)=1$. Applying induction on $d$, we verify that for all $j$
 $$\Lambda_{1j}(\zeta)\tilde\Lambda_{1j}(\zeta)=1+O(|\zeta|^d), \quad d>1.
 $$
 Having found the formula for
   $\tau_1T_0$, we obtain the desired formula of $\tau_1$ via composition $(\tau_1T_0)T_0$.
\end{proof}

Let ${\mathbf D}_1:=\text{diag}(\mu_{11},\ldots,\mu_{1n}),\ldots, {\mathbf D}_\ell:=\text{diag}(\mu_{\ell 1},\ldots,\mu_{\ell n})$ be diagonal invertible matrices of $\cc^n$. Let us set $D:=\{{\mathbf D}_iz\}_{i=1,\ldots \ell}$.
\begin{defn}\label{ccst}  Let $F$ be a formal mapping of $\cc^n$ that is tangent to the identity.
\bppp
\item Let $n=2p$.   $F$   is  {\it normalized} with respect to $\hat S$,  if $F=(f,g)$ is tangent to the identity
 and $F$ contains no resonant terms, i.e.
 \eq{fj0g}
 \nonumber
 f_{j,(A+e_j)A}=0=g_{j,A(A+e_j)},  \quad  |A|>1. 
 \eeq
 \item Let $n=2p$.
$F$ is {\it normalized} with respect to $\{\hat T_1,\hat T_2\}$, if $F=(f,g)$ is tangent to
the identity and
\eq{fjmg}
\nonumber
 f_{j,(A+e_j)A}=-g_{j,A(A+e_j)},\quad |A|>1.
\eeq
\item
$F$ is {\it normalized} with respect to $D$ if it does not have components along the centralizer of $D$, i.e.   for each
$Q$ with $|Q|\geq2$,
\eq{norm-D}
\nonumber
f_{ j,Q}=0,\quad\text{if}\; \mu_i^Q=\mu_{ij}\; \text{for all}\; i. 
\eeq
\eppp
Let ${\cL C}^{\mathsf{ c}}(\hat S)$ (resp. ${\cL C}^{\mathsf{ c}}(\hat T_1,\hat T_2)$, ${\cL C}^{\mathsf{ c}}(D)$) denote the set of formal mappings  normalized with respect to $\hat S$ (resp. $\{\hat T_1,\hat T_2\}$, the family $D$).
 For convenience, we
let ${\cL C}^{\mathsf{ c}}_2(\hat S)$ (resp. ${\cL C}^{\mathsf{ c}}_2(\hat T_1,\hat T_2)$, ${\cL C}^{\mathsf{ c}}_2(D)$) denote the set of formal mappings  $F-\I$ with
$F\in {\cL C}^{\mathsf{ c}}(\hat S)$ (resp. ${\cL C}^{\mathsf{ c}}(\hat T_1,\hat T_2)$, ${\cL C}^{\mathsf{ c}}(D)$).
 \end{defn}

Recall that for $j=1,\ldots, p$, we define
$$
Z_j\colon\xi'=\xi, \quad \eta_k'=\eta_k, \ k\neq j, \quad \eta_j'=-\eta_j.
$$
We have seen in section~\ref{secquad} how invariant functions of $Z_j$ play a role in constructing normal form of quadrics.
In section~\ref{nfin}, we will also need
a centralizer for non linear maps (see~\rl{cnnl}) to obtain normal forms for two families of involutions. Therefore, let us
first recall   the following lemma
on the centralizer of $Z_1,\ldots, Z_p$, which is a special case of \cite[Lemma 4.7]{GS15}.
\begin{lemma}\label{lehphi}The  centralizer, ${\cL C}(Z_1,\ldots, Z_p)$,  consists of  formal  mappings $$(\xi,\eta)\to (U(\xi,\eta),  \eta_1V_1(\xi,\eta),\dots, \eta_pV_p(\xi,\eta))$$ such that $U(\xi,\eta), V(\xi,\eta)$
 are even in each $\eta_j$. Let
 $
 {\cL C}^{\mathsf{ c}}(Z_1,\ldots, Z_p) $
 denote the set of mappings $I+ (U, V)$
 which are tangent to the identity such that
 \eq{vjpej}\nonumber
 U_{j,PQ}=  V_{j,P(e_j+Q')}=0, \quad Q, Q'\in 2\nn^p, \
|P|+|Q|>1, \  |P|+|Q'|>1. \eeq
 Let $\psi $  be   a mapping that is tangent to the identity.
  There exist  unique
$\psi_0\in {\cL C}(Z_1,\ldots, Z_p)$
and $\psi_1\in{\cL C}^{\mathsf{ c}}(Z_1,\ldots, Z_p)$ such that
$ 
 \psi=\psi_1\psi_0^{-1}.
$ 
Moreover, if $\psi$ is convergent, then $\psi_0$ and $\psi_1$ are convergent.
\end{lemma}
 Analogously, for any formal  mapping $\psi$ that is tangent to the identity, there is a unique decomposition $\psi=\psi_1\psi_0^{-1}$ with $\psi_1\in {\cL C}^{\mathsf{ c}}(\hat S)$
and $\psi_0\in  {\cL C}(\hat S)$. If $\psi$ is convergent, then  $\psi_0,\psi_1$ are convergent.
  Let $F=(F_1,\dots, F_n)\colon\cc^n\to \cc^n$ be a formal mapping. Define a formal mapping $F_{sym}\colon\cc^n\to\cc^n$ by
$$
(F_{sym})_{i,P}=\max_{1\leq j\leq n,\nu\in S_n}|\{g_j\circ\nu\}_P|,
$$
where $S_n$ is the set of permutations $\nu$ of coordinates $z_i\to z_{\nu(i)}$.
Let us recall the following lemma from \cite[Lemma 4.3]{GS15}. \begin{lemma}\label{fhg-}  Let  $\hat {\cL H}$ be  a real subspace of $({\widehat { \mathfrak M}}_n^2)^n$.
Let $\pi : ({\widehat { \mathfrak M}_n^2)^n}\rightarrow \hat  {\cL H}$ be
a $\rr$ linear projection $($i.e. $\pi^2=\pi)$
 that preserves the degrees of the mappings and let $\hat  {\cL G}:= (\I-\pi)({\widehat { \mathfrak M}}_n^2)^n$.
Suppose that there is a positive constant $C$ such that
$ 
\pi(E)\prec  CE_{sym}
$ 
for any $E\in (\widehat { \mathfrak M}_n^2)^n$.
 Let $F$ be a formal map tangent to the identity.
There exists a unique decomposition
$ 
F=HG^{-1}
$ 
with $G-I\in\hat  {\cL G}$ and $H-I\in \hat  {\cL H}$.
  If $F$ is convergent, then $G$ and $H$ are also convergent.
\end{lemma}
 \begin{lemma}\label{fhg}
  Let $\psi $  be   a mapping that is tangent to the identity.
  There exist  unique
$\psi_0\in \cL C(\hat T_1,\hat T_2)$
and $\psi_1\in \cL C^{\mathsf{ c}}(\hat T_1,\hat T_2)$
 such that
$ 
 \psi=\psi_1\psi_0^{-1}.
$ 
Moreover, if $\psi$ is convergent, then $\psi_0$ and $\psi_1$ are convergent. \end{lemma}
\begin{proof}
Let  $\hat  {\cL G}= {\cL C}_2(\hat T_1,\hat T_2)$ and $\hat  {\cL H}= {\cL C}^{\mathsf{ c}}_2(\hat T_1,\hat T_2)$. We need to find a $\rr$-linear projection such that $\hat  {\cL H}=\pi({\widehat { \mathfrak M}}_n^2)^n$,  $ \hat  {\cL G}=(\id-\pi)({\widehat { \mathfrak M}}_n^2)^n$, and
$
\pi(E)\prec CE_{sym}.
$
Note that  $g\in\cL C_2(\hat T_1,\hat T_2)$ and $h\in \cL C^{\mathsf{ c}}_2(\hat T_1,\hat T_2)$ are determined by conditions
\gan
g_{j,(\gamma+e_j)\gamma}=g_{(j+p),\gamma(\gamma+e_j)}, \quad
h_{j,(\gamma+e_j)\gamma}=-h_{(j+p),\gamma(\gamma+e_j)}, \quad 1\leq j\leq p,
\\
g_{j,PQ}=g_{(j+p),QP}=0,\quad P-Q\neq e_j.
\end{gather*}
Thus, if $h-g=K$, we determine $g$ uniquely by combining the above identities with
\aln
g_{j,(\gamma+e_j)\gamma}&=\frac{-1}{2}\left\{K_{j,(\gamma+e_j)\gamma}
+K_{(j+p),\gamma(\gamma+e_j)}\right\},\\ 
h_{j,(\gamma+e_j)\gamma}&=
\frac{1}{2}\left\{K_{j,(\gamma+e_j)\gamma}-
K_{(j+p),\gamma(\gamma+e_j)}\right\} 
\end{align*}
for $1\leq j\leq p$.
For the remaining coefficients of $h$, set $h_{i,PQ}=K_{i,PQ}$.
 Therefore, $\pi(K):=h\prec K_{sym}$  and the lemma follows form  \rl{fhg-}.
\end{proof}

\setcounter{thm}{0}\setcounter{equation}{0}

\setcounter{thm}{0}\setcounter{equation}{0}
\section{Formal normal forms of the reversible map $\sigma$ 
}\label{secnfs}

Let us first describe our plans to derive the normal forms of $M$.
We would like
to show that two  families of involutions $\{\tau_{1j},\tau_{2j},\rho\}$ and
$\{\tilde\tau_{1j},\tilde\tau_{2j},\tilde\rho\}$ are holomorphically equivalent, if their corresponding
normal forms are equivalent under a much smaller set of changes of coordinates.
Ideally, we would like to conclude that  $\{\tilde\tau_{1j},\tilde\tau_{2j},\tilde\rho\}$ are holomorphically equivalent
if and only if their corresponding normal forms are the same, or if they are the same under a change of coordinates with finitely many parameters. For instance the Moser-Webster normal form for real analytic surfaces ($p=1$) with non-vanishing elliptic Bishop invariant falls into
  the former situation, while the Chern-Moser theory \cite{chern-moser} for real analytic hypersurfaces with non-degenerate Levi-form is an example for the latter.
Such a normal form will tell us if
 the real manifolds have infinitely many invariants or not. One of our goals is to understand if the normal form so achieved can
 be realized by a convergent normalizing transformation.
We will see soon 
that we can achieve our last goal  under  some assumptions on the  family of involutions. Alternatively and perhaps for simplicity of the normal form theory, we would
like to seek normal forms which are dynamically or geometrically significant.

Recall that for each real analytic manifold that
has  $2^p$, the maximum number of,  commuting
deck transformations $\{\tau_{1j}\}$,
we have found
 a unique set of generators $\tau_{11},\ldots, \tau_{1p}$  so that each $\fix(\tau_{1j})$
has codimension $1$. More importantly $\tau_1=\tau_{11}\cdots\tau_{1p}$ is the unique deck
transformation of which the set of fixed points  
has dimension $p$. Let $\tau_2=\rho\tau_1\rho$ and $\sigma=\tau_1\tau_2$.
To normalize $\{\tau_{1j}, \tau_{2j},\rho\}$, we will choose $\rho$ to be the standard anti-holomorphic
involution determined by the linear parts of $\sigma$. Then we
normalize $\sigma=\tau_1\tau_2$ under formal mapping commuting with $\rho$. This will determine a 
normal form for $\{\tau_1^*,\tau_2^*,\rho\}$. This part of normalization is analogous to the Moser-Webster normalization.
  When $p=1$, Moser and Webster obtained a unique normal form by a simple argument.
However, this last step of simple normalization is not available when $p>1$.   By assuming
$\log \hat M$ associated to  $\hat\sigma$
 is tangent to the identity,
  we
will obtain a unique formal normal form $\hat\sigma, \hat\tau_1,\hat\tau_2$  for $\sigma,\tau_1,\tau_2$.
Next, we need to construct the normal form for the families of involutions. We first ignore the reality condition, by
finding $\Phi$ which transforms $\{\tau_{1j}\}$ into a set of
involutions  $\{\hat \tau_{1j}\}$ which is decomposed canonically according to
$\hat\tau_1$.  This allows us to express $\{\tau_{11},\ldots, \tau_{1p},\rho\}$
via $\{\hat\tau_1,\hat\tau_2,\Phi,\rho\}$,   as  in the classification of the families of linear involutions.
Finally, we further normalize $\{\hat\tau_1,\hat\tau_2,\Phi,\rho\}$ to get our normal form.

\

\begin{defn}\label{notation}
 Throughout this section and next, we denote $\{h\}_{d}$ the set of coefficients
of $h_P$ with $|P|\leq d$ if $h(x)$ is a map or function in $x$ as power series. We
  denote by $\cL A_{P}(t), \cL A(y;t)$, etc.,   a universal {\it polynomial} whose
coefficients and degree depend  on a multiindex.
The variables in these polynomials will involve a collection of Taylor coefficients of various
mappings. The collection will also depend on $|P|$.
As such dependency (or independency to  coefficients of higher degrees) is crucial to our computation,
we will remind the reader the dependency when  emphasis is necessary.
\end{defn}

For instance, let us take two formal mappings $F,G$ from $\cc^n$ into itself. Suppose
that  $F=\id+f$ with $f(x)=O(|x|^2)$ and $ G= LG+g$
with $g(x)=O(|x|^2)$ and $LG$ being linear. For $P\in\nn^n$ with $ |P|>1$, we can express
\ga\label{F-1P5}
(F^{-1})_{P}=-f_{P}+\cL F_{P}(\{f\}_{|P|-1}), \\
 (G\circ F)_{P}=g_P+((LG)\circ f)_P+\cL G_{P}(LG;\{f,g\}_{|P|-1}),
\label{GFPg}\\
\label{F-1GF}
 (F^{-1}\circ G\circ F)_{P}=g_P-(f\circ (LG))_P+((LG)\circ f)_P+\cL H_{P}(LG;\{f,g\}_{|P|-1}).
\end{gather}

\subsection{
Formal normal forms of  pair of involutions $\{\tau_1,\tau_2\}$}

We first find a normal form for $\sigma$ in $\cL C(S)$.

\begin{prop}\label{ideal0} Let $\sigma$ be a holomorphic map. Suppose that $\sigma$ has a non-resonant  linear part $$
\hat S\colon\xi_j'=\mu_j\xi_j, \quad \eta_j=\mu_j^{-1}\eta_j, \quad 1\leq j\leq p.
$$
Then there exists a unique normalized formal map $\Psi\in {\cL C}^{\mathsf{ c}}(\hat S)$ such that $\sigma^*=\Psi^{-1}\sigma\Psi\in{\cL C}(\hat S)$.
Moreover,  $\tilde \sigma=\psi_0^{-1}\sigma^*\psi_0\in\cL C(\hat S)$,
if and only if $\psi_0\in\cL C(\hat S)$  and it is invertible. Let
\gan
\sigma^*\colon\xi_j'=M_j(\xi\eta)\xi_j, \quad \eta_j'=N_j(\xi\eta)\eta_j,\\
\tilde \sigma\colon\xi_j'=\tilde  M_j(\xi\eta)\xi_j, \quad \eta_j'=\tilde  N_j(\xi\eta)\eta_j,\\
\psi_0\colon\xi_j'= a_j(\xi\eta)\xi_j,\quad \eta_j'=  b_j(\xi\eta)\eta_j. 
\end{gather*}
\bppp
\item
Assume that    $\tau_1,\tau_2$
are holomorphic involutions and $\sigma=\tau_1\tau_2$.  Then $\sigma^*=\tau_1^*\tau_2^*$ with
\ga\label{tauis}
\tau_{i}^*=\Psi^{-1}\tau_i\Psi\colon\xi_j'=\Lambda_{ij}(\xi\eta)\eta_j, \quad \eta_j'=\Lambda_{ij}^{-1}(\xi\eta)\xi_j;\\
N_j=M_j^{-1}, \quad  M_j=\Lambda_{1j}\Lambda_{2j}^{-1}.
\nonumber
\end{gather}
Let the linear part of $\tau_i$ be given by
$$
\hat T_i\colon\xi_j'=\lambda_{ij}\eta_j,\quad\eta_j'=\lambda_{ij}^{-1}\xi_j.
$$
Suppose that $\lambda_{2j}^{-1}=\lambda_{1j}$.
There exists a unique $\psi_0\in   {\cL C}^{\mathsf{ c}}(\hat T_1,\hat T_2)$ such that
\ga \nonumber
\tilde\tau_{i}=\psi_0^{-1}\tau_i^*\psi_0\colon\xi_j'=\tilde\Lambda_{ij}(\xi\eta)\eta_j, \quad \eta_j'=\tilde
\Lambda_{ij}^{-1}(\xi\eta)\xi_j;\\
\label{tl21}
\tilde  M_{j}=\tilde \Lambda_{1j}^2  =\tilde N_j^{-1}, \quad \tilde \Lambda_{2j}=\tilde \Lambda_{1j}^{-1}.
\end{gather}
 Let
 $\psi_1$ be a formal biholomorphic map. Then  $\{\psi_1^{-1}\tilde\tau_1\psi_1,
 \psi_1^{-1}\tilde\tau_2\psi\}$ has the same form as of $\{\tilde \tau_{1},\tilde \tau_{2}\}$
    if and only
if $\psi_1\in{\cL C}(\hat T_1,\hat T_2)$;   moreover, $\tilde\Lambda_{ij}(\xi\eta)$, $\tilde M_j(\xi\eta)$ are transformed
into
\eq{lamtpsi}
\tilde\Lambda_{ij}\circ\tilde\psi_1, \quad \tilde M_j\circ\tilde\psi_1.
\eeq
Here $\tilde\psi_1(\zeta)=(\diag c(\zeta))^2\zeta$ and
$\psi_1(\xi,\eta)=((\diag c(\xi\eta))\xi,(\diag c(\xi\eta))\eta)$.
\item
Assume further that $\tau_2=\rho\tau_1\rho$, where    $\rho$
is defined by \rea{eqrh}. Let
\eq{rhze}
\nonumber
{\rho_z}\colon\zeta_j\to\ov\zeta_j, \quad 1\leq j\leq e_*+h_*;
\quad \zeta_{s}\to\ov\zeta_{s+s_*}, \quad e_*+h_*<s\leq p-s_*.
\eeq
 Then $\rho\Psi=\Psi\rho$,   $\tau_2^*=\rho\tau_1^*\rho$,  and $(\sigma^*)^{-1}=\rho\sigma^*\rho$. The
  last two identities are equivalent to
\begin{alignat}{5} \label{a2e-}
&\Lambda_{2e}^{-1}
&&=\ov{\Lambda_{1e}\circ{\rho_z}}, \quad &&\ov {M_e\circ{\rho_z}}=M_e,
\quad &&1\leq e\leq e_*;\\
&\Lambda_{2h}&&=\ov{\Lambda_{1h}\circ{\rho_z}},\quad&&
\ov {M_h\circ{\rho_z}}=M_h^{-1},\quad&& e_*< h\leq h_*+e_*;\\
  & \Lambda_{2(s)}&&=\ov{\Lambda_{1(s_*+s)}\circ{\rho_z}}, &&
\\
& \Lambda_{2(s_*+s)}&&=\ov{\Lambda_{1s}\circ{\rho_z}},\quad&&
\ov {M_s^{-1}\circ{\rho_z}}=M_{s_*+s}, \quad && h_*+e_*< s\leq p- s_*.
\label{aiss}\end{alignat}
Let   $\psi_0$ and $\tilde\tau_i=\psi_0^{-1}\tau_i^*\psi_0$ be as in (i). Then $\rho\psi_0=\psi_0\rho$, and $\hat\tau_1,\hat\tau_2$ satisfy
\ga\label{ohai}
\tilde\Lambda_{ie}=\ov{\tilde\Lambda_{ie}\circ{\rho_z}}, \quad \tilde\Lambda_{ih}^{-1}=\ov{\tilde\Lambda_{ih}\circ{\rho_z}},\quad
\tilde\Lambda_{i{s+s_*}}=\ov{\tilde\Lambda_{is}^{-1}\circ{\rho_z}}.
\end{gather}
\eppp
\end{prop}
\begin{proof} We will use the Taylor formula
$$
f(x+y)=f(x)+\sum_{k=1}^m\frac{1}{k!}D_kf(x;y)+R_{m+1}f(x;y)
$$
with $D_kf(x;y)=\{\pd_t^kf(x+ty)\}|_{t=0}$ and
\eq{taylor}
 R_{m+1}f(x;y)=(m+1)\int_0^1(1-t)^m\sum_{|\alpha|=m+1}\frac{1}{\alpha!}\partial^{\alpha}f(x+ty)y^{\alpha}\, dt.
\eeq
Set $D=D_1$.
Let $\sigma$ be given by
\eq{idfg-}
\nonumber
\xi_j'=M_j^0(\xi\eta)\xi_j+f_j(\xi,\eta), \quad \eta_j'=N_j^0(\xi\eta)\eta_j+g_j(\xi,\eta)
\eeq
with
\eq{idfg}
 (f,g)\in  {\cL C}^{\mathsf{ c}}_2(\hat S),
 \quad \ord(f,g)= d \geq2.
\eeq

We need to find    $\Phi\in  {\cL C}^{\mathsf{ c}}(S)$
such that $\Psi^{-1}\sigma\Psi=\sigma^*$ is given by
$$
\xi_j'=M_j(\xi\eta)\xi_j, \quad\eta_j'=N_j(\xi\eta)\eta_j.
$$
By definition, $\Psi$ has the form
$$
\xi'_j=\xi_j+U_j(\xi,\eta), \quad\eta_j'= \eta_j+V_j(\xi,\eta), \quad U_{j,(P+e_j)P}=V_{j,P(P+e_j)}=0.
$$
The components of $\Psi\sigma^*$  are
\al\label{xjmj}
\xi_j'&=M_j(\xi\eta)\xi_j+U_j(M(\xi\eta)\xi,N(\xi\eta)\eta),\\
\eta_j'&=N_j(\xi\eta)\eta_j+V_j(M(\xi\eta)\xi,
N(\xi\eta)\eta).
\label{ejmj}\end{align}
To derive the normal form, we only need Taylor theorem in order one. This can also demonstrate  small divisors in the normalizing transformation; however, one cannot see the small divisors in the normal forms.  Later we will  show the existence of divergent normal forms. This requires us to use Taylor formula whose remainder
has  order two.
By the Taylor theorem,
we write the components of $\sigma\Psi$ as
\al\label{mj0x}
\xi_j'&=(M^0_j(\xi\eta)+DM_j^0(\xi\eta)(\eta U +\xi V  +UV))(\xi_j+U_j)\\
\nonumber
&\quad +f_j(\xi,\eta)+ Df_j(\xi,\eta)(U, V)+A_j(\xi,\eta),\\
\label{nj0x}\eta_j'&=(N_j^0(\xi\eta)+DN_j^0(\xi\eta)(\eta U +\xi V+UV))(\eta_j+V_j)\\
\nonumber
&\quad +g_j(\xi,\eta)+Dg_j(\xi,\eta)(U,V)+B_j(\xi,\eta).
\end{align}
Recall our notation that
 $UV=(U_1(\xi,\eta)V_1(\xi,\eta),\ldots, U_p(\xi,\eta)V_p(\xi,\eta))$. The second order remainders are
\al \label{ajpq}
A_j(\xi,\eta)&=
 R_2M_j^0(\xi\eta;\xi U+\eta V+UV)(\xi_j+U_j) +R_2f_j(\xi,\eta;U,V),\\
  B_j(\xi,\eta)&=
 R_2N_j^0(\xi\eta;\xi U+\eta V+UV)(\eta_j+V_j) +R_2g_j(\xi,\eta;U,V).\label{bjpq--}
 \end{align}
Note that the remainder $R_2M^0$ is independent of the linear part of $M^0$. Thus
\gan
R_2M_j^0=R_2(M_j^0-LM_j^0), \quad R_2N_j^0=R_2(N_j^0-LN_j^0).
\end{gather*}
Let us calculate the largest degrees   $w,d'$ of coefficients of $M^0-LM^0, (U,V,f,g)$ on which $A_{j,PQ}$ depend.
  It is easy to see that $d'\geq d\geq2$  and $w\geq2$.
We have
\gan
 2(w- 2)+  2(d+1)+1\leq |P|+|Q|; \\
 3+d+d'\leq |P|+|Q| \quad \text {or} \quad  2d+d'-2\leq |P|+|Q|,
\end{gather*}
where the first two inequalities are obtained from the first term on the right-hand side of \re{ajpq} and its second
term yields the last inequality.
Thus, we have  crude bounds
$$
w\leq\f{|P|+|Q|+  1-2d}{2}, \quad d'\leq   |P|+|Q|-d.
$$
Analogously, we can estimate the degrees of coefficients of $N^0$. We obtain
\al
\label{ajpq+}
A_{j,PQ}&=\cL A _{j,PQ}( \{M^0-LM^0\}_{\f{|P|+|Q|+1-2d}{2}};\{f, U, V\}_{|P|+|Q|-d}),\\
B_{j, QP}&=\cL B _{j,QP}(\{N^0-LN^0\}_{\f{|P|+|Q|+1-2d}{2}}; \{g, U, V\}_{|P|+|Q|-d}).
\end{align}
  Recall our notation that $\{f,U,V\}_d$ is the set of coefficients of $f_{PQ}, U_{PQ}, V_{PQ}$
with $|P|+|Q|\leq d$.  Here
 $\cL A_{j,PQ}(t';t''),\cL B_{j,QP}(t';t'')$ are polynomials  of which each has  coefficients that depend only on $j,P,Q$
 and they vanish at $t''=0$.

To finish the proof of the proposition, we will not need the explicit expressions
involving $DM_j^0, DN_j^0$, $Df_j$, $Dg_j$. We will use these derivatives
in the proof of \rl{mj0mj}. So we derive these expression in this proof too.

We apply the projection \re{xjmj}-\re{ejmj} and \re{mj0x}-\re{nj0x} onto $\cL C^{\mathsf c}_{2}(S)$,  via monomials in each
component of both sides of the identities.
  The images of the mappings
\begin{align*}(\xi,\eta)&\mapsto (U(M(\xi\eta)\xi,N(\xi\eta)\eta)
,V(M(\xi\eta)\xi,N(\xi\eta)\eta)),\\
(\xi,\eta)
&\mapsto (M^0(\xi\eta)U(\xi,\eta), N^0(\xi\eta)V(\xi,\eta))\end{align*}
under the projection are $0$.
We obtain from \re{xjmj}-\re{nj0x}   and \re{ajpq}-\re{bjpq--} that $d_0=d$.
Next, we project \re{xjmj}-\re{ejmj} and \re{mj0x}-\re{nj0x} onto $\cL C_{2}(\hat S)$,  via monomials in each
component of both sides of the identities.
  Using \re{idfg} and \re{ajpq+} we obtain
\al
\label{mpmp-}
M_{j,P}&=M^0_{j,P}+\{Df_j(U,V)\}_{P+e_j,P}+\cL M_P(\{M^0\}_{\f{2|P|+1-2d}{2}}; \{f, U, V\}_{  P(d)}), \\
N_{j,P}&=N^0_{j,P}+\{Dg_j(U,V)\}_{P,P+e_j}+\cL N_P(\{N^0\}_{\f{2|P|+1-2d}{2}}; \{g, U, V\}_{P(d)})\label{npnp}
\end{align}
with
$$ 
 P(d)= 2|P|+1-d.
$$ 
  Here $\cL M_{P}, \cL N_{P}$ are polynomials  of which each has  coefficients  that depend only on $P$, and $\{M^0\}_a$
stands for the set of coefficients $M^0_{Q}$ with $|Q|\leq a$ for a real number $a\geq0$.
Note that $\cL U _{j,PQ}=\cL V _{j, QP}=0$ when $|P|+|Q|=2$, or $\ord (f,g)> |P|+|Q|$. And $\cL M_P=\cL N_P=0$ when $\ord (f,g)>  P(d)$, by \re{idfg}. We have
\al\nonumber
\{U_j(M(\xi\eta)\xi,N(\xi\eta)\eta)\}_{PQ}=\mu^{P-Q}U_{j,PQ}+\cL U_{j,PQ}(\{M,N\}_{\f{|P|+|Q|-d}{2}},\{U\}_{|P|+|Q|-2}).
\end{align}
Comparing coefficients in \re{xjmj}, \re{mj0x}, and using \re{ajpq+}, we get for $\ell=|P|+|Q|$
\aln
(\mu^{P-Q}-\mu_j)
U _{j,PQ} &=\{f _j+Df_j(U,V)\}_{PQ}\\
&\quad +\cL U _{j,PQ}( \{M^0\}_{\f{\ell+1-2d}{2}},\{M,N\}_{\f{\ell-d}{2}}; \{f, U, V\}_{\ell-2}).
\end{align*}
We have analogous formula for $V_{j,QP}$.
Using \re{mpmp-}, we obtain with $|P|+|Q|=\ell$
\al
\label{mp-q}
(\mu^{P-Q}-\mu_j)
U _{j,PQ} &=\{f _{j}+Df_j(U,V)\}_{PQ}+\cL U _{j,PQ}(  \{M^0, N^0\}_{\f{\ell-d}{2}}; \{f, g,U, V\}_{\ell-2}),\\
(\mu^{Q-P}-\mu_j^{-1})
V _{j,QP}&=\{g _{j}+Dg_j(U,V)\}_{PQ}+\cL V _{j,QP}(\{M^0,N^0\}_{\f{\ell -d}{2}}; \{f,g, U, V\}_{\ell-2}).
\label{mp-q2}\end{align}
for $\mu^{P-Q}\neq\mu_j$, which are always solvable.
 Inductively, by using
\re{mp-q}-\re{mp-q2} and \re{mpmp-}-\re{npnp}, we obtain unique solutions $U,V, M, N$. Moreover, the solutions and their dependence on the coefficients of $f,g$ and small divisors have the form
\al
(\mu^{P-Q}-\mu_j)U _{j,PQ}&=\{f _{j}+Df_j(U,V)\}_{PQ}+\cL U _{j,PQ}^*(\del_{\ell-2},
\{M^0,N^0\}_{\f{\ell-d}{2}};\{ f,g\}_{\ell-2}),
\label{ujpq+}
\\
(\mu^{Q-P}-\mu_j^{-1})V _{j,QP}&=\{g _{j}+Dg_j(U,V)\}_{PQ}+\cL V _{j,QP}^*(\del_{\ell-2},
\{M^0,N^0\}_{\f{\ell-d}{2}};\{ f,g\}_{\ell-2}).
\label{vjpq}
\end{align}
where $\ell=|P|+|Q|$ and $\mu^{P-Q}\neq\mu_j$, and
$
\del_{i}
$ is the union of $\{\mu_1,\mu_1^{-1},\ldots,
\mu_p,\mu_p^{-1}\}$ and
$$
\left\{\frac{1}{\mu^{A-B}-\mu_j}\colon |A|+|B|\leq i,j=1,\dots,p,
  A,B\in \nn^p\right\}.
$$
This shows that for any $M^0, N^0$ there exists  a unique mapping $\Psi$ transforms $\sigma$ into $\sigma^*$.
  Furthermore,  $\cL U^*_{j,PQ}(t';t''), \cL V^*_{j,QP}(t';t'')$ are polynomials  of which each has  coefficients that depend only on $j,P,Q$, and they vanish at $t''=0$.

  For later purpose,  let us express $M,N$ in terms of $f,g$. We substitute expressions \re{ujpq+}-\re{vjpq}
for $U,V$ in \re{mpmp-}-\re{npnp} to obtain
\al
M_{j,P}&=M^0_{j,P}+\{Df_j(U,V)\}_{P+e_jP}+\cL M_{j,P}^*(\del_{P(d)},\{M^0,N^0\}_{\f{P(d)}{2}}; \{f, g\}_{P(d)}),  \label{mpmpnsd}
 \\
N_{j,P}&=N^0_{j,P}+\{Dg_j(U,V)\}_{PP+e_j}+\cL N_{j,P}^*(
\del_{P(d)},\{M^0,N^0\}_{\f{ P(d)}{2}}; \{f, g\}_{P(d)}).\label{npnpnsd}
\end{align}
with $f,g$ satisfying \re{idfg}.

Assume that   $\tilde\sigma=\psi_0^{-1}\sigma^*\psi_0$ commutes with $\hat S$. By \nrc{fhg},
we can decompose  
$\psi_0=HG^{-1}$ with $G\in\cL C(\hat S)$ and $H\in\cL C^{\mathsf{ c}}(\hat S)$. Furthermore,
 $G^{-1}\tilde \sigma G$ commutes with $\hat S$ and
  $H^{-1}\sigma^* H$. By the uniqueness conclusion for the above $\psi_0$,   $H$ must be the identity.  This shows that $\psi_0\in\cL C(\hat S)$.

(i). Assume that we have normalized $\sigma$. We now use it to normalize the pair of involutions.   Assume that $\sigma=\tau_1\tau_2$ and $\tau_j^2=I$. Then $\sigma^*=\tau_1^*\tau_2^*$.
Let $T_0(\xi,\eta):=(\eta,\xi)$.
We have 
 $T_0(\sigma^*)^{-1}T_0=T_0\tau_1^*\sigma^*\tau_1^*T_0$.
 By the above normalization,  $T_0(\sigma^*)^{-1}T_0$ commutes with $\hat S$. Therefore, $\tau_1^*T_0$ belongs to the centralizer of $\hat S$ and  it must be of the form
 $(\xi,\eta)\to(\xi \Lambda_1(\xi\eta),\eta  \Lambda_1^*(\xi\eta))$. Then $(\tau_1^*)^2=I$ implies that
 $$
 \Lambda_1(\xi\eta (\Lambda_1 \Lambda_1^*)(\xi\eta)) \Lambda_1^*(\xi\eta)=1.
 $$
  The latter implies, by induction on $d>1$, that $\Lambda_1 \Lambda_1^*=1+O(d)$ for all $d>1$, i.e. $\Lambda_1 \Lambda_1^*=1$.

Let $\tau_i^*$ be given by \re{tauis}. We want to achieve
$\tilde\Lambda_{1j}\tilde\Lambda_{2j}=1$ for $\tilde\tau_i=\psi_0^{-1}\tau_i^*\psi_0$
by applying
 a transformation $\psi_0$ in $\cL C^{\mathsf{ c}}(\hat T_1,\hat T_2)$ that commutes with $\hat S$. According to Definition \ref{ccst}, it has the form
 $$
 \psi_0\colon \xi_j=\tilde\xi_j(1+a_j(\tilde\zeta)), \quad
 \eta_j=\tilde\eta_j (1-a_j(\tilde\zeta))
  $$
with $a_j(0)=0$. Here $\tilde \zeta_j:=\tilde\xi_j\tilde \eta_j$ and
 $\tilde \zeta:=(\tilde \zeta_1,\ldots, \tilde \zeta_p)$. Computing the products $\zeta$ in $\tilde\zeta$ and solving $\tilde\zeta$ in $\zeta$, we obtain
$$
\psi_0^{-1}\colon \tilde\xi_j=\xi_j(1+b_j(\zeta))^{-1}, \quad
 \tilde\eta_j=\eta_j (1-b_j(\zeta))^{-1}.
$$
Note that $(a_j^2)_P=\cL A _{j,P}(\{a\}_{|P|-1})$, and
$$\xi_j\eta_j=\tilde\xi_j\tilde\eta_j(1-a_j^2(\tilde\zeta)),\quad
\tilde\xi_j\tilde\eta_j=\xi_j\eta_j(1-b_j^2(\zeta))^{-1}.$$
From $\psi_0^{-1}\psi_0=I$, we get
\ga
b_j(\zeta)=a_j(\tilde\zeta), \quad
\label{bjpq}
b _{j,P}=a _{j,P}+\cL B _{j,P}(\{a\}_{|P|-1}).
 \end{gather}
By a simple computation we see that $\tilde\tau_i=\psi_0^{-1}\tau_i^*\psi_0$ is given by
$$
\tilde\xi_j'=\tilde\eta_j\tilde\Lambda_{ij}(\tilde\zeta),
\quad\tilde\eta_j'=\tilde\xi_j\tilde\Lambda_{ij}^{-1}(\tilde\zeta)
$$
with
$$
\tilde\Lambda_{1j}\tilde\Lambda_{2j}(\tilde\zeta)=(\Lambda_{1j}\Lambda_{2j})(\zeta)
(1+b_j(\zeta'))^{-2}(1- a_j(\tilde\zeta))^{2}.
$$
Here $\zeta_j'=\zeta_j(1-a_j^2(\tilde\zeta))$.
 Using \re{bjpq} and 
 the implicit function theorem, we determine $a_j$ uniquely to achieve $\tilde\Lambda_{1j}\tilde\Lambda_{2j}=1$.

To identify the transformations that preserve the form of $\tilde\tau_1,\tilde\tau_2$,
 we first verify that each element $\psi_1\in\cL C(\hat T_1,\hat T_2)$ preserves that form.
According to \re{bjaj}, we have
\gan
\psi_1\colon\xi_j=\tilde\xi_j  \tilde a_j(\tilde\zeta), \quad\eta_j=\tilde\eta_j\tilde a_j(\tilde\zeta),\\
\psi_1^{-1}\colon \tilde\xi_j= \xi_j\tilde b_j( \zeta), \quad\tilde\eta_j= \eta_j\tilde b_j( \zeta),\\
\tilde b_j(\zeta)\tilde a_j(\tilde\zeta)=1.
\end{gather*}
This shows that $\psi_1^{-1}\tilde\tau_i$ is given by  
$$
\tilde\xi_j'=\tilde\Lambda_{ij}(\zeta)\tilde b_j(\zeta)\eta_j,\quad \tilde\eta_j'=\tilde\Lambda_{ij}^{-1}(\zeta)\tilde b_j(\zeta)\xi_j.
$$
Then $\psi_1^{-1}\tilde\tau_i\psi_1$ is given by
$$
\tilde\xi_j'=\tilde\Lambda_{ij}(\zeta)\tilde\eta_j,\quad \tilde\eta_j'=\tilde\Lambda_{ij}^{-1}(\zeta)\tilde\xi_j.
$$
Since $\zeta_j=\tilde\zeta_j\tilde a_j^2(\tilde\zeta)$, then $\psi_1^{-1}\tilde\tau_i\psi_1$
still satisfy \re{tl21}.
Conversely,
suppose  that  $\psi_1$ preserves the forms of $\tilde\tau_1,\tilde\tau_2$.
We apply \nrc{fhg} to decompose  
$\psi_1=\phi_1\phi_0^{-1}$ with $\phi_0\in\cL C(\hat T_1,\hat T_2)$
 and $\phi_1\in\cL C^{\mathsf{ c}}(\hat T_1,\hat T_2)$. Since we just proved that each element in $\cL
 C(\hat T_1,\hat T_2)$ preserves the form of $\tilde\tau_{i}$, then $\phi_1=\psi_1\phi_0$
 also preserves the forms of $\tilde\tau_1,\tilde\tau_2$. On the other hand, we have shown that there exists a unique
 mapping   in $\cL C^{\mathsf c}(\hat T_1,\hat T_2)$ which transforms $\{\tau_1^*,\tau_2^*\}$ into $\{\tilde\tau_1,\tilde\tau_2\}$. This shows that $\phi_0=I$.
We have verified all assertions in (i).

 (ii).  
 It is easy to   see that $\cL C^{\mathsf{ c}}(\hat S)$ and $\cL C^{\mathsf{ c}}(\hat T_1,\hat T_2)$
 are invariant under conjugacy by $\rho$.
 We have $\Psi^{-1}\sigma\Psi=\sigma^*$ and $\Psi\in\cL C^{\mathsf{ c}}(\hat S)$.   Note that $\rho\sigma\rho
 =\sigma^{-1}$ and $\rho\sigma^*\rho$ have the same form as of    $(\sigma^*)^{-1}$, i.e. they
 are in $\cL C(\hat S)$   and have the same linear part.
 We have $\rho\Psi\rho\sigma\rho\Psi^{-1}\rho=\rho(\sigma^*)^{-1}\rho$.
 The uniqueness of $\Psi$ implies that $\rho\Psi\rho
=\Psi$ and
$\tau_2^*=\rho\tau_1^*\rho$. Thus, we obtain relations \re{a2e-}-\re{aiss}.
 Analogously, $\rho\psi_0\rho$ is still in $\cL C^{\mathsf{ c}}(\hat T_1,\hat T_2)$, and
 $\rho\phi_0\rho$ preserves the form of $\tilde\tau_1,\tilde\tau_2$. Thus $\rho\psi_0\rho=\psi_0$ and
 $\tilde\tau_2=\rho\tilde\tau_1\rho$, which gives us \re{ohai}.
\end{proof}

  We will also need the following uniqueness result.
\begin{cor}\label{finitever} Suppose that $\sigma$ has a non-resonant  linear part $\hat S$.
Let $\Psi$ be the 
unique
formal mapping in $\cL C^{\mathsf c}(\hat S)$ such that $\Psi^{-1}\sigma\Psi\in \cL C(\hat S)$. If $\tilde\Psi\in\cL C^{\mathsf c}(\hat S)$ is a polynomial
map of degree at most $d$ such that $\tilde\Psi^{-1}\sigma\tilde\Psi(\xi,\eta)=\tilde\sigma(\xi,\eta)+O(|(\xi,\eta)|^{d+1})$ and $\tilde\sigma\in\cL C(\hat S)$,
then $\tilde\Psi$ is unique. In fact, $\Psi-\tilde\Psi=O(d+1)$.
\end{cor}
\begin{proof} The proof is contained in the proof of \rp{ideal0}.
 Let us recap  it by using \re{ujpq+}-\re{vjpq} and the proposition.
We take a unique normalized mapping  $\Phi$
such that $\Phi^{-1}\tilde\Psi^{-1}\sigma\tilde\Psi\Phi\in\cL C(\hat S)$. By \re{ujpq+}-\re{vjpq},  $\Phi=I+O(d+1)$. From \rp{ideal0} it follows that
$\psi_0:=\tilde\Psi\Phi\Psi^{-1}\in \cL C(\hat S)$.  We obtain $\tilde\Psi\Phi=\psi_0\Psi$. Thus $\psi_0\Psi=\tilde\Psi+O(d+1)$. Since $\psi_0\in\cL C(\hat S)$,
and $\Psi$, $\tilde\Psi$ are in $\cL C^{\mathsf c}(\hat S)$, we conclude that   $\Psi=\tilde\Psi+O(d+1)$.
\end{proof}
 For clarity, we state the following uniqueness results on normalization.
\begin{cor}\label{uniqueM} Let $\sigma$ have a non-resonant linear part and let $\sigma$ be given by
\eq{idfg-++}
\nonumber
\xi_j'=M_j^0(\xi\eta)\xi_j+f^0_j(\xi,\eta), \quad \eta_j'=N_j^0(\xi\eta)\eta_j+g^0_j(\xi,\eta).
\eeq
Let $\Psi=I+(U,V)\in {\cL C}^{\mathsf{ c}}(\hat S)$ and let
$\sigma^*=\Psi^{-1}\sigma\Psi$ be given by
$$ 
\xi_j'=M_j(\xi\eta)\xi_j+ f_j(\xi,\eta), \quad\eta_j'=N_j(\xi\eta)\eta_j+g_j(\xi,\eta).
$$
Suppose that $
 (f^0,g^0)$ and $(f,g)$ are in $ {\cL C}^{\mathsf{ c}}_2(\hat S)$, $\ord(f^0,g^0)\geq d$, $
 \ord(f_j,g_j)\geq d$, and $d\geq2$.
Then  $\ord(U,V)\geq d$ and
\ga \label{m=m0}
M_{j,P}=M_{j,P}^0, \quad N_{j,P}=N_{j,P}^0, \quad 1\leq2|P|+1<2d-1.
\end{gather}
\end{cor}
\begin{proof}By  Corollary~\ref{finitever}, we know that $\ord(U,V)\geq d$.  Expanding both sides of $\sigma\Psi=\Psi\sigma^*$ for terms of degree less than $2d-1$, we obtain
\aln
M^0_j(\xi\eta)(\xi_j&+U_j(\xi,\eta))+DM^0_j(\xi\eta)(\xi V+\eta U)\xi_j+f_j^0(\xi,\eta)\\
&=M_j(\xi\eta)\xi_j+f_j(\xi,\eta)+U_j(M(\xi\eta)\xi,N(\xi\eta)\eta)+O(2d-1).
\end{align*}
Note that  $\xi_iV_i(\xi,\eta)\xi_j$ and $ \eta_iU_i(\xi,\eta)\xi_j$ and $U_j(M(\xi\eta)\xi,N(\xi\eta)\eta)$ do not
contain terms of the form $\xi^Q\eta^Q\xi_j$.  Comparing the coefficients of $\xi^P\eta^P\xi_j$ for $2|P|+1<2d-1$, we obtain the first identity in \re{m=m0}.  The second identity can be obtained similarly.
\end{proof}

When $p=1$, \rp{ideal0} is due to Moser and Webster~\ci{MW83}. In fact, they  achieved
$$
\tilde M_1(\zeta_1)=e^{\delta (\xi_1\eta_1)^s}.
$$
Here $\delta=0, \pm1$
for the elliptic case and  $\delta=0,\pm i$ for the hyperbolic case when $\mu_1$ is not a root of unity,   i.e.
$\gaa$ is {\em non-exceptional}.
In particular the normal form is always convergent, although the normalizing transformations
are generally divergent for the hyperbolic case.

Let us find out further normalization  that can be performed to preserve the form of $\sigma^*$.
In \rp{ideal0}, we have proved that if $\sigma$ is tangent to $\hat S$, there exists a unique
$\Psi\in\cL C^{\mathsf{ c}}(\hat S)$ such that $\Psi^{-1}\sigma\Psi$ is an element $\sigma^*$ in the
centralizer of $\hat S$. Suppose now that $\sigma=\tau_1\tau_2$
while $\tau_i$ is tangent to $\hat T_i$. Let   
$\tau_i^*=\Psi^{-1}\tau_i\Psi$.
 We have also proved that there is a unique $\psi_0\in\cL C^{\mathsf{ c}}(\hat T_1,\hat T_2)$ such that
 $\tilde\tau_i=\psi_0^{-1}\tau_i^*\psi_0$, $i=1,2$, are of the form \re{tl21}, i.e.
 \gan
 \tilde\tau_i\colon\xi_j'=\tilde\Lambda_{ij}(\zeta)\eta_j,
 \quad\eta_j'=\tilde\Lambda_{ij}^{-1}(\zeta)\xi_j;\\
 \tilde\sigma\colon\xi_j'=\tilde M_{j}(\zeta)\xi_j,\quad\eta_j'
 =\tilde M_{j}^{-1}(\zeta)\eta_j.
 \end{gather*}
Here $\zeta=(\xi_1\eta_1,\ldots, \xi_p\eta_p)$,  $\tilde\Lambda_{2j}=\tilde\Lambda_{1j}^{-1}$ and $\tilde M_j=\tilde\Lambda_{1j}^2$. We still have freedom to further normalize $\tilde\tau_1,\tilde\tau_2$ and to preserve their forms. However, any new coordinate transformation must be in $\cL C(\hat T_1,\hat T_2)$, i.e. it must have the form
$$
\psi_1\colon \xi_j\to a_j(\xi\eta)\xi_j, \quad\eta_j\to a_j(\xi\eta)\eta_j.
$$
When $\tau_{2j}=\rho\tau_{1j}\rho$, we require that
  $\psi_1$   commutes with $\rho$, i.e.
$$
a_e=\ov a_e,\quad a_h=\ov a_h, \quad a_s=\ov a_{s+s_*}.
$$
In $\zeta$ coordinates, the transformation $\psi_1$ has the form
\eq{zjaj}
\var\colon \zeta_j\to b_j(\zeta)\zeta_j, \quad 1\leq j\leq p
\eeq
with $b_j=a_j^2$.
Therefore, the mapping $\var$ needs to satisfy
$$
b_e>0, \quad b_h>0, \quad b_s=\ov b_{s+s_*}.
$$
Recall from \re{a2e-}-\re{aiss} the reality conditions on $\tilde M_j$
\begin{alignat*}{4} 
& \ov {\tilde M_e\circ{\rho_z}}&&=\tilde M_e,
\quad &&1\leq e\leq e_*;  \\
&\ov {\tilde M_h\circ{\rho_z}}&&=\tilde M_h^{-1},\quad && e_*< h\leq h_*+e_*;\\
&\tilde M_{s_*+s}&&=\ov {\tilde M_s^{-1}\circ{\rho_z}}, \quad && h_*+e_*< s\leq p- s_*.
\end{alignat*}
Here
\eq{rhoz5}\nonumber
{\rho_z}\colon\zeta_j\to\ov\zeta_j,
\quad\zeta_s\to\ov\zeta_{s+s_*}, \quad\zeta_{s+s_*}\to\ov\zeta_s
\eeq
for $1\leq j\leq e_*+h_*$ and $e_*+h_*<s\leq p-s_*$.

Therefore, our normal form problem leads to another normal form problem which is interesting in its own right.
To formulate a new normalization problem, let us  define
\ga\label{logm}
(\log \tilde M)_j(\zeta):=\begin{cases}
\log(\tilde M_j(\zeta)/{\tilde M_j(0)}), & 1\leq j\leq e_*,\\
-i\log(\tilde M_j(\zeta)/{\tilde M_j(0)}), & e_*< j\leq   p. 
\end{cases}
\end{gather}
Let $F=\log\tilde M:=((\log \tilde M)_1,\ldots, (\log \tilde M)_p)$.
Then   the reality conditions on $\tilde M$ become
\ga\label{reaf}
F={\rho_z} F{\rho_z}.
\end{gather}
 The transformations \re{zjaj}
will then satisfy
\eq{reality-phi}\nonumber
{\rho_z}\var{\rho_z}=\var, \quad b_j(0)>0, \quad 1\leq j\leq e_*+h_*.
\eeq
 By using $\log \tilde M$, we have transformed the reality condition on $M$ into a linear condition \re{reaf}. This will be useful   to further normalize $\tilde M$.
Therefore,  when $F'(0)$ is furthermore
diagonal and invertible and its $j$th diagonal entry is positive for $j=e,h$,
we apply a dilation $\varphi$ satisfying the above condition so that
$F$ is tangent to the identity. Then any further change of coordinates
must be tangent to the identity too.
Thus, we need to normalize the formal holomorphic mapping $F$ by composition $F\circ\var$,
for which we study in next subsection.

\subsection{A normal form
for maps  tangent to the identity  
}
%

Let us consider a germ of holomorphic mapping $F(\zeta)$ in ${\mathbf
C}^p$ with an invertible linear part ${\mathbf A}\zeta$ at the origin.
According to the inverse function theorem, there exists a holomorphic
mapping $\Psi$ with $\Psi(0)=0$, $\Psi'(0)=I$ such that $F\circ
\Psi(\zeta)=\mathbf A\zeta$.  On the other hand, if we impose some
restrictions on $\Psi$, we can no longer linearize
$F$ in general.

To focus on applications to CR singularity and to limit the scope
of our investigation,  we now deliberately restrict
 our analysis to the simplest case : $F$ is tangent to the identity.
We shall apply our result to $F=\log \tilde M$ as defined in the previous subsection.
In what follows, we shall devise  a normal
form of such an $F$ under right composition by $\Psi$ that
{\it preserve all coordinate hyperplanes}, i.e.
$\Psi_j(\zeta)=\zeta_j\tilde\Psi_j(\zeta)$, $j=1,\ldots, p$.

%

\begin{lemma}\label{fcfp}
Let $F$ be a  
formal holomorphic map of $\cc^p$ that is tangent to the identity  at the origin.
\bppp\item
There exists a  unique formal biholomorphic map $\psi$
which preserves all $\zeta_j=0$
such that   $\hat F \colonequals F\circ\psi$    has the form
\ga\label{hfj1}
\hat F(\zeta)= \zeta+ 
\hat f(\zeta), \quad\hat f(\zeta)=O(|\zeta|^2); \quad \pd_{\zeta_j}\hat f_j=0, \quad 1\leq j\leq p.
\end{gather}
\item If $F$ is convergent,  the $\psi$  in $(i)$ is convergent.
 If $F$ commutes with ${\rho_z}$,  so does  the
 $\psi$.
 \item   The formal normal form in (i) has the form 
 \eq{hfjqf}
 \hat f_{j,Q}=f_{j,Q} -\{Df_j \cdot f\}_{Q}+\cL F_{j,Q}(\{f\}_{|Q|-  2}), \quad q_j=0,
\quad |Q|>1.
 \eeq
 Here $\cL F_{j,Q}$ are universal polynomials   and   vanish at $0$.
\eppp
\end{lemma}
\begin{proof} (i)
Write  $F(\zeta)=\zeta+f(\zeta)$ and
\gan
\psi\colon \zeta_j' =\zeta_j+\zeta_jg_j(\zeta), \quad g_j(0)=0.
 \end{gather*}
 For $\hat F=F\circ\psi$,  we need to solve for $\hat f, g$ from
 \eq{hatfj}\nonumber
\hat f_j(\zeta) =\zeta_jg_j(\zeta)+f_j\circ\psi(\zeta).
 \eeq
Fix $Q=(q_1,\ldots, q_p)\in \nn^p$ with $|Q|>1$.
We obtain unique solutions
\ga
\label{jina} g_{j,Q-e_j}=-\{  f_j(\psi(\zeta))\}_Q,\quad q_j>0,\\
\hat f _{j,Q}:=\{  f_j(\psi(\zeta))\}_Q, \quad q_j=0.
\label{fjpk}
\end{gather}
 We first obtain $g_{j,Q-e_j}=-f_{j,Q}+\cL G_Q(\{f\}_{|Q|-1},\{g\}_{|Q|-2})$. This determines
\eq{ghqe}
g_{j,Q-e_j}=-f_{j,Q}+\cL G_Q(\{f\}_{|Q|-1}).
\eeq
Next, we expand $f_j(\psi(\zeta))=f_j(\zeta)+Df_j(\zeta)\cdot (\zeta_1 g_1(\zeta),\dots, \zeta_pg_p(\zeta))+\cL R_2f_j(\zeta; \zeta g(\zeta))$. The last term, with $\ord g\geq1$,  has the form
$$
\{\cL R_2f_j(\zeta; \zeta g(\zeta))\}_{Q}=\cL F_{j,Q}(\{f\}_{|Q|-2}, \{g\}_{|Q|-2})=\tilde {\cL F}_{j,Q}(\{f\}_{|Q|-2}).
$$
 Combining \re{fjpk}, the expansion, and \re{ghqe}, we obtain  \re{hfjqf}.

(ii)  Assume that $F$ is convergent.   Define
$  \ov h(\zeta)=\sum|h_Q|\zeta^Q.
$
We obtain for every multi-index  $Q=(q_1,\ldots, q_p)$ and for every $j$ satisfying  $q_j\geq 1$
$$
\ov g_{j,Q-e_j}\leq   \left\{\ov{  f_j}(\zeta_1+\zeta_1\ov g_1(\zeta), \ldots, \zeta_p+\zeta_p\ov g_p(\zeta))\right\}_Q.
$$
  Set $w(\zeta)=\sum\zeta_k\ov g_k(\zeta)$.  We obtain
$$
w(\zeta)\prec \sum \ov{f_j}(\zeta_1+  w(\zeta), \ldots, \zeta_p+w(\zeta)). 
$$
Note that $  f_j(\zeta)=O(|\zeta|^2)$ and $w(0)=0$.
By the Cauchy majorization and the implicit function theorem,  $w$ and hence $g, \psi, \hat f$ are convergent.

(iii)   Assume that ${\rho_z} F{\rho_z}=F$. Then $\rho_z \hat F\rho_z$ is   normalized,  $\rho_z\psi\rho_z$
is   tangent to the identity,  and  the $j$th component of $\rho_z\hat F\rho_z(\zeta)-\zeta$ is   independent of $\zeta_j$. Thus $\rho_z\psi\rho_z$ normalizes $F$ too.
   By the uniqueness of $\psi$, we obtain $\rho_z\psi\rho_z=\psi$.

  By rewriting  \re{fjpk},  we obtain
\eq{fjqfjq}
\hat f_{j,Q}=f_{j,Q}+\{f_j(\psi)-f_j\}_Q=f_{j,Q}+\cL F'_{j,Q}(\{f\}_{|Q|-1}, \{g\}_{|Q|-2}).
\eeq
 From \re{jina}, it follows that
 $$g_{k,Q-e_k}=-f_{k,Q}+\cL G_{k,Q-e_k}(\{f\}_{|Q|-1}, \{g\}_{|Q|-2}), \quad |Q|>1.$$
Note that $ \{g\}_0=0$ and $\{f\}_1=0$. Using the identity repeatedly, we obtain
$g_{k,Q-e_k}=-f_{k,Q}+\cL G_{k,Q-e_k}^*(\{f\}_{|Q|-1}).$ Therefore, we can rewrite \re{fjqfjq} as \re{hfjqf}.
\end{proof}

\subsection{A unique formal normal form of a reversible map $\sigma$}

We now state a normal form for $\{\tau_1,\tau_2,\rho\}$
under a
 condition on the third-order invariants  of $\sigma$.

\begin{thm}\label{ideal5}  Let $\tau_{1}$, $\tau_{2}$
be a pair of holomorphic involutions with linear parts $\hat T_i$. Let $\sigma=\tau_1\tau_2$.
Assume that the linear part of $\sigma$ is $$
\hat S\colon\xi_j'=\mu_j\xi_j, \quad \eta_j=\mu_j^{-1}\eta_j, \quad 1\leq j\leq p
$$
and $\mu_1,\ldots, \mu_p$ are non-resonant. Let
 $\Psi \in   {\cL C}^\mathsf{c}(\hat S)$ be the unique
 formal mapping such that
\gan
\tau_{i}^*=  \Psi^{-1}\tau_i\Psi\colon\xi_j'=\Lambda_{ij}(\xi\eta)\eta_j, \quad \eta_j'=\Lambda_{ij}(\xi\eta)^{-1}\xi_j;\\
\sigma^*=\Psi^{-1}\sigma\Psi\colon\xi_j'=M_j(\xi\eta)\xi_j, \quad \eta_j'=M_j(\xi\eta)^{-1}\eta_j
\end{gather*}
with $M_j=  \Lambda_{1j}\Lambda_{2j}^{-1}$.
Suppose that $\sigma$ satisfies the
 condition that   $\log M$
is  tangent to the identity. 
\bppp\item
Then there exists an invertible   formal map $\psi_1\in  {\cL C}(\hat S)$   such that
\ga\label{htai}\nonumber
\hat\tau_{i}=\psi_1^{-1}\tau_i^*\psi_1\colon\xi_j'
=\hat\Lambda_{ij}(\xi\eta)\eta_j, \quad \eta_j'=\hat\Lambda_{ij}(\xi\eta)^{-1}\xi_j;\\
\hat\sigma=\psi_1^{-1}\sigma^*\psi_1\colon\xi_j'=\hat
M_j(\xi\eta)\xi_j, \quad \eta_j'
=\hat M_j(\xi\eta)^{-1}\eta_j.\label{hsig}
\end{gather}
Here $\hat\Lambda_{2j}=\hat\Lambda_{1j}^{-1}$,   and  $\hat T_i$ is the linear part of $\hat\tau_i$.  Moreover,
 $\log\hat M_j(\zeta)-\zeta_j=O(2)$  is independent of $\zeta_j$
 for each $j$. 
\item    The centralizer of $\{\hat\tau_1,\hat\tau_2\}$ consists of $2^p$ dilations
  $(\xi,\eta)\to (a\xi,a\eta)$ with $a_j=\pm1$. And $\hat\Lambda_{ij}$ are unique.
   If   $ \Lambda_{ij}$ are
convergent, then $\psi_1$ is convergent too.
\item Suppose that    $\hat\sigma$ is
divergent. If $\sigma$ is formally equivalent
to a mapping $\tilde\sigma\in  {\cL C}(\hat S)$ then $\tilde\sigma$ must be divergent too.
\item  Let $\rho$ be given by \rea{eqrh} and let
 $\tau_2=
\rho\tau_1\rho$. Then the above $\Psi$ and $\psi_1$
  commute with $\rho$.
    Moreover, $\hat\tau_i$, $\hat\sigma$ are unique.
\eppp
\end{thm}
\begin{proof}  Assertions in (i)  are   direct consequences of \rp{ideal0} and \rl{fcfp}
 in which $F$ is the $\tilde M$ in \rp{ideal0}.  The assertion in (ii) on   the centralizer of $\{\hat\tau_1,\hat\tau_2\}$ is obtained
from   \re{lamtpsi}   of  \rp{ideal0} in which    $\tilde \Lambda_{ij}=\hat \Lambda_{ij}$. Indeed, by \re{lamtpsi}, if $\psi$ preserves $\{\hat\tau_1,\hat\tau_2\}$, then $\psi(\xi,\eta)=(c(\xi\eta)\xi,c(\xi\eta)\eta)$ and
$
\hat M_j(c^2(\xi\eta)\xi\eta)=\hat M_j(\xi\eta).
$
This shows that $\hat M\circ\tilde\psi=\hat M$ for $\tilde\psi(\zeta)=c^2(\zeta)\zeta$.  Since $\hat M-\hat M(0)$ is invertible then $\tilde \psi$ is the identity, i.e. $c_j=\pm1$.
Now (iii) follows from (ii) too.
Indeed, suppose $\sigma$ is formally equivalent to some convergent
$$
\tilde\sigma\colon\xi_j=\tilde M_j(\xi\eta)\xi_j,\quad\eta'_j=\tilde M_j(\xi\eta)^{-1}\eta_j.
$$
Then  by the assumption on the linear part of $\log
M$,  
 we can apply a dilation to achieve that $(\log \tilde M)'(0)$ is tangent to the identity.
 By \rl{fcfp}, there exists a unique convergent mapping $\var\colon\zeta_j'
=b_j(\zeta)\zeta_j$ ($1\leq j\leq p$) with $b_j(0)=1$  such that $\log\tilde M\circ\var$ is in the normal form $\log M_*$.
Then
$$
(\xi'_j,\eta'_j)=(b_j^{1/2}(\xi\eta)\xi_j,b_j^{1/2}(\xi\eta)\eta_j), \quad 1\leq j\leq p
$$
 transforms $\tilde\sigma$ into
a convergent mapping  $\sigma_*$. Since the normal form for $\log M$ is unique,
then $\hat\sigma=\sigma_*$. In particular, $\hat\sigma$   is convergent.

(iv). Note that $\rho\sigma\rho=\sigma^{-1}$. Also $\rho(\sigma^*)^{-1}\rho$ has the same form as $\sigma^*$.
By $(\rho\Psi^{-1}\rho)\sigma(\rho\Psi\rho)=(\rho\sigma^*\rho)^{-1}$, we conclude that $\rho\Psi\rho=\Psi$.
The rest of assertions
can be verified easily.
\end{proof}
 Note that $M^{-1}(\zeta)$ is also normalized in the sense that $\log M_j^{-1}(\zeta)+\zeta_j=O(|\zeta|^2)$ is independent of $\zeta_j$.
Under the condition that  $\log M$ is
tangent to the identity,
the above theorem completely
settles  the formal
classification of $\{\tau_1,\tau_2,\rho\}$.
It also says that {\bf the normal form $\hat\tau_1,\hat\tau_2$ can be achieved by a convergent transformation, if and only if $\sigma^*$ can be achieved by some convergent transformation}, i.e.  the $\Psi$ in the theorem is convergent.

 However, we would like state clear that our results do not rule out the case where
a refined normal form for $\{\tau_1^*,\tau_2^*,\rho\}$ is achieved  by  convergent transformation, while $\Psi$ is divergent, when
 $\log M$ is tangent to the identity. 

\subsection{An algebraic manifold with linear $\sigma$}

We conclude the section showing that when $\tau_1,\tau_2$ are normalized as   in this section,
$\{\tau_{ij}\}$ might still be very general;
  in particular $\{\tau_{1j},\rho\}$ cannot always be simultaneously linearized
even at the formal level.
This is one of main differences between $p=1$
and $p>1$.

\begin{exmp}\label{texpl} Let $p=2$.
Let $\phi$ be a holomorphic mapping of the form  
$$
\phi\colon\xi'_i=\xi_i+q_i(\xi,\eta), \quad \eta'_i=\eta_i+\la^{-1}_iq_i(T_1(\xi,\eta)),\; i=1,2.
$$
Here $q_i$ is a homogeneous quadratic polynomial map and
$$T_1(\xi,\eta)=(\la_1\eta_1,\la_2\eta_2,\la^{-1}_1\xi_1,\la^{-1}_2\xi_2).$$
Let $\tau_{1j}=\phi T_{1j}\phi^{-1}$ and $\tau_{2j}=\rho\tau_{1j}\rho$. Then $\phi$ commutes with $T_1$
and $\tau_1=T_1$. In particular $\tau_2=\rho T_1\rho$ and $\sigma=\tau_1\tau_2$ are in linear normal forms.
However,  $\tau_{11}$ is given by
\aln
 \xi_1'&=\la_1\eta_1-q_1(\la\eta,\la^{-1}\xi)+q_1(\la_1\eta_1,\xi_2,\la_1^{-1}\xi_1,\eta_2)+O(3),\\
 \xi_2'&=\xi_2-q_2(\xi,\eta)+q_2(\la_1\eta_1,\xi_2,\la_1^{-1}\xi_1,\eta_2)+O(3),\\
 \eta_1'&=\la_1^{-1}\xi_1-\la_1^{-1}q_1(\xi,\eta)+\la_1^{-1}q_1(\xi_1,\la_2\eta_2,\eta_1,\la_2^{-1}\xi_2)+O(3),\\
 \eta_2'&=\eta_2-\la_2^{-1}q_2(\la\eta,\la^{-1}\xi)+\la_2^{-1}q_2(\xi_1,\la_2\eta_2,\eta_1,\la_2^{-1}\xi_2)+O(3).
 \end{align*}
 Notice that the common zero set $V$ of $\xi_1\eta_1$ and $\xi_2\eta_2$ is invariant under $\tau_1,\tau_2,\sigma$ and $\rho$.
 In fact, they are linear on $V$.
 However, for $(\xi',\eta')=\tau_{11}(\xi,\eta)$, we have
\aln
\xi_1'\eta_1'&=-\eta_1q_1(0,\xi_2,\eta)+\eta_1q_1(0,\la_2\eta_2,\eta_1,\la_2^{-1}\xi_2) -\la_1^{-1}\xi_1q_1(0,\la_2\eta_2,\la^{-1}\xi)\\
&\quad+\la_1^{-1}\xi_1q_1(0,\xi_2,\la_1^{-1}\xi_1,\eta_2)\mod (\xi_1\eta_1,\xi_2\eta_2
, O(4)).
\end{align*}
For a generic $q$, $\tau_{11}$ does not preserve $V$.
\end{exmp}

By \rp{ideal0},  when the above linear $\sigma$ is non-resonant,  $\{\tau_{11},\tau_{12},\rho\}$ is not linearizable. 
By a simple computation, we can verify that $\sigma_j=\tau_{1j}\tau_{2j}$ for $j=1,2$ do not commute with each other. In fact,
we  proved in~\cite{GS15} that  if the  $\mu_1,\ldots, \mu_p$ are nonresonant, $\sigma_j$ commute pairwise, and $\sigma$ is linear as above, then $\tau_{1j}$ must be linear.

\setcounter{thm}{0}\setcounter{equation}{0}

\section{
Divergence of all 
normal forms of
 a reversible map $\sigma $}\label{div-sect}

Unlike the Birkhoff normal form for a Hamiltonian system, the Poincar\'{e}-Dulac normal form
is not unique for a general $\sigma$;   it just belongs to the centralizer of the linear part $S$ of $\sigma$. One can obtain a divergent normal form  easily from any non-linear Poincar\'{e}-Dulac normal form of
$\sigma=\tau_1\tau_2$
by conjugating with a divergent transformation in the centralizer of $S$; see \re{lamtpsi}.
We have seen  how
the small divisors   enter in the computation of the normalizing transformations via \re{ujpq+}-\re{vjpq}   and   \re{mpmp-}-\re{npnp}  in the computation of the normal forms.
To see the effect of small divisors on normal forms,  we  first assume a condition,  to be achieved later,
 on the third order invariants of $\sigma$ and then we shall need   to modify the
normalization procedure.
We will use two sequences of normalizing mappings to normalize $\sigma$.
The composition of
normalized mappings might not be normalized. Therefore, the new normal form $\tilde \sigma$
 might not be the
$\sigma^*$ in \rp{ideal0}. 
We will show that this $\tilde\sigma$,
  after it is transformed into the normal form $\hat\sigma$
  in  \rt{ideal5} (i),  is divergent. Using the divergence
of  $\hat\sigma$, we will then show that any other normal forms of $\sigma$ that are in the centralizer of $S$ must be divergent too. This last step requires a convergent solution
given by \rl{fcfp}.

Our goal is to see a small divisor in a normal form $\tilde\sigma$; however they appear as a product. This is more complicated
than the situation for the normalizing transformations, where a small divisor appears in a much simple
way.  In essence, a small divisor problem occurs naturally when one applies a Newton iteration scheme for
a convergence proof.
For a small divisor to show up in the normal form, we have to go beyond the Newton iteration scheme, measured
in the degree or order of approximation in power series.
Therefore, we first  refine the formulae \re{mpmp-}.
\begin{lemma}\label{mj0mj}
 Let $\sigma$ be a holomorphic mapping,
given by
$$
\xi_j'=M_j^0(\xi\eta)\xi_j+f_j(\xi,\eta), \quad
\eta_j'=N_j^0(\xi\eta)\eta_j+g_j(\xi,\eta), \quad 1\leq j\leq p.
$$
Here $ M_j^0(0)=\mu_j=N^0_j(0)^{-1}.$
Suppose that $\ord (f,g)\geq d  \geq4$ 
 and $I+(f,g)\in \cL C^{\mathsf c}(S)$.
There exist unique polynomials $U,V$
 of degree at most $2d-1$ such that $\Psi=I+(U,V)\in\cL C^{\mathsf c}(S)$  transforms $\sigma$
into
$$
\sigma^*\colon\xi'= M(\xi\eta)\xi+\tilde f(\xi,\eta),\quad
\eta'=N(\xi\eta)\eta+\tilde g(\xi,\eta)
 $$
 with $I+(\tilde f,\tilde g)\in\cL C^{\mathsf c}(S)$ and $\ord(\tilde f,\tilde g)\geq 2d$.
Moreover,  
\al
\label{ujpqcc}
U _{j,PQ}&=(\mu^{P-Q}-\mu_j)^{-1}\left\{f _{j,PQ}+\cL U _{j,PQ}^*(\del_{ \ell- 2},
\{M^0,N^0\}_{ \f{\ell-d}{2}};\{ f,g\}_{\ell-2})\right\},\\
V _{j,QP}&=(\mu^{Q-P}-\mu_j^{-1})^{-1}\left\{g _{j,QP}+\cL V _{j,QP}^*(\del_{\ell-2},
\{M^0,N^0\}_{d};\{ f,g\}_{\ell-2})\right\},
\label{ujpqcc6}
\end{align}
for  $2\leq |P|+|Q|=\ell\leq 2d-1$ and $\mu^{P-Q}\neq\mu_j$.  In particular,  $\ord(U,V)\geq d$.  Also,
\ga
M_{j,P}=M_{j,P}^0,  \quad  2|P|+1<2d-1,
 \label{mjpp0}  \\
 \label{mjppn}
 M_{j,P}=M_{j,P}^0+   \{Df_j(U,V)\}_{(P+e_j)P}, \quad
  2|P|+1=2d-1.
\end{gather}
 Assume further that
   \eq{dijm}
(M^0)'(0)=\diag(\mu_1, \ldots, \mu_p). 
\eeq
 Then for $2|P|+1=2d+1$, we have
 \begin{gather}
\label{mjpp--}
 M_{j,P}=M_{j,P}^0+ \mu_j\left\{2(U_j V_j)_{PP}+(U_j^2)_{(P+e_j)(P-e_j)}\right\}+\{Df_j(U,V)\}_{(P+e_j)P}.
\end{gather}
\end{lemma}
\begin{rem}\label{keyrem}  Note that \re{mjpp0} follows from  \re{m=m0}.
Formulae  \re{mjpp0}, \re{mjppn}, \re{mjpp--} give us an effective way to compute the Poincar\'e-Dulac normal form.
Although \re{mjppn} contains small divisors, it will be more convenient to associate small divisors to \re{mjpp--} when we have $3$ elliptic components in $\sigma$.
 \end{rem}
\begin{proof}
Let $D_i$ denote $\partial_{\zeta_i}$. Let $Du(\xi,\eta)$ and $Dv(\zeta)$ denote the gradients of two functions.
    Let us expand both sides of the $\xi_j$ components of $\Psi\sigma^*=\sigma\Psi$ for terms of degree $2d+2$. For its left-hand side,   Corollary~\ref{finitever} implies that $\ord(U,V)\geq d$ and we can use $\ord DU_j\cdot (\tilde f,\tilde g)\geq 2d-1+d\geq 2d+2$ as $d\geq3$.
For its right-hand side, we use \re{mj0x}-\re{bjpq--}. We obtain
\al\label{mjxexc}
M_j\xi_j&+\tilde f_j(\xi,\eta)+U_j(M\xi,N \eta)
=f_j(\xi,\eta)+ Df_j(\xi,\eta)(U, V)\\
\nonumber&  +A_j(\xi,\eta)
+\left(M^0_j +DM_j^0 (\eta U +\xi V  +UV)\right)(\xi_j+U_j)+O(2d+2), 
\end{align}
 where $M,N,M^0,N^0$ are evaluated at $\xi\eta$ and
$U,V$ are evaluated at $(\xi,\eta)$.

Since $\tilde f(M(\xi\eta)\xi,N(\xi\eta)\eta)=O(|\xi,\eta|^{2d})$, then      \re{ujpqcc}-\re{ujpqcc6} follow from \re{ujpq+}-\re{vjpq},
where  by Definition~\ref{notation}
 $$\cL U_{j,PQ}^*(\cdot; 0)=\cL V^*_{j,QP}(\cdot;0) =0.$$
Next, we refine \re{mpmpnsd} to verify the remaining assertions.
We recall  from \re{ajpq} the   remainders
\aln 
A_j(\xi,\eta)&=
 R_2M_j^0(\xi\eta;\xi U+\eta V+UV)(\xi_j+U_j) +R_2f_j(\xi,\eta;U,V). 
 \end{align*}
Here by \re{taylor}, we have the Taylor remainder formula
\eq{}
\nonumber
 R_{2}f(x;y)=2\int_0^1(1-t)\sum_{|\alpha|=2}\frac{1}{\alpha!}\partial^{\alpha}f(x+ty)y^{\alpha}\, dt.
\eeq
Since $\ord(U,V)\geq d$, $\ord(f,g)\geq d$, and  $d\geq4$, then $A_j$, 
defined by \re{ajpq},
 satisfies
\gan
A_j(\xi,\eta)=O(|(\xi,\eta)|^{2d+2}). 
\end{gather*}
Recall that $f_j(\xi,\eta)$ and $ U_j(\xi,\eta)$ do not contain terms of the form $\xi_j\xi^P\eta^P$, while
 $g_j(\xi,\eta)$ and $ V_j(\xi,\eta)$ do not contain terms of the form $\eta_j\xi^P\eta^P$.
 Comparing both sides of \re{mjxexc} for coefficients of $\xi_j\xi^P\eta^P$ with   $|P|=d-1$, we get  \re{mjppn}.

   Assume now that \re{dijm} holds.
Assume that $i\neq j$.  Then $D_iM_j^0(\xi\eta)=O(|\xi\eta|)$. We  see  that
$D_iM_j^0(\xi\eta)\eta_iU_i(\xi,\eta)$ and $ D_iM_j^0(\xi\eta)\xi_iV_i(\xi,\eta)$ do not
contain terms of $\xi^P\eta^P$, and
$$D_iM_j^0(\xi\eta)\xi_iU_i(\xi,\eta)V_i(\xi,\eta)=O(2d+3).$$
 Since $(\tilde f,\tilde g)\in \cL C^{\mathsf c}_2(S)$ and
$(\tilde f,\tilde g)=O(2d)$, then
$\tilde f_j(M(\xi\eta)\xi,N(\xi\eta)\eta)$ does not contain terms $\xi^P\eta^P\xi_j$ for $2|P|+1=2d+1$.
Now \re{mjpp--} follows from a direction computation. \end{proof}

Set  $|\delta_N(\mu)|\colonequals\max\left\{|\nu|\colon\nu\in\delta_N(\mu)\right\}$ for
\eq{delnmu}\nonumber
\delta_N(\mu)=\bigcup_{j=1}^p\left\{\mu_j, \mu_j^{-1}, \f{1}{\mu^{P}-\mu_j}\colon P\in\zz^p, P\neq e_j,|P|\leq N\right\}.
\eeq
 \begin{defn}\label{sdon}We
say that   $\mu^{P_*-Q_*}-\mu_j$ and
 $\mu^{Q_*-P_*}-\mu_j^{-1}$ are {small divisors} of   {\it height} $N$, if there exists a partition
\gan
\bigcup_{  i=1}^p\Bigl\{|\mu^{P-Q}-\mu_i|\colon P,Q\in\nn^p, |P|+|Q|\leq N, \mu^{P-Q}\neq\mu_i\Bigr\}=S_N^0\cup S^1_N
\end{gather*}
with $
 |\mu^{P_*-Q_*}-\mu_j|\in S_N^0$ and $
S^1_N\neq\emptyset$ such that
\gan
 \max S_N^0<C\min S_N^0,\quad
\max S_N^0<  
(\min S^1_N)^{L_N}<1.
\end{gather*}
Here $C$ depends only on an upper bound of  $|\mu|$ and $|\mu|^{-1}$ and
$$
L_N\geq N.
$$
  If $|\mu^{P_*-Q_*}-\mu_j|$ is in $S_N^0$ and if $P_*,Q_*\in\nn^p$,  we call $  |P_*-Q_*|$ the   {\it degree} of the small
divisors $\mu^{P_*-Q_*}-\mu_j$ and $\mu^{Q_*-P_*}-\mu_j^{-1}$.
\end{defn}

  To avoid confusion, let
us call $\mu^{P_*-Q_*}-\mu_j$ that appear in $S_N^0$ the  {\it exceptional} small divisors. These small divisors have been used by Cremer~\cite{Cr28} and Siegel~\ci{Si41}.
The degree and height play different roles in computation. The height serves as the maximum degree of all small  divisors that
need to be considered in computation.

Roughly speaking, the quantities in $S^0_N$ are comparable but they are much smaller than the ones in $S^1_N$.
We will construct $\mu$ for any prescribed
sequence of positive integers $L_N$ so that
$$
\max S_N^0<(\min S^1_N)^{L_N}<1
$$
for a subsequence $N=N_k$ tending to $\infty$. Furthermore, to use the small divisors we will
 identify  all   exceptional small divisors of   height  $2N_k+1$ and  all    degrees of the exceptional small divisors
with $N_k$ being the smallest.

We start with the following lemma which gives us small divisors that decay as rapidly as we wish.
\begin{lemma}\label{smallvec}
Let  $L_k$ be a   strictly increasing
 sequence of positive integers. There exist a real number $\nu\in(0,1/2)$
  and a sequence $(p_k,q_k)\in \nn^2$  such that
$e,1,\nu$ are linearly independent over $\qq$, and
\begin{gather}\label{epq-}
|q_k\nu-p_k-e|\leq\Delta(p_k,q_k)^{ L_{  p_k+q_k}},  \\
\label{dpkqk}\Delta(p_k,q_k)=\min\Bigl\{\frac{1}{2}, |q\nu-p-re|\colon  0<|r|+|q|<3(q_k+1),
\\
\nonumber  (p,q,r)\neq  0, \pm(p_k,q_k,1), \pm2(p_k,q_k,1)\Bigr\}.
\end{gather}
\end{lemma}
\begin{proof} We consider two increasing sequences  $  \{m_k\}_{k=1}^\infty, \{n_k\}_{k=1}^\infty$ of positive integers, which are
to be chosen. For $k=1,2,\ldots$, we set
\gan
\nu=\nu_k+\nu_k',\quad \nu_k=\sum_{\ell=1}^{ k}\frac{1}{m_\ell!}\sum_{j=0}^{  n_\ell}\f{1}{j!}, \quad
\nu_k'=\sum_{\ell>k}\frac{1}{m_\ell!}\sum_{j=0}^{  n_\ell}\f{1}{j!},\\
q_k=m_k!. \end{gather*}
We choose $m_k>(m_\ell)!(  n_\ell!)$ for $k>\ell$ and decompose
\begin{gather*}
q_k\nu=p_k+e_k+e_k',\\
p_k=m_k!\nu_{k-1}\in\nn,\quad
e_k=\sum_{\ell=0}^{  n_k}\f{1}{k!},\quad e_k'=m_k!\nu_k'.
\end{gather*}
We have   $e_k'<m_k!\sum_{\ell>k}\f{e}{m_\ell!}$ and
\ga
\nonumber
q_k\nu=p_k+e+e_k'-\sum_{\ell=  n_k+1}^\infty\frac{1}{\ell!},\\
\label{ekpl}
|q_k\nu-p_k-e|\leq m_k!\nu_k'+\sum_{\ell=n_k+1}^\infty\f{1}{\ell!}<\left\{12(3(q_k+1)^3)!\right\}^{-    L_{ p_k+q_k}}.
\end{gather}
Here $\re{ekpl}_k$ is achieved by  choosing   $(m_2,n_1)$, \ldots, $(m_{k+1},n_{k})$ successively.
 Clearly we can get $0<\nu<1/2$ if $m_1$ is sufficiently large.

Next, we want to show that $re+p+q\nu\neq0$ for all integers $p, q,r$ with $(p,q,r)\neq(0,0,0)$. Otherwise, we  rewrite  $-m_k!p=
m_k!(q\nu+re)$ as
$$
-m_k!p=qp_k+r\sum_{j=  0}^{m_k}\f{m_k!}{j!}+qe+q\left(e_k'-\sum_{\ell=  n_k+1}^\infty\f{1}{\ell!}\right)+r\sum_{j>m_k}\frac{m_k!}{j!}.
$$
The left-hand side is an integer. On the right-hand side,     the first two terms are integers, $qe$ is a fixed
irrational number, and the rest terms tend to $0$
as $k\to\infty$. We get a contradiction.

To verify \re{epq-}, we need to show that for each tuple $(p,q,r)$ satisfying \re{dpkqk},
\eq{qn-p}
|q\nu-p-re|\geq |q_k\nu-p_k-e|^{\f{1}{  L_{p_k+q_k}}}.
\eeq
We first note the following elementary inequality
\eq{elemineq}
|p+qe|\geq \frac{1}{(q-1)!}\min\left\{3-e, \frac{1}{q+1}\right\}, \quad p, q\in \zz, \quad q\geq1.
\eeq
Indeed,  the inequality holds for $q=1$. For $q\geq2$ we have $q!e=m+\e$ with $m\in \nn$ and
$$
\e:=\sum_{k=q+1}^\infty\frac{q!}{k!}>\frac{1}{q+1}.
$$
Furthermore,  $1-\e>1-\frac{2}{q+1}=\frac{q-1}{q+1}$ as
$$
\e<\frac{1}{q+1}+\sum_{k\geq q+2}\frac{1}{k(k-1)}=\frac{2}{q+1}.
$$
We may assume
that $q\geq0$.
If $q=0$, then $|r|<3q_k+3$   by condition in \re{dpkqk} and hence
$|p+re|\geq \frac{1}{(3q_k+4)!}$ by \re{elemineq}. Now \re{qn-p} follows from \re{ekpl}.
  Assume that $q>0$.
We have
\al \label{qnpq}
|-q\nu+p+re|&\geq|-q\f{p_k+e}{q_k}+p+re|-q
\f{|e+p_k-q_k\nu|}{q_k}\\
&=\left|\f{q_kp-qp_k}{q_k}+\f{rq_k-q}{q_k}e\right|-q
\f{|e+p_k-q_k\nu|}{q_k}.
\nonumber\end{align}
We first verify that $q_kp-qp_k$ and $q-rq_k$ do not vanish simultaneously. Assume that both are zero. Then $(p,q,r)=r(p_k,q_k,1)$. Thus $|r|\neq 1,2$,  and   $|r|\geq3$ by conditions in \re{dpkqk};
we obtain  $|r|+|q|\geq3(|q_k|+1)$, a contradiction.  Therefore,
 either $q_kp-qp_k$ or $rq_k-q$
is not zero. By \re{elemineq} and \re{qnpq},
\begin{align*}
&|-q\nu+p+re|\geq  \f{1}{q_k}\cdot \frac{1}{3}\cdot \f{1}{(|rq_k-q|+1)!}-q
\f{|e+p_k-q_k\nu|}{q_k}\\
&\qquad\geq  \f{1}{(3q_k+4)^2!}-4|e+p_k-q_k\nu|.
\end{align*}
Using \re{ekpl} twice,  we obtain the next two inequalities:
$$|-q\nu+p+re|\geq\yt
\left\{(3q_k+4)^2!\right\}^{-1}\geq|p_k+e-q_k\nu|^{\f{1}{  L_{p_k+q_k}}}.$$  The two ends give us \re{qn-p}.
\end{proof}

We now reformulate the above lemma as follows.
\begin{lemma}\label{smallvec+} Let $L_k$ be a strictly increasing sequence of positive integers. Let  $\nu\in(0,1/2)$, and let $p_k$ and $ q_k$ be
positive integers as in \rla{smallvec}.
Set 
$(\mu_1,\mu_2,\mu_3):=(  e^{-1},e^\nu, e^{e})$.  Then
\begin{align}\label{epqcc}
|\mu^{P_k}-\mu_3|&\leq (C\Delta^*(P_k))^{  L_{|P_k|}},\quad  P_k= (p_k,q_k,0),\\
\label{delspk} \Delta^*(P_k)&=\min_j\Bigl\{|\mu^{R}-\mu_j|\colon R\in\zz^3,  |R|\leq 2(  q_k+p_k)+1,
\\
&\quad  R-e_j\neq 0,  \pm(  p_k,q_k,-1), \pm2(p_k,q_k,-1)\Bigr\}.
\nonumber\end{align}
Here $C$ does not depend on $k$.  Moreover,
  all   exceptional small divisors   of height $2|P_k|+1$ have   degree at least $|P_k|$. Moreover,
   $\mu^{P_k}-\mu_3$ is the only   exceptional small divisor
of  degree $|P_k|$ and
 height $2|P_k|+1$.
\end{lemma}
In the definition of $\Delta^*(P_k)$, equivalently we require that
$$ 
R \neq  P_k,\     R_k^1, \ R_k^2,\  R_k^3
$$ 
with  $R_k^1:=- P_k+2e_3,  R_k^2:=2 P_k-e_3,$ and $ R_k^3:=-2P_k+3e_3$.
Note that $ |R_k^1|=|P_k|+2, |R_k^2|=2|P_k|+1$, and $|R_k^3|=2|P_k|+3$  are bigger
than  $|P_k|$, i.e.  the degree of the exceptional small divisor $\mu^{P_k}-\mu_3$.
 Each $\mu^{R_k^i}-\mu_3$
is a small divisor comparable with $ \mu^{P_k}-\mu_3$. Finally,  $\Del^*(P_k)$ tends to zero as $|P_k|\to\infty$.
Let us set
$N:=2|P_k|+1$,
 and
\aln
S_N^0:&=\left\{  |\mu^{P_k}-\mu_3|,|\mu^{R_k^1}-\mu_3|,|\mu^{R_k^2}-\mu_3|,|\mu^{R_k^3}-\mu_3|\right\},\\
S_N^1: &= \bigcup_j\Bigl \{ |\mu^{R}-\mu_j|\colon R\in\zz^3,  |R|\leq 2(q_k+p_k)+1,  
\\
\nonumber&\qquad\quad R-e_j\neq 0,\   \pm(p_k,q_k,-1),\  \pm2(p_k,q_k,-1)\Bigr\}.\end{align*}
This
implies that the last paragraph of Lemma~\ref{smallvec+} holds when
the $L_N$ in Definition~\ref{sdon}, denoted it by $L_N'$, takes the value
$L_N'=   L_{2N+1}$, while  $L_N$ is prescribed in  Lemma~\ref{smallvec+}.
\begin{proof}
By \rl{smallvec}, we find a real number $\nu\in(0,1/2)$ and positive integers $p_k,q_k$ such that $e,1,\nu$ are linearly independent over $\qq$ and \begin{gather}\label{epq}
|q_k\nu-e-p_k|\leq\Delta(p_k,q_k)^{  L_{|P_k|}},\\
\Delta(p_k,q_k)=\min\left\{|q\nu-re-p|\colon  0<|r|+|q|<3(q_k+1),\right.\nonumber
\\
 \qquad \qquad\qquad\left. (p,q,r)\neq  0, \pm(  p_k,q_k, 1), \pm2(p_k,q_k,1)\right\}.
\nonumber\end{gather}
 Note that $\mu_1,\mu_2,\mu_3$
 are   positive 
  and non-resonant.  We have
  $$
 |\mu^{P_k}-\mu_3|=|\mu_3|\cdot|e^{q_k\nu-p_k-e}-1|.
 $$
Let $\nu^*:=(  -1,\nu,e)$. If $|R\cdot\nu^*-\nu_j^*|<2$, then by the intermediate value theorem
$$
  e^{-2}|\mu_j||R\cdot\nu^*-\nu_j^*|\leq |\mu^R-\mu_j|\leq  e^2|\mu_j||R\cdot\nu^*-\nu_j^*|.
$$
If $R\cdot\nu^*-\nu_j^*>2$ or $R\cdot\nu^*-\nu_j^*<-2$, we have
$$
|\mu^R-\mu_j|\geq   e^{-2}|\mu_j|.
$$
Thus,  we can restate the properties of $\nu^*$ as   follows:
\begin{align*}
\nonumber
&|\mu^{-(p_k,q_k,0)}-\mu_3|\leq  C' (C'\tilde\Delta(p_k,q_k))^{  L_{|P_k|}},\\
& \tilde\Delta(p_k,q_k)=\min\left\{ |\mu^{(p,q,r)}-1|\colon   0<|r|+|q|<3(q_k+1),\right.
\\
&\qquad\qquad\qquad  \left. (p,q,r)\neq 0,\pm(  p_k,q_k,-1),
\pm2(p_k,q_k,-1)\right\}.
\end{align*}
Recall that $0<\nu<1/2$. By \re{epq}, we have $|q_k \nu - e - p_k|<1$. Since
$p_k, q_k$ are positive,  then $p_k<\nu q_k<q_k/2$.
Assume that $|  \mu^{R}-\mu_j|=\Del^*(P_k)$,  $|R|\leq 2(p_k+q_k)+1$, and
$$R-e_j\neq 0,\pm(p_k,q_k,-1), \pm 2(p_k,q_k,-1).$$
Set $R':=R-e_j$ and $(p,q,r):=R'$. Then $\Del^*(P_k)=|\mu_j||\mu^{R'}-1|$.
Also, $|r|+|q|\leq|R'|\leq |R|+1\leq
2(p_k+q_k)+2\leq q_k+2q_k+2<3(q_k+1)$.  This shows that $|\mu^{Q'}-1|\geq\tilde\Del(p_k,q_k)$.
We obtain $\Del^*(P_k) \geq\mu_j\Del(p_,q_k)$.  We have verified \re{epqcc}.
  For the remaining assertions, see the remark following the lemma.
\end{proof}

  In the above we have retained $\mu_j>0$ which
are sufficient to realize $\mu_1, \mu_2, \mu_3$, $\mu_1^{-1},\mu_2^{-1},\mu_3^{-1}$ as eigenvalues of $\sigma$ for an elliptic
complex tangent.  Indeed, with $0<\mu_1<1$, interchanging $\xi_1$ and $\eta_1$ preserves $\rho$ and changes the $(\xi_1,\eta_1)$
components of $\sigma$ into   $(\mu_1^{-1}\xi_1,\mu_1\eta_1)$.

We are ready to prove   \rt{divsig}, which is restated here:
\begin{thm}\label{divnf1} 
There exists a non-resonant elliptic  real analytic $3$-submanifold $M$ in $\cc^6$ such
that $M$
  admits the maximum number of deck transformations
 and all Poincar\'e-Dulac normal forms of  the $\sigma$ associated to $M$ are divergent.
 \end{thm}
\begin{proof}
We will not construct  the real analytic submanifold $M$ directly. Instead, we will
construct a family  of involutions $\{\tau_{11}, \ldots, \tau_{1p},\rho\}$ so that all Poincar\'e-Dulac
normal forms of $\sigma$ are divergent. By the realization
 in \rp{mmtp}, we get the desired submanifold.

We first give an outline of the proof.
To prove the theorem, we first deal with the associated  $\sigma$ and its normal form $\tilde \sigma$,  which belongs to
the centralizer of $S$,   the linear part of $\sigma$ at the origin.
Thus $\sigma^*
$ has the form
$$
\sigma^*\colon \xi'= M(\xi\eta)\xi, \quad \eta'= N(\xi\eta)\eta.
$$
We assume that $\log M$ is tangent to identity at the origin.
We then normalize $
\sigma^*$ into the normal form
$\hat \sigma$ stated in \rt{ideal5}  (i). (In \rl{mj0mj} we take $F=\log M$
and $\hat F=\log\hat M$.)
We will show that $\hat\sigma$ is divergent if $\sigma$ is well chosen. By \rt{ideal5} (iii),   all normal forms
 of $\sigma$ in the centralizer of  $S$
  are divergent.
To get $\sigma^*$, we use the normalization  of \rp{ideal0} (i). To get $\hat\sigma$,
we  normalize further  using
\rl{fcfp}.  To find a divergent $\hat\sigma$,  we need  to tie the normalizations of
two formal normal forms together, by keeping
track of the small divisors  in the two normalizations.

We will start with our initial pair of involutions $\{\tau_1^0,\tau_2^0\}$
 satisfying $\tau_2^0=\rho\tau_1^0\rho$ such that $\sigma^0$
 is  a  third order  perturbation of $S$. We
 require that $\tau_1^0$ be the composition of $\tau_{11}^0, \ldots, \tau_{1p}^0$. The latter
can be realized by a real analytic submanifold by using \rp{mmtp}. We will then perform a sequence
of holomorphic  changes of coordinates $\var_k$ such that $\tau_1^{k}=\var_k\tau_1^{k-1} \var_k^{-1}$, $\tau_2^{k}=\rho\tau_1^{k}\rho$, and  $\sigma^{k}=\tau_1^{k}\tau_2^{k}$.
  By abuse of notation, $\tau_i^k, \sigma^k$, etc.  do not stand for iterating the maps $k$ times.
Each $\var_k$  is tangent to the identity
to order $d_k$.
For a suitable choice of $\var_k$, we want to show that
the coefficients of order $d_k$ of the normal form of $\sigma^k$
increase rapidly to the effect that the coefficients of the normal form of the limit
mapping $\sigma^\infty$  increase rapidly too.
Here we will use the  exceptional small divisors to achieve the rapid growth of the coefficients of
the normal forms.
  Roughly speaking,  the latter requires us to keep track  the rapid growth for a sequence of coefficients in the normal form in a sequence of  two-step normalizations. Recall from \rl{mj0mj} that if we have
$$
\sigma\colon \xi_j'=M_j(\xi\eta)\xi_j+f_j(\xi,\eta), \quad
\eta_j'=N_j(\xi\eta)\eta_j+g_j(\xi,\eta), \quad 1\leq j\leq p.
$$
with $\ord (f,g)\geq d  \geq4$ 
 and $(f,g)\in \cL C_2^{\mathsf c}(S)$, then there is a polynomial mapping $\Psi\colon \xi'=\xi+\hat U(\xi,\eta),\eta'=\eta+V(\xi,\eta)$ in $\cL C^{\mathsf c}(S)$ that has order $d$ and  degree at most $2d-1$ such that for the new mapping
 $$
\hat\sigma:=\Psi^{-1}\sigma\Psi\colon\xi_j'=\hat M_j(\xi\eta)\xi_j+\hat f_j(\xi,\eta), \quad
\eta_j'=\hat N_j(\xi\eta)\eta_j+\hat g_j(\xi,\eta),
 $$
the coefficients of $M_j(\xi\eta)\xi_j$ of degree $2d+1$
have the form
\eq{hatmjm}
\hat M_{j,P}=M_{j,P}+ \mu_j\left\{2(\hat U_j \hat V_j)_{PP}+(\hat U_j^2)_{(P+e_j)(P-e_j)}\right\}+\{Df_j(\hat U,\hat V)\}_{(P+e_j)P}.\eeq
It is crucial that for suitable multi-indices, {\em both} $\hat U_j,\hat V_j$ contain the exceptional small divisors of degree $d$ as formulated in \rl{smallvec+} (see also Definition~\ref{sdon}). Although $Df_j(\hat U,\hat V)$ contains (exceptional) small divisors, they can only appear at most once in each term, provided $f_j$ contains no small divisor of degree $d$.
The formula \re{hatmjm} appears as simple as it is, it requires
that $\sigma$ has been normalized to degree $d$. To achieve such a $\sigma$, we need to use a preliminary change of coordinates $\Phi\colon \xi'=\xi+U(\xi,\eta),\eta'=\eta+V(\xi,\eta)$ via polynomials of degree less than $d$. The $\Phi$ depends only on small divisors of degree $<d$, but none of them are exceptionally small. Therefore, by composing $\Psi\Phi$, we obtain \re{hatmjm} where small divisors of degree $<d$ are absorbed into terms $M_{j,P}$ and the products of two exceptional small divisors in \re{hatmjm}, if they exist, dominate the other terms in $\hat M_{j,P}$. Of course, we need to apply a sequence of transformations $\Phi,\Psi$ and we should leave the coefficients of a certain degree unchanged in the process once they become large, which are possible   by Corollary~\ref{uniqueM}.


We now present the proof. Let $\sigma^0=\tau_1^0\tau_2^0$, $\tau_2^0=\rho\tau_1^0\rho$, and
\begin{alignat*}{3}
&\tau_1^0\colon\xi_j'=\Lambda_{1j}^0(\xi\eta)\eta_j, \quad &&\eta_j'=(\Lambda_{1j}^0(\xi\eta))^{-1}\xi_j,\\
&\sigma^0\colon\xi_j'=(\Lambda_{1j}^0(\xi\eta))^{2}\xi_j, \quad &&\eta_j'=(\Lambda_{1j}^0(\xi\eta))^{-2}(\xi\eta)\eta_j.
\end{alignat*}
Since we consider the elliptic case, we require that $(\Lambda_{1j}^0(\xi\eta))^2=\mu_je^{\xi_j\eta_j}$.
 So $\zeta\to (\Lambda_{1}^0)^{2}(\zeta)$ is  biholomorphic. Recall that $\sigma^0$ can be realized by $\{\tau_{11}^0, \ldots, \tau_{1p}^0,\rho\}$.
We will take
\ga\label{rest1}
\var_k\colon\xi_j'=(\xi  -h^{(k)}(\xi),\eta),\quad  \ord h^{(k)}=d_k>3,\\
\label{rest2} d_{k}\geq   2d_{k-1},   \quad |h^{(k)}_P|\leq 1.
\end{gather}
 We will also choose each $h_j^{(k)}(\xi)$ to have one monomial only.  Let $\Del_r:=\Del_r^3$ denote the polydisc of radius $r$. Let $\|\cdot\|$
be the sup norm on $\cc^3$.
  Let $H^{(k)}(\xi)=\xi -h^{(k)}(\xi)$ and we first verify that $H_k=H^{(k)}\circ \cdots\circ H^{(1)}$ converges to a holomorphic function
 on the polydisc $\Delta_{r_1}$ for $r_1>0$ sufficiently small; consequently, $\var_k\circ\cdots\circ\var_1$ converges to a germ of  holomorphic map   $\var^\infty$ at the
 origin. Note that $H^{(k)}$ sends   $\Del_{r_k}$ into $\Del_{r_{k+1}}$ for $r_{k+1}=r_k+r_k^{d_k}$. We want to show that when $r_1$
 is sufficiently small,
 \eq{rksk}
 r_k\leq s_k:=(2-\f{1}{k})r_1.
 \eeq
 It holds for $k=1$.
 Let us show that  ${r_{k+1}}/{r_k}-1\leq \theta_k:= s_{k+1}/s_k-1$, i.e.
 $$
 r_k^{d_k-1}\leq \theta_k=\f{1}{(k+1)(2k-1)}.
 $$
 We have $(2r_1)^{d_k-1}\leq (2r_1)^k$ when $0<r_1<1/2$.  Fix $r_1$  sufficiently small such that  $(2r_1)^k<\f{1}{(k+1)(2k-1)}$
 for all $k$.  By induction, we obtain \re{rksk} for all $k$. In particular, we have $\|h^{(k)}(\xi)\|\leq \|\xi\|+\|H^{(k)}(\xi)\|\leq
 2r_{k+1}$ for $\|\xi\|<r_k$.   To show the convergence of $H_k$, we write
 $
H_{k}- H_{k-1}(\xi)= -h^{(k)}\circ   H_{k-1}.$  By the Schwarz lemma,  we obtain $$
\| h^{(k)}\circ   H_{k-1}(\xi)\|\leq \f{ 2r_{k+1}}{r_1^{d_k}}\|\xi\|^{d_k}, \quad \|\xi\|<r_1.
 $$
  Note that the above estimate is uniform under conditions \re{rest1}-\re{rest2}. Therefore, $  H_k$ converges to a holomorphic function on $\|\xi\|<r_1$.

  Throughout the proof, we make initial assumptions that
  $d_k$ and $h^{(k)}$ satisfy \re{rest1}-\re{rest2},   $e^{-1}\leq \mu_j\leq e^e$, and $\mu^Q\neq1$ for $Q\in\zz^3$ with $Q\neq0$.
Set $\sigma^{k}=\tau_1^{k}\tau_2^{k}$, $\tau_2^{k}=\rho\tau_1^{k}\rho$, and
$$
\tau_1^{k}=\var_k\tau_1^{k-1} \var_k^{-1}.
$$
We want $\sigma^{ k }$ not to be holomorphically equivalent to $\sigma^{ k-1 }$. Thus we have chosen
a $\varphi_k$ that does not commute with $\rho$ in general.  Note that $\sigma^{ k }$ is still generated by a real analytic submanifold; indeed, when $\tau_{i}^{ k-1 }
=\tau_{i1}^{ k-1 }\cdots\tau_{ip}^{ k-1 }$ and $\tau_{2j}^{ k-1 }=\rho\tau_{1j}^{ k-1 }\rho$, we
still have the same identities if the superscript $ k-1 $ is replaced by $ k $
and $\tau_{1j}^{ k }$ equals $\varphi_k\tau_{1j}^{ k-1 } \varphi_k^{-1}$.
It is clear that $\sigma^{ k }=\sigma^{ k-1 }+O(d_k)$.  As power series, we have
\eq{sinf}
\sigma^{ \ell }=\sigma^{ k-1 }+O(d_k), \quad  k \leq \ell\leq \infty.
\eeq
  Note that as limits in convergence,
$\tau_{ij}^\infty=\lim_{k\to\infty}\tau_{ij}^{ k }$, $\tau_i^\infty=\lim_{k\to\infty}\tau_i^{ k }$ and $\sigma^\infty=\lim_{k\to\infty}\sigma^{ k }$ satisfy
\begin{gather*}
\tau_{2j}^\infty=\rho\tau_{1j}^\infty\rho, \quad \tau_i^\infty=\tau_{i1}^\infty\cdots\tau_{ip}^\infty, \quad
 \tau_2^\infty=\rho\tau_1^\infty\rho,  \quad\sigma^\infty=\tau_1^\infty\tau_2^\infty.
\end{gather*}
Of course, $\{\tau_{11}, \dots,\tau_{1p},\rho\}$ satisfies all the conditions that ensure it can be realized by a real analytic submanifold.

We know that $\sigma^\infty$ does not have a unique normal form in the centralizer $  S$.
Therefore, we will choose a procedure that arrives at a unique formal normal form in $  S$. We show that this unique normal form is divergent; and hence
by \rt{ideal5} (iii) any normal form of  $\sigma$ that is  in the centralizer of $  S$ must diverge.

We now describe the procedure. For a formal mapping $F$, we have a unique decomposition
$$
F=NF+N^{\mathsf c}F, \quad NF\in \cL C(  S), \quad N^{\mathsf c}F\in\cL C^{\mathsf c}(  S).
$$
Set $\hat\sigma_{0}^\infty=\sigma^\infty$.   For $k=0,1,\ldots,$  we take a normalized polynomial map $\Phi_{k}\in \cL C_2^{\mathsf c}(  S)$
of degree less than $d_k$  such that $\sigma_{k}^\infty:=\Phi_k^{-1}\hat\sigma_{k}^\infty\Phi_{k}$ is normalized
up to degree $d_k-1$.   Specifically, we require that
\gan
\deg\Phi_k\leq d_k-1, \quad\Phi_k\in \cL C^{\mathsf c}(  S);\quad \quad N^c\sigma_k^\infty(\xi,\eta)=O({d_k}).
\end{gather*}
Take a normalized polynomial map $\Psi_{k+1}$   such that $\Psi_{k+1}$ and
 $\hat\sigma_{k+1}^\infty:=\Psi_{k+1}^{-1}\sigma_k^\infty\Psi_{k+1}$ satisfy
 \ga\label{ncsigm}
 \nonumber
\deg\Psi_{k+1}\leq 2d_k-1; \quad\Psi_{k+1}\in \cL C_2^{\mathsf c}(  S),\quad
N^c\hat\sigma_{k+1}^\infty=O(  2d_k).
\end{gather}
We can repeat this for $k=0,1,\ldots$. Thus we apply two sequences of normalization as follows
 \eq{} \nonumber
   \hat\sigma_{k+1}^\infty=\Psi_{k+1}^{-1}\circ\Phi_k^{-1}\cdots\Psi_1^{-1}
\circ \Phi_0^{-1}\circ \sigma^\infty\circ \Phi_0\circ\Psi_1\cdots\Phi_{k}\circ\Psi_{k+1}.
 \eeq
 We will show that
 $\Psi_{k+1}=I+O(d_k)$ and $\Phi_k=I+O(  2d_{k-1})$. This shows that  the
  sequence $\Phi_0\Psi_1\cdots\Phi_k\Psi_{k+1}$
  defines  a formal biholomorphic mapping $\Phi$ so that
 \eq{hsiginf}
  \hat\sigma^\infty:=\Phi^{-1}\sigma^\infty\Phi
 \eeq
 is in a normal form. Finally, we need to combine the above normalization with
the normalization for the unique normal form
 in \rl{fcfp}.  We will show that the unique normal form    diverges.

Let us recall previous results to show that $\Phi_k,\Psi_{k+1}$ are uniquely determined.
Set \ga\label{hsink}
\hat\sigma^{\infty}_k\colon\left\{
\begin{split} \xi'=\hat M^{(k)}(\xi\eta)\xi+\hat f^{(k)}(\xi,\eta), \\
\eta'=\hat N^{(k)}(\xi\eta)\eta+\hat g^{(k)}(\xi,\eta),
\end{split}\right.
\\
 (\hat f^{(k)},\hat g^{(k)})\in \cL C_2^{\mathsf c}(  S).
 \end{gather}
 Recall that $\hat\sigma_0
=\sigma^\infty$.  Assume that we have achieved
 \eq{2dk-1}
 (\hat f^{(k)},\hat g^{(k)})=O(2d_{k-1}).
  \eeq
  Here we take $d_{-1}=2$ so that \re{hsink}-\re{2dk-1} hold for $k=0$.
By \rp{ideal0}, there is  a unique normalized     mapping $\tilde\Phi_k$ that
transforms $\hat\sigma_k^\infty$ into a normal form.  We denote by $\Phi_k$
the truncated polynomial mapping of $\tilde\Phi_k$ of degree $d_k-1$. We write
\gan
\Phi_{k}\colon\xi'=\xi+U^{(k)}(\xi,\eta), \quad \eta'=\eta+V^{(k)}(\xi,\eta),\\
(U^{(k)},V^{(k)})=O(2),
\quad \deg (U^{(k)}, V^{(k)})\leq d_k-1.\end{gather*}
  By Corollary~\ref{finitever},  $\Phi_k$ satisfies
\ga
\nonumber
\sigma^\infty_k=\Phi_{k}^{-1}\hat\sigma^\infty_{k}\Phi_{k}\colon\left\{
\begin{split}\xi'=M^{(k)}(\xi\eta)\xi+f^{(k)}(\xi,\eta), \\
\eta'=N^{(k)}(\xi\eta)\eta+g^{(k)}(\xi,\eta),
\end{split}\right.
\\
 (f^{(k)},g^{(k)})\in \cL C_2^{\mathsf c}(  S), \quad \ord(f^{(k)},g^{(k)})\geq d_k.\label{fkgkd}
 \end{gather}
  In fact, by \re{ujpq+}-\re{vjpq} (or \re{mp-q}-\re{mp-q2}), we have
 \al \label{ujpq6}
U _{j,PQ}^{(k)}&=(\mu^{P-Q}-\mu_j)^{-1}\left\{\hat f ^{(k)}_{j,PQ}+\cL U _{j,PQ}(\del_{d-1},
\{\hat M^{(k)},\hat N^{(k)}\}_{[\f{d-1}{2}]};\{ \hat f^{(k)},\hat g^{(k)}\}_{d-1})\right\},\\
\label{vjpq6}
V^{(k)} _{j,QP}&=(\mu^{Q-P}-\mu_j^{-1})^{-1}\left\{\hat g^{(k)}_{j,QP}+\cL V _{j,QP}(\del_{d-1},
\{\hat M^{(k)},\hat N^{(k)}\}_{[\f{d-1}{2}]};\{ \hat f^{(k)},\hat g^{(k)}\}_{d-1})\right\},
\end{align}
for  $|P|+|Q|=d< d_k$ and $\mu^{P-Q}\neq\mu_j$.
 By \re{mpmpnsd}-\re{npnpnsd} (or \re{mpmp-}-\re{npnp}), we have
\al
\label{mpmpnsd6}
M_P^{(k)}&=\hat M^{(k)}_P+\cL M_P(\del_{2|P|-1},\{\hat M^{(k)},\hat N^{(k)}\}_{|P|-1}; \{\hat f^{(k)}, \hat g^{(k)}\}_{2|P|-1}), \\
N_P^{(k)}&=\hat N^{(k)}_P+\cL N_P(
\del_{2|P|-1},\{\hat M^{(k)},\hat N^{(k)}\}_{|P|-1}; \{\hat f^{(k)}, \hat g^{(k)}\}_{2|P|-1})  \label{npnpnsd6}
\end{align}
  for $2|P|-1< d_k$.
Recall that  ${\mathcal U}_{j,PQ}, {\mathcal V}_{j,QP}, {\mathcal M}_{j,P}$,
and ${\mathcal N}_{j,P}$  are universal polynomials in their variables. In  notation defined by Definition~\ref{notation},
\eq{uvmn=0}
\nonumber
\cL U_{j,PQ}({\sbt}\,; 0)= \cL V_{j,QP}(\sbt\,; 0)=0, \quad \cL M_{P}(\sbt\,; 0)=
\cL N_{P}(\sbt\,; 0)=0.
\eeq
We apply \re{ujpq6}-\re{vjpq6}   for $d<2d_{k-1}\leq d_k$
and \re{mpmpnsd6}-\re{npnpnsd6}  for $2|P|-1< 2d_{k-1}\leq d_k$ to obtain
\ga
\Phi_k-I=(U^{(k)}, V^{(k)})=O(2d_{k-1}),
 \label{phiki} \\ \label{mphmp}
 M_P^{(k)}=\hat M_P^{(k)}, \quad N_P^{(k)}=\hat N_P^{(k)}, \quad |P|\leq d_{k-1}.
 \end{gather}
  In fact, by Corollary~\ref{uniqueM}, the above holds for $|P|<2d_{k-1}-1$.

By \rl{mj0mj},  there is a unique normalized polynomial mapping
 \gan
 \Psi_{k+1}(\xi,\eta)=(\xi+\hat U^{(k+1)}(\xi,\eta), \eta+\hat V^{(k+1)}(\xi,\eta)),\\
 (\hat U^{(k+1)}, \hat V^{(k+1)})\in \cL C_2^{\mathsf c}(  S),\\
 (\hat  U^{(k+1)}, \hat V^{(k+1)})=O(2),\quad
  \deg (\hat U^{(k+1)},\hat V^{(k+1)})\leq 2d_k-1 \end{gather*}
such that
$\hat\sigma^\infty_{k+1}=
\Psi_{k+1}^{-1}\Phi_{k}^{-1}\sigma^\infty_{k}\Phi_{k}\Psi_{k+1}$ satisfies  the following:
\ga
\nonumber
\hat\sigma^\infty_{k+1}\colon\xi'=\hat M^{(k+1)}(\xi\eta)\xi+  \hat f^{(k+1)}, \quad
\eta'=\hat N^{(k+1)}(\xi\eta)\eta+  \hat g^{(k+1)},\\
(\hat f^{(k+1)},\hat g^{(k+1)})\in \cL C^{\mathsf c}_2(  S), \quad
\ord(\hat f^{(k+1)},\hat g^{(k+1)} )\geq2d_k.
\label{hfk1}
\nonumber
\end{gather}
 By     \re{ujpqcc}-\re{ujpqcc6}, we know that
\al
\label{ujpqc}
\hat U _{j,PQ}^{(k+1)}&=(\mu^{P-Q}-\mu_j)^{-1}\left\{f _{j,PQ}^{(k)}+\cL U _{j,PQ}^*(\del_{  \ell-1},
\{M^{(k)},N^{(k)}\}_{[\f{\ell-1}{2}]};\{ f^{(k)},g^{(k)}\}_{\ell-1})\right\},\\
\hat V _{j,QP}^{(k+1)}&=(\mu^{Q-P}-\mu_j^{-1})^{-1}\left\{g _{j,QP}^{(k)}+\cL V _{j,QP}^*(\del_{\ell-1},
\{M^{(k)},N^{(k)}\}_{[\f{\ell-1}{2}]};\{ f^{(k)},g^{(k)}\}_{\ell-1})\right\},
\label{ujpqc6} \end{align}
for $d_k\leq |P|+|Q|=\ell\leq 2d_k-1$ and $\mu^{P-Q}\neq\mu_j$.
Recall that  ${\mathcal U}^*_{j,PQ}$ and ${\mathcal V}^*_{j,QP}$  are universal polynomials in their variables. In  notation defined by Definition~\ref{notation},
$
\cL U_{j,PQ}^*(\cdot; 0)= \cL V_{j,QP}^*(\cdot; 0)=0$. Thus
\ga
\label{psik1}\Psi_{k+1}-I=(\hat U^{(k+1)}, \hat V^{(k+1)})=O(d_k), \\
\label{hUk+1}\hat U^{(k+1)}_{j,PQ}=\f{f^{(k)}_{j,PQ}}{\mu^{P-Q}-\mu_j},\quad
\hat V^{(k+1)}_{j,QP}=\f{g^{(k)}_{j,QP}}{\mu^{Q-P}-\mu_j^{-1}}, \quad |P|+|Q|=d_k.
\end{gather}
Here $ \mu^{P-Q}\neq
\mu_j.$
By \re{mjpp0}-\re{mjpp--}, we have
\al
\label{mjpp-2}
\hat M^{(k+1)}_{j,P'}&=M^{(k)}_{j,P'}, \quad |P'|<  d_k-1;\\
\hat M^{(k+1)}_{j,P'}&=M^{(k)}_{j,P}+\left\{Df^{(k)}_j(\xi,\eta)(\hat U^{(k+1)},\hat V^{(k+1)})\right\}_{(P+e_j)P}, \quad |P_k|=d_k-1;\\
\label{mjpp2}
 \hat M^{(k+1)}_{j,P}&=M^{(k)}_{j,P}+   \mu_j \left\{
 2(\hat U^{(k+1)}_j \hat V^{(k+1)}_j)_{PP}+(
 (\hat U^{(k+1)}_j)^2)_{(P+e_j)(P-e_j})  \right\}
 \\
 &\quad+\left\{Df^{(k)}_j(\xi,\eta)(\hat U^{(k+1)},\hat V^{(k+1)})\right\}_{(P+e_j)P}, \quad |P|=d_k.
\nonumber\end{align}
 As   stated in Corollary~\ref{uniqueM},
  the coefficients of $\hat M^{(k+1)}_j(\xi\eta)\xi_j$
 of degree $2d_k+1$ do not depend on the coefficients of $f^{(k)}, g^{(k)}$ of degree $\geq 2d_k$, provided
 $(f^{(k)},g^{(k)})=O(d_k)$  is in $\cL C_2^{\mathsf c}(S)$ as it is assumed.

 Next, we need to estimate the size of coefficients of $M^{(k)}$ 
  that appear
  in \re{mjpp-2}-\re{mjpp2}. Recall that we apply  two sequences of normalization. We have
 \eq{}
 \nonumber
  \hat\sigma_{k+1}^\infty=\Psi_{k+1}^{-1}\circ\Phi_k^{-1}\cdots\Psi_1^{-1}
\circ \Phi_0^{-1}\circ \sigma^\infty\circ \Phi_0\circ\Psi_1\cdots\Phi_{k}\circ\Psi_{k+1}.
 \eeq
 Thus, $M^{(k)}, N^{(k)}$ depend  only
 on $\sigma^\infty$, $\Phi_0, \Psi_1, \Phi_1, \ldots, \Psi_{k-1},\Phi_k$.

 Recall that if $u_1,\ldots, u_m$ are power series, then
 $\{u_1,\ldots, u_m\}_d$ denotes the set of their coefficients of degree at most $d$,
 and $|\{u_1,\ldots, u_m\}_d|$ denotes the sup norm.
 We  need some crude estimates on the
 growth of Taylor coefficients.
 If $F=I+f$ and
  $f=O(2)$ is a map in formal power series, then   \re{F-1P5}-\re{F-1GF} imply
 \al
 |\{F^{-1}\}_{m}|&\leq    (2+|\{f\}_{m}|)^{\ell_{m}},\\
 |\{G\circ F\}_{m}|&\leq   (2+|\{f,G\}_{m}|)^{\ell_{m}}, \nonumber \\
 |\{F^{-1}\circ G\circ F\}_m|&\leq    (2+|\{f,G\}_{m}|)^{\ell_{m}},
\label{FGPF} \end{align}
In general,  if $F_j$ are formal mappings of $\cc^n$ that are tangent to the identity, then
\eq{Lmk}\nonumber
|\{F_k^{-1}\cdots F_1^{-1}GF_1\cdots F_k\}_m|\leq
(2+|\{ F_1,\dots, F_k,G\}_{m}|)^{ \ell_{m,k}}, \quad 1\leq k<\infty
\eeq
In particular, if $F_j=I+O(j)$ for $j=1,\dots, m$, then
for any $k\geq|P| :=m$ we have
\gan
(F_k^{-1}\cdots F_1^{-1}GF_1\cdots F_k)_P=(F_{m}^{-1}\cdots F_1^{-1}GF_1\cdots F_{m})_P, \\
\label{Lmindk}\nonumber
|\{F_k^{-1}\cdots F_1^{-1}GF_1\cdots F_k\}_m|\leq
(2+|\{F_1,\dots, F_{m},G\}_{m}|)^{\ell_{m}'}, \quad 1\leq k\leq\infty.
\end{gather*}
We may take $\ell_m'$ by $\ell_m$, while $\ell_m$ depends only on $m$.  We have similar estimates for $F_k\cdots F_1GF_1^{-1}\cdots F_k^{-1}$.  Recall that $1/\sqrt 2< \la_j<e^{e/2}<4$.
  Using $\tau_{1j}^{ k }=\var_k\tau_{1j}^{ k-1 }\var_k^{-1}=\var_k\dots\var_1\tau_{1j}^0\var_1^{-1}\dots\var_k^{-1}$ and hence
  $\tau_1^{k}=\var_k\dots\var_1\tau_{1}^0\var_1^{-1}\dots\var_k^{-1}$ and $\sigma_k=\tau_1^k(\rho\tau_2^k\rho)$, we obtain
  $
  |\{\tau_1^{k}\}_m|\leq (2+|\{\tau_1^0,\var_1,\dots, \var_m\}_m|)^{\ell_m'}\leq 6^{\ell_m'} $
   and $|\{\sigma^{ k }\}_m|\leq (8^{\ell_m'})^{\ell_m}$. Thus  we obtain
  \eq{siginfP}
 |\{ \sigma^k\}_{m}| \leq   8^{\ell_m\ell_m'}, \quad
|\{\sigma^{\infty}\}_{m}|\leq     8^{\ell_m\ell_m'}.  \eeq
  Here we have used
\re{rest1}-\re{rest2}.

  For simplicity,  let $\del_i$ denote $\del_i(\mu)$.  Inductively, let us  show that   for $k=0,1, \ldots$,
  \ga
  \label{hmk1-6}
 |\{\hat M^{(k)},\hat N^{(k)}\}_{P}|\leq |\del_{d_{k-1}-1}|^{L_m}, \quad m=  2|P|+1<2d_{k-1} -1,
\\
|\{\hat\sigma^\infty_{k}\}_{PQ}|\leq  |\del_{  2d_{k-1}-1}|^{L_m}, \quad m=|P|+|Q| \geq 2d_{k-1}  -1.
\label{hsigke}\\
 |M^{(k)}_{j,P}|+|N^{(k)}_{j,P}|\leq |\del_{d_k-1}|^{L_m}, \quad m=2|P|+1,
\label{mkjpn} \\ \label{fjpqg+}
 |f_{j,PQ}^{(k)}|+|g^{(k)}_{j,QP}|
  \leq |\del_{  d_k-1}|^{L_m}, \quad m=|P|+|Q|\geq d_k.
 \end{gather}
  Note that the last inequalities are equivalent to $|\{\sigma_k^{\infty}\}_m|\leq |\del_{d_k-1}|^{L_m}$.
 Here and in what follows
$L_m$ does not depend on the choices of $\mu_j, d_k,h^{(k)}$ which  satisfy the initial conditions, i.e. $1/e\leq\mu_j\leq e^e$ and \re{rest1}-\re{rest2} but are arbitrary otherwise.   However, it suffices to find constants $L_{m,k}$ replacing $L_m$ and  depending on $k$ such that \re{hmk1-6}-\re{fjpqg+} hold. Indeed, by  \re{phiki} and \re{psik1} we have $\Psi_{k+1}=I+O(d_k)$ and $\Phi_k=I+O(2d_{k-1})$. Since $\hat\sigma^\infty_{k+1}=\Psi_{k+1}^{-1} \sigma^\infty_k \Psi_{k+1}$ and
$\sigma^\infty_{k}=\Phi_{k}^{-1}\hat\sigma^\infty_k\Phi_{k}$, then $$
\hat\sigma^\infty_{k+1} =\hat\sigma^\infty_k+O(2d_{k-1}), \quad
\sigma^\infty_{k+1}=\sigma^\infty_k+O(d_{k})$$
as $d_k\geq 2d_{k-1}$. Since $d_k$ increases  to $\infty$ with $k$, then \re{hmk1-6}-\re{fjpqg+} with $L_{m,k}$ in place of $L_m$ imply that they also hold for
\eq{minmax}\nonumber
L_m=\min_k\max\{L_{m,1},\dots, L_{m,k}\colon (\hat\sigma^\infty_{\ell} -\hat\sigma_k^\infty, \sigma^\infty_\ell-\hat\sigma_k^\infty)=O(m),\  \ell>k\}.
\eeq
Therefore, in the following  the dependence of $L_m$ on $k$ will not be indicated.
The   estimates \re{hmk1-6}-\re{hsigke}  hold trivially for  $\hat\sigma^\infty_0=\sigma^\infty$,  $k=0$ and $d_{-1}=2$ by
\re{sinf} and \re{siginfP}. Assuming \re{hmk1-6}-\re{hsigke}, we want to verify \re{mkjpn}-\re{fjpqg+}. We also want to verify \re{hmk1-6}-\re{hsigke} when $k$ is replaced by $k+1$.

%
The  $\Phi_k=I+(U^{(k)}, V^{(k)})$  is a polynomial mapping. Its  degree is  at most $d_k-1$ and   its coefficients are polynomials in
  $\{\hat\sigma_k\}_{d_k-1}$ and $\del_{d_k-1}$; see \re{ujpq6}-\re{vjpq6}.
  Hence
 \ga\label{fjpqg}
  |U^{(k)}_{j,PQ}|+|V^{(k)}_{j,QP}|\leq|\del_{  d_k-1}|^{L_m}, \quad m=|P|+|Q|.
 \end{gather}

 Applying \re{FGPF} to $\sigma_k^\infty=\Phi_k^{-1}\hat\sigma_k^\infty\Phi_k$,  we obtain \re{mkjpn}-\re{fjpqg+} from \re{hmk1-6}-\re{hsigke}.
 Here we use that fact that since $d_k  \geq 2d_{k-1}$, the small divisors in $\del_{2d_{k-1}-1}$ appear in $\del_{d_k-1}$ too.
To obtain \re{hmk1-6}-\re{hsigke} when $k$ is replaced by $k+1$,  we note that $\Psi_{k+1}$ is a polynomial map that has
 degree  at most $2d_k-1$ and
 the coefficients of degree $m$   bounded by $\del_{2d_k-1}^{L_m}$;
see \re{ujpqc}-\re{ujpqc6}.  This shows that
\eq{hUVk}
|\hat U^{(k+1)}_{j,PQ}|+|\hat V^{(k+1)}_{j,QP}|\leq  |\del_{2d_k-1}|^{L_m}, \quad |P|+|Q|=m.
\end{equation}
We then obtain \re{hmk1-6}-\re{hsigke} when $k$ is replaced by $k+1$ for $\hat\sigma_{k+1}^\infty$ by applying \re{FGPF} to  $\hat\sigma_{k+1}^\infty=\Psi_{k+1}^{-1}\sigma_k^\infty\Psi_{k+1}$ and by using \re{mkjpn}-\re{fjpqg+} for $\sigma_k^\infty$  and \re{hUVk} for $\Psi_{k+1}$.

%
%

Let us summarize the above computation for $\hat\sigma^\infty$ defined by \re{hsiginf}.
We know that $\hat\sigma^\infty$ is the unique power series such that $\hat\sigma^\infty-
\hat\sigma^\infty_k=O(d_k)$ for all $k$, and $\hat\sigma^\infty$ is a formal normal form of $\sigma^\infty$.
Let us write
\gan
\hat\sigma^{\infty}\colon\left\{
\begin{split} \xi'=\hat M^{\infty}(\xi\eta)\xi, \\
\eta'=\hat N^{\infty}(\xi\eta)\eta.
\end{split}\right.
 \end{gather*}
Let $|P|\leq d_k$. By \re{mphmp}, we get $\hat M^{(k+1)}_P=M^{(k+1)}_P$; by \re{mjpp-2} in which $k$ is replaced by $k+1$,
we get $\hat M^{(k+2)}_P=M^{(k+1)}_P$ as $|P|\leq d_k<d_{k+1}  -1$.  Therefore,
\eq{mkmdk}\hat M^{\infty}_P=\hat M^{(k+1)}_P, \quad |P|\leq d_k.
\eeq
For $|P|<d_k  -1$,  \re{mjpp-2} says that $\hat M^{(k+1)}_{j,P}=M_P^{(k)}$; by \re{mkjpn} that holds for any $P$,  we obtain
\ga\label{hmk1}
|\hat M^{\infty}_P|=|\hat M^{(k+1)}_{j,P}|\leq |\del_{d_k-1}|^{L_{m}},\quad m=2|P|+1, \quad |P|<d_k-1,  \\
\label{hmk1+}
|\hat M_{j,P}^{\infty}| \leq |\del_{d_k-1}|^{L_{m}}(1  +|\del_{d_k}|),\quad m=2|P|+1=2d_k-1.\end{gather}
 We have verified \re{hmk1-6}-\re{fjpqg+}. The sequence $L_m$ depend only on
\eq{Lmandd}
m=d_k+1, \quad d_k,  d_{k-1}, \dots, d_0, \quad d_j\geq 2d_{j-1}, \quad d_j>3.
\eeq

To obtain rapid increase of coefficients of $\hat M^{(k+1)}_{j,P}$, we want to use both small divisors
hidden in $\hat U^{(k)}_{j,PQ}$ and $\hat V^{(k)}_{j,QP}$ in \re{mjpp2}.   Therefore, if $M^{(k)}_{j,P}$  is already sufficiently
large for $|P|=d_k$ that will be specified later, we take $\var_k$ to be the identity, i.e. $\tau_1^k=\tau_1^{k-1}$. Otherwise, we need to achieve it by choosing
$$
\tau_1^k=\var_k\tau_1^{k-1}\var_k^{-1}.
$$
Therefore, we  examine the effect of a coordinate change by $\var_k$ on these coefficients.

Recall that we are in the elliptic case. We have
$
\rho(\xi,\eta)=(\ov\eta,\ov\xi)
$
and $\tau_2^k=\rho\tau_1^k\rho.$ Recall that
$$
\var_k\colon\xi_j'=(\xi  -h^{(k)}(\xi),\eta),\quad  \ord h^{(k)}=d_k>3.
$$
By a simple computation, we obtain   
\gan
\tau_1^{k}(\xi,\eta)=\tau_1^{k-1}(\xi,\eta)+(-h^{(k)}(\la\eta),\la^{-1}h^{(k)}(\xi))+O(|(\xi,\eta)|^{d_k+1}),\\
\tau_2^{k}(\xi,\eta)=\tau_2^{k-1}(\xi,\eta)+
(\la^{-1}\ov{h^{(k)}}(\eta),-\ov{h^{(k)}}(\la\xi))+O(|(\xi,\eta)|^{d_k+1}).
\end{gather*}
Then we have
\al
\label{sigkk}\sigma^{k}&=\sigma^{k-1}+(r^{(k)},s^{(k)})+O(d_k+1);\\
r^{(k)}(\xi,\eta)&=
 -\la\ov{h^{(k)}} (\la\xi)-h^{(k)}(\la^{2}\xi),
 \nonumber
  \\  
s^{(k)}(\xi,\eta)&=
 \la^{-2}\ov{h^{(k)}}(\eta)+\la^{-1}h^{(k)}(\la^{-1}\eta). 
\nonumber
\end{align}
 Since $\sigma^k$ converges to $\sigma^\infty$, from \re{sigkk} it follows that
\eq{siginfk}
\sigma^\infty=\sigma^{k-1}+(r^{(k)}, s^{(k)})+O(d_k+1).
\eeq
For $|P|+|Q|=d_k$, we have
\eq{}\begin{array}{l}
r^{(k)}_{j,PQ}=\left\{ -\la_j\ov{h_j^{(k)}}(\la\xi)-h_j^{(k)}(\la^{2}\xi)\right\}_{PQ},
\vspace{1ex}\\
s^{(k)}_{j,QP}=\left\{ \la_j^{-2}\ov{h_j^{(k)}} (\eta)+\la_j^{-1}h_j^{(k)}(\la^{-1}\eta)\right\}_{QP}.
\end{array}\nonumber
\eeq
We obtain
\begin{alignat}{4} \label{gjgj}
r^{(k)}_{j, P0} &=-\la^{P+e_j}\ov{h^{(k)}_{j,P}}-\la^{2P} h^{(k)}_{j,P},   &&
\\
s^{(k)}_{j,0P} &=\la_j^{-2}\ov{h^{(k)}_{j,P}}+\la^{-P-e_j} h^{(k)}_{j,P},\quad && |P|=d_k,\\
\label{gjgj3}
r^{(k)}_{j, PQ}&=s^{(k)}_{j,QP}=0,\quad                                                    & &|P|+|Q|=d_k,  \ Q\neq0.
\end{alignat}

The above computation is actually sufficient to construct a divergent normal form $\tilde\sigma
\in\cL C(  S)$. To show that all normal forms of $\sigma$ in $\cL C(  S)$ are divergent,
We need to related it to the normal form $\hat\sigma$ in \rt{ideal5}, which is unique. This requires us to
keep track of the small divisors in the normalization
procedure in the proof of \rl{fcfp}.

Recall that  $  F^{(k+1)}=\log \hat M^{(k+1)}$ is defined by
\eq{fjlogm}
F^{(k+1)}_j(\zeta)=\log(\mu_j^{-1}\hat M^{(k+1)}_j(\zeta))=\zeta_j+a_j^{(k+1)}(\zeta), \quad 1\leq j\leq 3.
\eeq
We also have $F^{\infty} =\log \hat M^{\infty}$ with $
F^{\infty}_j(\zeta)= \zeta_j+a_j^{\infty}(\zeta)$.
Then  by \re{mkmdk},
\eq{hajpa}
 a_{j,P}^{\infty}= a^{(k+1)}_{j,P}, \quad |P|\leq d_k.
\eeq
By \re{fjlogm}   and $\log(1+x)=x+\f{x^2}{2}+O(3)$, we have
\eq{fjpmuj}
a^{(k+1)}_{j,P}(\zeta)= \mu_j^{-1} \hat M^{(k+1)}_{j,P}+
\mu_j^{-1}\hat M_{j,P-e_j}^{(k+1)} +{\cL A}_{j,P}(\{ \hat M^{(k+1)}_j\}_{|P|-  2}), \quad|P|>1.
\eeq
By \re{hmk1}-\re{hmk1+}, we estimate the last two terms as follows
\ga\label{Ajph}
| {\cL A}_{j,P}(\{ \hat M^{(k+1)}_j\}_{|P|-  2})|\leq |\del_{d_k-1}|^{ L_m^*L_m}, \quad |P|=d_k, \quad m=2|P|+1,\\
|\hat M_{j,Q}^{(k+1)}|\leq  |\del_{d_k-1}|^{L_{m}}(1+ |\del_{d_k}|), \quad \quad m=2|Q|+1=2d_k-1.
\label{Ajph+}
 \end{gather}
 Here $L_m^*\geq 1$ is independent of $k$ and depends only on the degrees of the polynomials $\mathcal A_{j,P}$.   Recall from the formula \re{hfjqf}
  that  $F^{(k+1)}$, $F^\infty$ have the normal forms
$\hat F^{(k+1)}=I+\hat a^{(k+1)}$ and $\hat F^\infty =I+\hat a^\infty$, respectively.
The coefficients of $\hat a_{j,Q}^{(k+1)}$ and $\hat a_{j,Q}^{\infty}$
   are zero, except the ones given by
\aln
 \hat a^{(k+1)}_{j,Q}&=a^{(k+1)}_{j,Q}-  \{Da^{(k+1)}_{j} \cdot a^{(k+1)}\}_{Q}+ {\cL B}_{j,Q}(\{a^{(k+1)}\}_{|Q|-  2}),
 \\
  \hat a^{(\infty)}_{j,Q}&=a^{(\infty)}_{j,Q}-\{Da^{(\infty)}_j \cdot a^{(\infty)}\}_{Q}+ {\cL B}_{j,Q}(\{a^{(\infty)}\}_{|Q|- 2}),
 \end{align*}
  for $Q=(q_1,\ldots, q_p), q_j=0$, and $ |Q|>1$.
 Derived from the same normalization, the $ {\cL B}_{j,Q}$ in both formulae
 stands for the same polynomial   and independent of $k$.
 Hence $\hat a_{j,P}^{(\infty)}=\hat a_P^{(k+1)}$ for $|P|\leq d_k$, by \re{hajpa}.
Combining \re{mjpp2} and \re{mkmdk} yields
\al\label{uvprod}
\hat a^{\infty}_{3,Q} &= \hat a^{(k+1)}_{3,Q}= 2(\hat U^{(k+1)}_3 \hat V^{(k+1)}_3)_{QQ}  +((\hat U^{(k+1)}_3)^2)_{(Q+e_3)(Q-e_3)}
+\mu_3^{-1}M_{3,Q}^{ (k)}
\\
&\quad + \mu_3^{-1}\{Df^{( k)}_{3}(\xi,\eta)(\hat U^{(k+1)},\hat V^{(k+1)})\}_{(Q+e_3)P_k} 
+\cL {A}_{  Q}(\{
  \hat M^{(k+1)}\}_{|Q|-  2})\nonumber\\
&\quad  +\mu_3^{-1}\hat M^{(k+1)}_{Q-e_3}-\{Da^{(k+1)}_{j} \cdot a^{(k+1)}\}_{Q}.\nonumber
\end{align}
 The above
formula    holds for any $Q$ with $|Q|=d_k$.
To examine the effect
of small divisors,  we assume  that
 $$ P_k=(p_k,q_k,0), \quad |P_k|=d_k$$
are given by \rl{smallvec+}, so are $\mu_1,\mu_2$,
and $\mu_3$.  {\bf However, $P_k$ and $\mu$ depend on a sequence $L_m$ (to be renamed as $L_m'$) in \rl{smallvec+}. We will determine the sequence $L_m'$ and hence $P_k$ and $\mu_j$ later.}

Note that  the second term in \re{uvprod} is $0$ as the third component of $P_k-e_3$ is negative.
We apply the above computation to  the  $P_k.$
  Taking a subsequence of $P_k$ if necessary, we may assume that $d_k\geq 2d_{k-1}$ and $d_{k-1}>3$ for all $k\geq1$.
 The $4$ exceptional small divisors of height $2|P_k|+1$ in \re{delspk} are
$$
\mu^{P_k}-\mu_3, \quad \mu^{-P_k}-\mu_3^{-1}, \quad \mu^{2P_k-e_3}-\mu_3, \quad \mu^{-2P_k+e_3}-\mu_3^{-1}.
$$
 The last two  cannot show up in $\hat a_{3,P_k}^\infty$, since   their degree,  $2d_k+1$,  is larger than
 the degrees of Taylor coefficients in
$\hat a_{3,P_k}$ . 
We have $3$ products of   the   two exceptional small divisors of height
$2|P_k|+1$ and   degree $|P_k|$, which
are
$$
(\mu^{P_k}-\mu_3)(\mu^{-P_k}-\mu_3^{-1}), \quad
(\mu^{P_k}-\mu_3)(\mu^{P_k}-\mu_3), \quad
(\mu^{-P_k}-\mu_3^{-1})(\mu^{-P_k}-\mu_3^{-1}).
$$
The first product, but none of  the other two, appears in  $(\hat U^{(k+1)}_3 \hat V^{(k+1)}_3)_{P_kP_k}$.
The third  term and $f_3^{(k)}$ in $\hat a_{3,P_k}^\infty$
do not contain  exceptional small divisors
of degree $|P_k|=d_k>2d_{k-1}-1$.
  Since $f_3^{(k)}=O(d_k)$ by \re{fkgkd}, the  exceptional small divisors of height $2|P_k|+1$
can show up at most once in the fourth term of  $\hat a_{3,P_k}^\infty$.
Therefore, we arrive at
\aln\label{vacuouscase}
& \hat a^{\infty}_{3,P_k}=2\hat U^{(k+1)}_{3,P_k0}\hat V^{(k+1)}_{3, 0P_k}+\hat{\cL A}_{  P_k}^1(\del_{d_k-1},\f{1}{\mu^{P_k}-\mu_3}; \{f^{(k)},g^{(k)}\}_{d_k})\\
\nonumber
&   \qquad\  \ +\hat {\cL A}_{   P_k}^2(\del_{d_k-1};\{f^{(k)},g^{(k)}\}_{d_k})+\mu_3^{-1}M_{3,P_k}^{ (k)}+ \cL {A}_{   P_k}( \{
  \hat M^{(k+1)}\}_{|P_k|-  2})\\
  \nonumber
  &\qquad\ \  +\  \mu_3^{-1}M_{3,P_k-1}^{ (k+1)}-\{Da_3^{(k+1)}\cdot a^{(k+1)}\}_{P_k},\\
  \nonumber
&\hat{\cL A}_{k}^1(\del_{d_k-1},
  \f{1}{\mu^{P_k}-\mu_3}; \{f^{(k)},g^{(k)}\}_{d_k})
=(\hat U^{(k+1)}_{3,P_k0},\hat V^{(k+1)}_{3,0P_k})\cdot \hat {\cL A}_{  P_k}^3(\del_{d_k-1};\{f^{(k)},g^{(k)}\}_{d_k}).
\end{align*}
  Note that $\hat{\cL A}^i_{P_k}$ and $\cL A_{P_k}$ are polynomials independent of $k$.  Set
  $$m=2d_k+1.$$
  In the following we can increase the value of $L_m^*$     in \re{Ajph} or when it reappears  for a finite number of times such that the estimates involving $L_m^*$ are valid for all $k$. By \re{mkjpn} and \re{Ajph}, we obtain $|M_{3,P_k}^{ (k)}|+| \cL {A}^{  3}_{   P_k}( \hat M^{(k+1)}\}_{|P_k|-2})|\leq \del_{d_k-1}^{L_m^*L_m}$.   By \re{fjpmuj}, the smallest $|Q|$ for which $a_{i,Q}$ contains an exceptional small divisor in $\del_{d_k}$ is $2|Q|+1=2d_k-1$.  Now, $\{D a_3^{(k+1)}\cdot a^{(k+1)}\}_{P_k}$ is a linear combination of products of two terms and at most one of the two terms contains an exceptional small divisor; if  the both terms contain an exceptional small divisor,   one term is  $a_{3,Q'}^{(k+1)}$ with $2|Q'|+1\geq 2d_k-1$, while another   is   $a^{(k+1)}_{i,Q''}$ with $2|Q''|+1\geq2d_k-1$. (Here $Q'-e_i+Q''=P_k$ and the $i$th component of $Q'$ is positive.)  Then  $d_k=|P_k|=|Q'|+|Q''|-1\geq 2d_k-2$, a contradiction.  Therefore, by \re{fjpmuj}-\re{Ajph+},
we have
$$
|\{a_3^{(k+1)}\cdot a^{(k+1)}\}_{P_k}|\leq |\del_{d_k-1}|^{L_m^*L_m}(1+|\del_{d_k}|).
$$
By \re{Ajph+}, we also have
 $|M_{j,P_k-1}^{(k+1)}|\leq |\del_{d_k-1}|^{L_m}(1+|\del_{d_k}|)$.
Omitting  the arguments in the polynomial functions, we obtain from \re{fjpqg}-\re{hUVk},
and \re{mkmdk} that
\aln
|\hat{\cL A}_{P_k}^1|+|\hat{\cL A}_{P_k}^2|+|M_{3,P_k-1}^{(k+1)}|&+|M_{3,P_k}^{(k)}|+|\cL A_{P_k}|+  |\{Da_3^{(k+1)}\cdot a^{(k+1)}\}_{P_k}|
\\
&\leq\f{|  \del_{d_k-1}(\mu)|^{L_m^* L_{  m}}}{|\mu^{P_k}-\mu_3|},
\end{align*}
for  $m=2|P_k|+1$ and
a possibly larger $L_{m}$. 
  We remark that although each term in the inequality depends on the choices of the sequences
$\mu_i,d_j,h^{(\ell)}$, the $L_m$   does not
depend on the choices, provided that $\mu_j,d_k,h^{(i)}$ satisfy our initial conditions.
Therefore, we have
\eq{}
\nonumber
|\hat a^{\infty}_{3,P_k}|\geq2 |\hat U^{(k+1)}_{3,P_k0}\hat V^{(k+1)}_{3, 0P_k}|-| \del_{d_k-1}(\mu)|^{L_{2|P_k|+1}^* L_{2|P_k|+1}}|\mu^{P_k}-\mu_3|^{-1}.
\eeq
Recall that $\sigma^\infty_k=
\Phi_k^{-1}\Psi_{k-1}^{-1}\cdots\Phi_0^{-1}\sigma^\infty\Phi_0\Psi_1\cdots\Phi_k$.
Set
\eq{}
\nonumber
\tilde\sigma^\infty_k:=
\Phi_k^{-1}\Psi_{k-1}^{-1}\cdots\Phi_0^{-1}\sigma^{k-1}\Phi_0\Psi_1\cdots\Phi_k.
\eeq
By \re{siginfk}, we get
\eq{sigtsig}
\sigma^\infty_k=\tilde\sigma_k^\infty+(r^{(k)}, s^{(k)})+O(d_k+1).
\eeq
  Recall that  $\Phi_k$  depends only on coefficients of
$\hat\sigma^\infty_{k-1}=\Psi_{k-1}^{-1}\sigma^\infty_{k-2}\Psi_{k-1}$ of degree less than $d_k$, while $\Psi_{k-1}$
 depends only on coefficients of $\sigma_{k-1}^\infty=\Phi_{k-1}^{-1}\hat\sigma_{k-1}\Phi_{k-1}$ of
degree at most $2d_{k-1}-1$ which is less than $d_k$ too. Therefore,
 $\Phi_k, \Psi_{k-1}, \ldots, \Phi_0$   depend only on
coefficients of $\sigma^\infty$ of degree less than $d_k$. On the other hand, $\sigma^\infty=\sigma^{k-1}+O(d_k)$.
Therefore,   $\tilde\sigma_k^\infty$
depends only on  $\sigma^{k-1}$, and hence it depends only on $h^{(\ell)}$ for $\ell<k$. By \re{sigtsig},
we can express
\eq{fktfk}
f^{(k)}_{j,PQ}=\tilde f^{(k)}_{j,PQ}+r^{(k)}_{j,PQ}, \quad
g^{(k)}_{j,QP}=\tilde g^{(k)}_{j,QP}+s^{(k)}_{j,QP},\eeq
where $|P|+|Q|=d_k$ and  $\tilde f^{(k)}_{j,PQ}, \tilde g^{(k)}_{j,QP}$ depend only on
$h^{(\ell)}$ for $\ell<k$.
Collecting  \re{hUk+1},   \re{fktfk}, and \re{gjgj}-\re{gjgj3}, we obtain
\eq{}
\nonumber
|\hat a^{\infty}_{3,P_k}|\geq2 \f{|T_k
|}{|\mu^{P_k}-\mu_3||\mu^{-P_k}-\mu_3^{-1}|}-\f{| \del_{d_k-1}(\mu)|^{L_{2d_k+1}^* L_{2d_k+1}}}{|\mu^{P_k}-\mu_3|}
\eeq
with
\aln
&T_k =
(-\la^{P_k+e_3}\ov{h^{(k)}_{3,P_k}}-\la^{2P_k} h^{(k)}_{3,P_k} + \tilde f^{(k-1)}_{3,P_k0})
(\la_3^{-2}\ov{h^{(k)}_{3,P_k}}+\la^{-P_k-e_3} h^{(k)}_{3,P_k}+\tilde g^{(k-1)}_{3,0P_k})\\
&=-\la^{2P_k-2e_3}
(\la^{e_3-P_k}\ov{h^{(k)}_{3,P_k}}+ h^{(k)}_{3,P_k} -\la^{-2P_k}\tilde f^{(k-1)}_{3,  P_k0})
(\la^{e_3-P_k} h^{(k)}_{3,P_k}+\ov{h^{(k)}_{3,P_k}}+\la_3^2\tilde g^{(k-1)}_{3,0P_k}).\nonumber
\end{align*}
Set  $\tilde T_k(h^{(k)}_{3,P_k}):=-\la^{2e_3-2P_k}T_k$.
We are ready to choose $h^{(k)}_{3,P_k}$ to get a divergent normal form. We have
$|\la^{P_k-e_3}+1|\geq1$. 
 Then one of $|\tilde T_k(0)|, |\tilde T_k(1)|, |\tilde T_k(-1)|$ is at least $1/4$; otherwise,
 we would have
\aln
2|\la^{P_k-  e_3}+1|^2=|\tilde T_k(1)+\tilde T_k(-1)-2\tilde T_k(0)|<1,
\end{align*}
which is a contradiction.
This shows that by taking $h^{(k)}_{3,P_k}$ to be one of $0, 1,-1$, we have achieved
$$
|T_k|\geq\f{1}{4}\mu^{P_k -e_3}.
$$
Therefore,
\eq{divhf}
|\hat a^{\infty}_{3,P_k}|\geq\f{\mu^{  P_k-e_3} }{2|\mu^{P_k}-\mu_3||\mu^{-P_k}-\mu_3^{-1}|} -\f{| \del_{d_k-1}(\mu)|^{L_{2d_k+1}^* L_{2d_k+1}}}{|\mu^{P_k}-\mu_3|}.
\eeq
Recall that $\mu_3=e^e$. If $|\mu^{P_k}-\mu_3|<1$ then $1/2<\mu^{P_k-e_3}<2$. The above inequality implies
\eq{divhf+}
|\hat a^{\infty}_{3,P_k}|\geq \f{ \mu^{   2P_k} }{4|\mu^{P_k}-\mu_3|^2},
\eeq
provided
$$
|\mu^{P_k}-\mu_3|\leq  \f{1}{4} | \del_{d_k-1}(\mu)|^{-L_{2d_k+1}^*L_{2d_k+1}}, \quad |P_k|=d_k.
$$
For the last inequality  to hold, it suffices have
\eq{epqccc}
|\mu^{P_k}-\mu_3|\leq  | \del_{d_k-1}(\mu)|^{-L_{2d_k+1}^*L_{2d_k+1}-1}, \quad | \del_{d_k-1}(\mu)|^{-1}<1/4.
\eeq

 When \re{divhf+}, we still have \re{divhf}. Thus we have derived universal constants $L_{2d_k+1}, L_{2d_k+1}^*$ for any $P_k=(p_k,q_k,0)$ as long as $|P_k|=d_k>3$.
The sequence $L_m^*, L_m$ do not depend on the choice of $\la$ and they are independent of $k$; however it depends on $d_0,d_1,\dots, d_k$ as described in \re{Lmandd}.  Let us denote the constants $L_{2d_k+1}$, $L_{2d_k+1}^*$ in \re{epqccc}  respectively by $(L_{2d_k+1}(d_0,\dots, d_k)$, $ L_{2d_k+1}^*)(d_0,\dots, d_k)$. We now remove the dependence of $L_m$ on the partition $d_0,\dots, d_k$ and define $L_m$ for $m>7$ as follows.  For each $m>7$, define
\aln
\cL D_m&=\{(d_0,\dots, d_k)\colon 3<d_0\leq d_1/2\leq\dots\leq d_k/{2^k},  2d_k+1\leq m, k=0,1,\dots\}, \\
L_N'&=N+2\max\{(L_{ 2d_k+1} L_{2d_k+1}^*)(d_0,\dots, d_k)\colon (d_0,\dots, d_k)\in\cL D_{2N+1}\}.
\end{align*}
 Let us apply  \rl{smallvec+} to the sequence  $L_N'$.  Therefore, there exist $\mu$ and a sequence of $P_k=(p_k,q_k,0)$ satisfying  $ |\mu^{P_k}-\mu_3|\leq  (C\Delta^*(P_k))^{  L_{|P_k|}'}$.   Taking a subsequence if necessary, we may assume that $d_k=|P_k|\geq 2d_{k-1}$  and $d_k>3$. Thus
\aln
|\mu^{P_k}-\mu_3|&\leq  (C\Delta^*(P_k))^{  L_{|P_k|}'}\leq  (\Delta^*(P_k)^{1/2})^{L_{|P_k|}'}\\
&\leq (\delta_{d_k-1}(\mu))^{-L_{|P_k|}'/2}
\leq | \del_{d_k-1}(\mu)|^{ -L'_{  2d_k+1}}
\\ &\leq | \del_{d_k-1}(\mu)|^{ -L_{2D_K+1}^*(d_0,\dots, d_k)L_{  2d_k+1}(d_0,\dots, d_k)-1}, \end{align*}
which gives us \re{epqccc}.   Here the second inequality follows from       $C(\Delta^*(P_k))^{1/2}<1$ when $k$ is sufficiently large.
   The third inequality is obtained as follows.   The definition of $\Delta^*(P_k)$ and $|P_k|=d_k$
   imply that any small divisor
  in $\del_{d_k-1}(\mu)$ is contained in $\Delta^*(P_k)$. Also, $\Delta^*(P_k)<\mu_i^{-1}$
  for $i=1,2,3$ and $k$ sufficiently large.  Hence, $\Delta^*(P_k)
  \leq\del^{-1}_{d_k-1}(\mu)$, which gives us the third inequality.
  We have that $L_k\geq k$. From \re{divhf+} and \re{epqccc}
it follows that
$$
|\hat a_{3,P_k}^\infty|>\del_{d_{k-1}}^{d_k+1}(\mu)=\del_{d_{k-1}}^{|P_k|+1}(\mu),
$$
for $k$ sufficiently large.  As $\del_{d_k}(\mu)\to+\infty$, this shows that
the divergence of $ \hat F_3$   and  the
divergence of  the normal form $\hat\sigma$.

As mentioned earlier, \rt{ideal5} (iii) implies
that any normal form of  $\sigma$ that is  in the centralizer of $\hat S$ must diverge.
  \end{proof}

\setcounter{thm}{0}\setcounter{equation}{0}
\section{A unique formal normal form of  a real submanifold}\label{nfin}

 Recall that we consider submanifolds of which the complexifications  admit the maximum number of deck transformations.
   The deck transformations of $\pi_1$  are generated by
   $\{\tau_{i1},\ldots, \tau_{1p}\}$. 
 We also set  $\tau_{2j}=\rho\tau_{1j}\rho$.  Each of $\tau_{i1},\ldots, \tau_{ip}$ fixes a hypersurface and
  $\tau_i=\tau_{11}\cdots\tau_{1p}$ is the unique deck transformation of $\pi_i$ whose set of fixed points has
 the smallest dimension. We first normalize  the composition $\sigma=\tau_1\tau_2$.
 This normalization is reduced to   two normal form problems.
  In \rp{ideal0} we obtain  a transformation  $\Psi$  to transform $\tau_1,\tau_2,$
 and $\sigma$ into
 \begin{alignat*}{4}
 \tau_i^*\colon\xi_j'&=\Lambda_{ij}(\xi\eta)\eta_j,\quad &&\eta_j'=\Lambda_{ij}^{-1}(\xi\eta)\xi_j,\\
 \sigma^*\colon\xi_j'&=M_j(\xi\eta)\xi_j, \quad &&\eta_j'=M_j^{-1}(\xi\eta)\eta_j, \quad
 1\leq j\leq p.
 \end{alignat*}
 Here $\Lambda_{2j}=\Lambda_{1j}^{-1}$ and $M_j=\Lambda_{1j}^2$ are power series in the product
  $\zeta=(\xi_1\eta_1,\ldots, \xi_p\eta_p)$.
  We also  normalize the map $M\colon
 \zeta\to M(\zeta)$ by a transformation  $\var$ which preserves all coordinate
 hyperplanes. 
 This is the  second normal form problem, which is solved formally   in \rt{ideal5}
under the condition on the normal form of $\sigma$, namely,
 that  $\log \hat M$ is tangent to the identity. 
 This   gives us a map $\Psi_1$
 which transforms  $\tau_1,\tau_2$, and $\sigma$
 into $\hat\tau_1,\hat\tau_2,\hat\sigma$ of the above form
  where $\Lambda_{ij}$ and $M_j$
become  $\hat\Lambda_{ij},\hat M_j$.

In this section, we derive a {\bf unique formal normal form for
  $\{\tau_{11},\ldots, \tau_{1p},\rho\}$ under the above 
  condition on $\log\hat M$}.
In this     case,
we know from \rt{ideal5} that $ {\cL C}(\hat\sigma)$ consists of   only  $2^p$ dilatations
\eq{upsi}
R_\e\colon (\xi_j,\eta_j)\to(\e_j\xi_j,\e_j\eta_j), \quad \e_j=\pm1, \quad 1\leq j\leq p.
\eeq
  We will consider two cases. In the first case, we impose no restriction on the linear parts of $\{\tau_{ij}\}$ but
the coordinate changes are restricted to mappings that are tangent to the identity. The second is for the family $\{\tau_{ij}\}$
that arises from a higher order perturbation of a product quadric, while no restriction is imposed on the changes of coordinates.
We will   show that in both cases, if the normal form  of $\sigma$ can be achieved
by a convergent transformation,  the normal form of $\{\tau_{11}, \ldots,\tau_{1p},\rho\}$
can be achieved by a convergent transformation too.

We now restrict our real submanifolds to some classes. First, we assume that $\sigma$ and $\tau_1,\tau_2$
are already in the normal form $\hat\sigma$ and $\hat\tau_{1}, \hat\tau_2$ such that
 \ga\label{htaix}
 \hat\tau_i\colon\xi'=\hat\Lambda_{i}(\xi\eta)\eta, \quad \eta'=\hat\Lambda_i(\xi\eta)^{-1}\xi, \quad\hat\Lambda_{2}=\hat\Lambda_1^{-1},\\
\label{hsixi} \hat\sigma\colon \xi'=\hat M(\xi\eta)\xi,\quad \eta'=\hat M(\xi\eta)^{-1}\eta, \quad \hat M=\hat\Lambda_1^2.
\end{gather}

Let us start with the general situation   without imposing the restriction on  the linear part of $\log M$.
Assume that
$\hat\sigma$ and $\hat\tau_i$ are in the above forms. We want to describe $\{\tau_{1j},\rho\}$.  Let us start with the linear normal forms described in  \rl{unsol} or in \rp{2tnorm}.
Recall that $\mathbf Z_j=\diag(1,\ldots, -1,\ldots, 1)$ with $-1$ at the $(p+j)$-th place, and $\mathbf
Z:=\mathbf Z_1\cdots \mathbf Z_p$.   Let $Z_j$ (resp. $Z$) be the linear transformation with the matrix $\mathbf Z_j$
(resp. $\mathbf Z$).   We also use notation
\ga
\mathbf B_*=\begin{pmatrix}
\mathbf I & \mathbf 0 \\
\mathbf 0 &\mathbf  B
\end{pmatrix},\quad
\mathbf E_{\mathbf{\hat{ \Lambda}}_i}=\begin{pmatrix}
\mathbf I &\mathbf {\hat\Lambda}_i  \\
-\mathbf {\hat\Lambda}_i^{-1} &\mathbf  I
\end{pmatrix}.
\label{t1jd+l}
\end{gather}
Here $\mathbf B$, as well as $\mathbf {\hat\Lambda}_i$   given by \re{htaix},   is a non-singular complex $(p\times p)$ matrix. Assume that $\mathbf B_1$ and $\mathbf B_2$ are invertible $p\times p$ matrices. Define
\eq{bise}
(B_i)_*\colon \begin{pmatrix}\xi\\ \eta\end{pmatrix}\to
 (\mathbf B_i)_* \begin{pmatrix}\xi\\ \eta\end{pmatrix}, \quad
  E_{\mathbf {\hat\Lambda}_i}\colon \begin{pmatrix}\xi\\ \eta\end{pmatrix}\to
\begin{pmatrix}
\mathbf I &\mathbf {\hat\Lambda}_i(\xi\eta)  \\
-\mathbf {\hat\Lambda}_i^{-1}(\xi\eta) &\mathbf  I
\end{pmatrix} \begin{pmatrix}\xi\\ \eta\end{pmatrix}.
\eeq
  Let us assume that
in suitable linear coordinates,   the linear parts $L\tau_{ij}=T_{ij}$ of two families
of involutions $\{\tau_{i1}, \ldots, \tau_{ip}\}$ for $i=1,2$ are given by
\ga\label{ltij}
{ T}_{ij}:={ E}_{\mathbf{\Lambda}_i, \mathbf  B_i}\circ{ Z}_j\circ{ E}_{\mathbf{\Lambda}_i, \mathbf  B_i}^{-1}, \\
{ E}_{\mathbf{\Lambda}_i, \mathbf  B_i}:={ E}_{\mathbf{\Lambda}_i}\circ({ B}_i)_*, \quad \mathbf{\Lambda}_i:=\hat{\mathbf\Lambda}_i(0).
\end{gather}
Note that $(B_i)_*$ commutes with $Z$.   Also,
$E_{\mathbf {\hat\Lambda}_i}\circ\hat\tau_i=Z\circ E_{\mathbf {\hat\Lambda}_i}$.
 We have the decomposition
\ga
\label{t1jdl}
\hat\tau_i=\hat\tau_{i1}\cdots\hat\tau_{ip},\\
\quad
{ E}_{\hat{\mathbf{\Lambda}}_i, \mathbf  B_i}:={ E}_{\hat{\mathbf{\Lambda}}_i}\circ({ B}_i)_*, \quad \hat \tau_{ij}:=  { E}_{\hat{\mathbf{\Lambda}}_i, \mathbf  B_i}\circ Z_j\circ
{ E}_{\hat{\mathbf{\Lambda}}_i, \mathbf  B_i}^{-1}.\label{tPtPhat}
\end{gather}
 As before, we assume that   $S$ is non resonant.
 For real submanifolds, we still impose the reality condition $\tau_{2j}=\rho\tau_{1j}\rho$ where $\rho$ is given by \re{57rhoz}.
The following lemma  describes  a way to classify all involutions $\{\tau_{11}, \ldots, \tau_{1p}, \rho\}$ provided that  $\sigma$ is in a normal form.
\begin{lemma}\label{clas}  Let $\{\tau_{1j}\}$ and
$\{\tau_{2j}\}$
be two families  of formal holomorphic commuting involutions.
Let $\tau_i=\tau_{i1}\cdots\tau_{ip}$ and $\sigma=\tau_1\tau_2$. Suppose that
\gan
\tau_{i}=\hat\tau_i\colon\xi_j'
=\hat\Lambda_{ij}(\xi\eta)\eta_j, \quad \eta_j'=\hat\Lambda_{ij}(\xi\eta)^{-1}\xi_j;\\
\sigma=\hat\sigma\colon\xi_j'=\hat
M_j(\xi\eta)\xi_j, \quad \eta_j'=\hat M_j(\xi\eta)^{-1}\eta_j
\end{gather*}
with $\hat M_j=\hat\Lambda_{1j}^2$ and $\hat M_j(0)=\mu_j$. Suppose that $\mu_1,\dots, \mu_p,\mu_1^{-1}, \dots, \mu_p^{-1}$ satisfy the non-resonant condition \rea{muqn1}. 
 Assume further that the linear parts $T_{ij}$ of $\tau_{ij}$ are given by \rea{ltij}.
 Then we have the following~$:$
\bppp
\item For $i=1,2$
there exists   $\Phi_{i}\in\cL C(\hat{\tau}_{i})$, tangent to
the identity,
 such that
$
\Phi_{i}^{-1}\tau_{ij}\Phi_{i}=\hat \tau_{ij}
$
for $1\leq j\leq p$.
\item
Let $\{\tilde\tau_{1j}\}$ and
$\{\tilde\tau_{2j}\}$
be two families  of formal holomorphic commuting involutions.
Suppose that $\tilde\tau_{i}=\hat\tau_i$ and $\tilde\sigma=\hat\sigma$ and
$
\tilde\Phi_{i}^{-1}\tilde\tau_{ij}\tilde\Phi_{i}=\widehat{\tilde\tau}_{ij}$ with $\tilde\Phi_{i}\in\cL C(\hat\tau_i)
$ being tangent to the identity and
\ga
\label{t1jdls}
\nonumber
\widehat{\tilde\tau}_{ij}= E_{{\mathbf{\hat\Lambda}}_i,\mathbf {\tilde B}_i}
 \circ Z_j\circ E_{{\mathbf{\hat\Lambda}}_i,\mathbf {\tilde B}_i}
^{-1}.
\end{gather}
  Here for $i=1,2$, the matrix $\mathbf {\tilde B}_i$ is non-singular.
Then
\ga
\Upsilon ^{-1}\tau_{ij}\Upsilon =\tilde\tau_{i\nu_i(j)},\quad i=1,2,\ j=1,\dots, p
\label{bs-1}
\nonumber
\end{gather}
 if and only if 
there exist $\Upsilon\in\cL C(\hat{\tau}_{1},\hat{\tau}_{2})$ and    $\Upsilon_{i}\in\cL C(\hat\tau_i)$ such
that
\ga\label{tpup}
\tilde\Phi_{i}=\Upsilon ^{-1}\circ\Phi_{i}\circ\Upsilon_{i},\quad i=1,2,\\
\label{el1-}
\nonumber
\Upsilon_{i}^{-1}\hat\tau_{ij}\Upsilon_{i}=\widehat {\tilde\tau}_{i\nu_i(j)},\quad 1\leq j\leq p.
\end{gather}
Here each $\nu_i$ is a permutation of $\{1,\ldots, p\}$.
\item
Assume further that $\tau_{2j}=\rho\tau_{1j}\rho$ with $\rho$ being  defined by  \rea{57rhoz}.
Define $\hat \tau_{1j}$ by \rea{t1jdl} and  let
$
\hat\tau_{2j}\colonequals\rho\hat\tau_{1j}\rho.
$
Then
we can choose $\Phi_{2}=\rho\Phi_{1}\rho$ for $(i)$. Suppose that $\tilde\Phi_2=\rho\tilde\Phi_1\rho$ where $\tilde\Phi_{1}$ is as in $(ii)$. Then $\{\tilde\tau_{1j},\rho\}$ is equivalent to
$\{\tau_{1j},\rho\}$  if and only if
there exist $\Upsilon_{i}$,   $\nu_i$ with $\nu_2=\nu_1$,  and $\Upsilon$ satisfying the
conditions in $(ii)$ and
 $\Upsilon_{2}=\rho\Upsilon_{1}\rho$. The latter
 implies that $\Upsilon \rho=\rho\Upsilon$.
\eppp
\end{lemma}
\begin{proof}
(i)
Note that $\hat\tau_{ij}$ is conjugate to
$Z_j$ via the   map
$E_{\boldsymbol{\hat{\Lambda}_i}, \mathbf B_i}$. Fix $i$.  Each  $\hat\tau_{ij}$
is an involution
and its set of fixed-point is a hypersurface. Furthermore,  $\fix(\tau_{11}), \ldots, \fix(\tau_{1p})$
intersect
transversally at the origin.  By   \cite[Lemma 2.4]{GS15},  there exists a formal mapping $\psi_i$
 such that
$\psi_i^{-1}\tau_{ij}\psi_i=L\tau_{ij}$.  Now  $L\psi_i$ commutes with $L\tau_{ij}$,
Replacing $\psi_i$ by $\psi_i(L\psi_i)^{-1}$, we may assume that $\psi_i$ is tangent
to the identity. We also find a formal mapping $\hat\psi_i$, which is tangent to the identity,  such
that $\hat\psi^{-1}_i\hat\tau_{ij}\hat\psi_i=L\hat\tau_{ij}=L\tau_{ij}$. Then  
 $\Phi_{1}= \psi_i\hat\psi_i^{-1}$ fulfills the requirements.

(ii) Suppose that
\eq{tijp}\nonumber
\tau_{ij}=\Phi_{i}\hat \tau_{ij}\Phi_{i}^{-1},\quad
\tilde\tau_{ij}=\tilde\Phi_{i}\widehat {\tilde\tau}_{ij}\tilde\Phi_{i}^{-1}.
\eeq
Assume that there is a formal biholomorphic mapping $\Upsilon$ that transforms $\{\tau_{ij}\}$ into $\{\tau_{ij}\}$
for $i=1,2$. Then
\eq{UtUt}
\Upsilon ^{-1}\tau_{ij}\Upsilon =\tilde\tau_{i\nu_i(j)}, \quad j=1,\ldots, p,\  i=1,2.
\eeq
Here $\nu_i$ is a permutation of $\{1,\ldots, p\}$. Then
\eq{upsh}
\hat\tau_i\Upsilon= \Upsilon \hat\tau_i, \quad \hat\sigma\Upsilon =\Upsilon \hat\sigma.
\eeq
  Set
$\Upsilon_{i}\colonequals\Phi_{i}^{-1}\Upsilon  \tilde \Phi_{i}$. We obtain
\ga\label{el1}
\Upsilon_{i}^{-1}\hat\tau_{ij}\Upsilon_{i}=\widehat {\tilde\tau}_{i\nu_i(j)},\quad 1\leq j\leq p,\\
\label{pitp}
\tilde\Phi_{i}= \Upsilon ^{-1}\Phi_{i}\Upsilon_{i}, \quad i=1,2.
\end{gather}
Conversely, assume that \re{upsh}-\re{pitp}
 are valid. 
 Then \re{UtUt} holds as
\ga
\label{el2}
\nonumber
\Upsilon ^{-1}\tau_{ij}\Upsilon =\Upsilon ^{-1}\Phi_{i}
\hat\tau_{ij}\Phi_{i}^{-1}\Upsilon = \tilde\Phi_{i}\Upsilon_{i}^{-1}
\hat \tau_{ij} \Upsilon_{i}\tilde\Phi_{i}^{-1} =\tilde\tau_{\nu_i(j)}. 
\end{gather}


(iii) Assume that we have the reality assumption
 $\tau_{2j}=\rho\tau_{1j}\rho$ and
$\tilde\tau_{2j}=\rho\tilde\tau_{1j}\rho$. As before,
we take $\Phi_1$,  tangent to the identity, such that  $\tau_{1j}=\Phi_{1}\hat \tau_{1j}\Phi_{1}^{-1}$. Let $\Phi_{2}
=\rho\Phi_{1}\rho$. By $\hat \tau_{2j}=\rho \hat \tau_{1j}\rho$, we get $\tau_{2j}=\rho\tau_{1j}\rho=\Phi_{2}\hat \tau_{2j}\Phi_{2}^{-1}$
  for $\nu_2=\nu_1$.
Suppose that $\tilde\Phi_{i}$ associated with
 $\tilde \tau_{1j}$ and $\rho$ satisfy the analogous properties.
Suppose that $\Upsilon ^{-1}\tau_{ij}\Upsilon =\tilde \tau _{i\nu_i(j)}$ with $\nu_2=\nu_1$,
 and $\Upsilon \rho=\rho\Upsilon $. Letting
$\Upsilon_{1}= \Phi_{1}^{-1}\Upsilon \tilde\Phi_{1}$  we get
$\Upsilon_{2}=\rho\Upsilon_{1}\rho$. Conversely, if   $\Upsilon_{1}$ and $
\Upsilon_{2}$ satisfy 
 $\Upsilon_{2}=\rho\Upsilon_{1}\rho$, then
$$\rho\Upsilon \rho=\rho\Phi_{1}\Upsilon_{1}\tilde\Phi_{1}^{-1}\rho=
\Phi_{2}\Upsilon_{2}\tilde\Phi_{2}^{-1}=\Upsilon .$$
This shows that $\Upsilon $ satisfies the reality condition.
\end{proof}

  Now we assume that 
    $\hat F=\log\hat M$ is 
  tangent to the identity and is in the 
      normal form
   \re{hfj1}. Recall the latter means that
  the $j$th component of $\hat F-I$ is independent of the $j$th variable.
We assume that the linear part $T_{ij}$ of $\tau_{ij}$ are given by \re{ltij}, where the non-singular matrix $\mathbf B$ is arbitrary.
  As mentioned earlier in this section,  the group of formal biholomorphisms that preserve  the form of $\hat\sigma$ consists of only linear involutions
$R_\e$ defined by \re{upsi}.  This   restricts  the holomorphic equivalence classes of the quadratic parts of $M$.
By \rp{2tnorm}, such quadrics are classified
by a more restricted equivalence relation,   namely,
$(\mathbf {\tilde B_1},\mathbf {\tilde B}_2)\sim (\mathbf B_1,\mathbf B_2)$, if and only if
$$ 
\tilde {\mathbf B}_i= ( \diag \mathbf a)^{-1}\mathbf B_i\diag_{  \nu_i}\mathbf d, \quad i=1,2.
$$ 
 To deal with a general situation, let us assume for the moment that $\mathbf B_1,\mathbf B_2$ are arbitrary invertible matrices.

Using the normal form $\{\hat\tau_{1},\hat\tau_2\}$ and the matrices $\mathbf B_1,\mathbf B_2$, we first decompose
$\hat\tau_{i}=\hat\tau_{11}\cdots\hat\tau_{1p}$. By \rl{clas} (i), we then find $\Phi_i$   such
that
\eq{tauPh}
\nonumber
\tau_{ij}=\Phi_i\hat\tau_{ij} \Phi_i^{-1}, \quad 1\leq j\leq p.
\eeq
For each $i$,  $\Phi_i$  commutes with $\hat\tau_i$. It is within this family of  $\{\mathbf B_i, \Phi_i; i=1,2\}$  with $\Phi_i\in\cL C(\hat\tau_i)$
for $i=1,2$ that we will find a normal form for $\{\tau_{ij}\}$.    When restricted  to $\tau_{2j}=\rho\tau_{1j}\rho$,  the classification
 of   the real submanifolds is within the family of  $\{\tau_{1j}, \rho\}$ as described  in \rl{clas} (iii).

 From \rl{clas} (ii),    the equivalence relation on $\cL C(\hat\tau_i)$ is given by
\ga
\label{pitp+}
\nonumber
\tilde\Phi_i= \Upsilon^{-1} \Phi_i\Upsilon_i, \quad i=1,2.
\end{gather}
Here $\Upsilon_i$ and $\Upsilon$   satisfy
\eq{el1+}
\nonumber
\Upsilon_i^{-1}\hat\tau_{ij}\Upsilon_i=\hat \tau_{i\nu_i(j)},\quad 1\leq j\leq p;
\quad\Upsilon^{-1}\hat\tau_i\Upsilon=\hat\tau_i,\quad i=1,2.
\eeq

We now construct  a normal form for $\{\tau_{ij}\}$ within the above family.
Let us first use the centralizer of $\cL C^{\mathsf c}(Z_1,\ldots, Z_p)$, described in \rl{lehphi},
to define the
complement of the
centralizer of the family of non-linear commuting involutions
$\{\hat\tau_{11}, \ldots, \hat\tau_{1p}\}$.
Recall that
the mappings  $E_{\mathbf{\hat\Lambda}_i}$ and $(B_i)_*$ are defined by   \re{bise}.
According to \rl{lehphi}, we have the following.
\begin{lemma}\label{cnnl} Let $i=1$,  $2$.
Let $\{\hat\tau_{i1},\ldots, \hat\tau_{ip}\}$ be given by \rea{tPtPhat}.
Then \gan
{\cL C}(\hat\tau_{i1}, \ldots, \hat\tau_{ip})
=\left\{{ E}_{\hat{\mathbf{\Lambda}}_i, \mathbf  B_i}\circ \phi_0\circ { E}_{\hat{\mathbf{\Lambda}}_i, \mathbf  B_i}^{-1}\colon
\phi_0\in{\cL C}(Z_1, \ldots,  Z_p)\right\},\\
{\cL C}(\hat\tau_{i})
=\left\{{ E}_{\hat{\mathbf{\Lambda}}_i, \mathbf  B_i}\circ \phi_0\circ { E}_{\hat{\mathbf{\Lambda}}_i, \mathbf  B_i}^{-1}\colon
\phi_0\in{\cL C}(Z)\right\}.  
\end{gather*}
 Set
\gan
{\cL C}^{\mathsf c}(\hat\tau_{i1},\ldots, \hat\tau_{ip})
\colonequals\left\{{ E}_{\hat{\mathbf{\Lambda}}_i, \mathbf  B_i} \circ\phi_1 \circ{ E}_{\hat{\mathbf{\Lambda}}_i, \mathbf  B_i}^{-1}\colon
\phi_1\in{\cL C}^{\mathsf c}(Z_1,\ldots, Z_p)\right\}.\label{ccht}
\end{gather*}
Each formal biholomorphic mapping $\psi$ admits a unique decomposition $\psi_1\psi_0^{-1}$
with $$\psi_1\in{\cL C}^{\mathsf c}(\hat\tau_{i1}, \ldots, \hat\tau_{ip}),\quad \psi_0\in
{\cL C}(\hat\tau_{i1}, \ldots, \hat\tau_{ip}).$$ If $\hat\tau_{ij}$ and $\psi$
are convergent, then $\psi_0,\psi_1$ are convergent.
  Assume further that
  $\hat\tau_{2j}=\rho\hat\tau_{1j}\rho$ with $\rho$ being
given by \rea{eqrh}.    
Then define $\cL C^{\mathsf c}(\hat\tau_{21}, \ldots, \hat\tau_{2p})=\{\rho\phi_1\rho\colon\phi_1\in \cL C^{\mathsf c}(\hat\tau_{11}, \ldots, \hat\tau_{1p})\}$.
\end{lemma}

\begin{prop}\label{ideal6}
Let $\hat\tau_i,\hat\sigma$  be given by \rea{htaix}-\rea{hsixi} in which  $\log\hat M$ is in
 the 
  formal normal form \rea{hfj1}.
 Let $\{\hat\tau_{ij}\}$ be given by \rea{tPtPhat}.
 Suppose that
 \ga\label{tauPhp}
\tau_{ij}=\Phi_i\hat\tau_{ij} \Phi_i^{-1}, \quad\tilde \tau_{ij}=\tilde\Phi_i\hat\tau_{ij}\tilde \Phi_i^{-1}
 \quad 1\leq j\leq p,\\
 \Phi_i\in\cL C(\hat\tau_i),\quad  \tilde\Phi\in \cL C(\hat\tau_i), \quad
 \tilde\Phi_i'(0)=\Phi_i'(0)= \mathbf I, \quad i=1,2.
\end{gather}
Then   $\{\Upsilon^{-1}\tau_{ij}\Upsilon\}=\{\tilde\tau_{ij}\}$
for $i=1,2$ and for some invertible  $\Upsilon\in {\mathcal C}(\hat \tau_1,\hat\tau_2)$,
if and only if there exist formal biholomorphisms  $\Upsilon, \Upsilon_1^*,\Upsilon_2^*$ such that
\ga\label{quaeqp}
\Upsilon^{-1}\circ( B_i)_*\circ Z_j\circ
(B_i)_*^{-1}\circ   \Upsilon=( B_i)_*\circ Z_{\nu_i(j)}\circ
(B_i)_*^{-1},\\
 \label{tpup+-}
\tilde \Phi_i= \Upsilon^{-1}\Phi_i\Upsilon_i^*\Upsilon,  \quad \Upsilon_i^*\in\cL C(\hat\tau_{i1},\ldots,\hat\tau_{ip}), \quad i=1,2,
\\
\Upsilon\hat\sigma\Upsilon^{-1}=\hat\sigma,
\label{el1-+}
\end{gather}
where each $\nu_i$ is a permutation of $\{1, \ldots, p\}$.  Assume further that
 $\hat\tau_{2j}=\rho\hat\tau_{1j}\rho$ and $\Phi_2=\rho\Phi_1\rho$
and $\tilde\Phi_2=\rho\tilde\Phi_1\rho$.  We   can take
 $\Upsilon_2^*=\rho\Upsilon_1^*\rho$   and $\nu_2=\nu_1$, if additionally
 \eq{Uprh}\nonumber
 \Upsilon\rho=\rho\Upsilon.
 \eeq
 \end{prop}
\begin{proof} Recall that
\eq{tijp+}
\nonumber
\tau_{ij}=\Phi_i\hat \tau_{ij}\Phi_i^{-1},\quad\Phi_i\in \cL C(\hat\tau_i); \qquad
\tilde\tau_{ij}=\tilde\Phi_i\widehat { \tau}_{ij}\tilde\Phi_i^{-1}, \quad\tilde\Phi_i\in \cL C(\hat\tau_i).
\eeq
Suppose that
\eq{UtUt+}
\Upsilon^{-1}\tau_{ij}\Upsilon=\tilde \tau_{i\nu_i(j)}, \quad j=1,\ldots, p,\  i=1,2.
\eeq
By \rl{clas}, there are invertible $\Upsilon_i$ such that
\ga\label{el1++}
\Upsilon_i^{-1}\hat\tau_{ij}\Upsilon_i=\widehat { \tau}_{i\nu_i(j)},\quad 1\leq j\leq p,\\
\label{tpup+}
\tilde\Phi_i=\Upsilon^{-1}\circ\Phi_i\circ\Upsilon_i,\quad i=1,2. 
\end{gather}
Let us simplify the equivalence relation.
  By \rt{ideal5}, $\cL C(\hat\tau_1,\hat\tau_2)$
consists of $2^p$ dilations $\Upsilon$ of the form
$(\xi,\eta)\to(a \xi,a\eta)$ with $a_j=\pm1$.
Since $\Phi_i,\tilde\Phi_i$ are tangent to the identity, then $D\Upsilon_i(0)$ is diagonal because
\eq{}
\nonumber
L\Upsilon_i=\Upsilon.
\eeq
   Clearly, $\Upsilon$ commutes with each  non-linear
transformation $  E_{\boldsymbol{\hat\Lambda_i}}$.    
Simplifying the linear parts of both sides of  \re{el1++}, we get
\eq{quaeq}
\Upsilon^{-1}\circ(( B_i)_*\circ Z_j\circ
(B_i)_*^{-1})\circ  \Upsilon=(   B_i)_*\circ Z_{\nu_i(j)}\circ
(  B_i)_*^{-1}.\eeq
 From the commutativity of $\Upsilon$ and $E_{\hat{\Lambda}_i}$ again and  the above identity, it follows that
\eq{LUht}
\Upsilon^{-1}\circ\hat\tau_{ij}\circ \Upsilon=\widehat{\tau}_{\nu_i(j)}, \quad j=1,\ldots, p,\  i=1,2.
\eeq
Using \re{tauPhp} and \re{LUht}, we can rewrite \re{UtUt+} as
\eq{}
\nonumber
\Upsilon^{-1}\Phi_i\hat\tau_{ij}\Phi_i^{-1}\Upsilon= \tilde\tau_{i\nu_i(j)}=\tilde\Phi_i\Upsilon^{-1} \hat{\tau}_{ij}\Upsilon\tilde\Phi_i^{-1}.
\eeq
It is equivalent to  $\Upsilon_i^*\hat\tau_{ij}=\hat\tau_{ij}\Upsilon_i^*$, where  we define
\eq{dfUp}\nonumber
\Upsilon_i^*:=\Phi_i^{-1}\Upsilon\tilde\Phi_i   \Upsilon^{-1}.
\eeq
 Therefore, by \re{tpup},  in
$\cL C(\hat\tau_i)$, $\tilde\Phi_i $ and $\Phi_i $ are equivalent, if and only if
\eq{tppi} \nonumber
\tilde \Phi_i= \Upsilon^{-1}\Phi_i\Upsilon_i^*\Upsilon,  \quad \Upsilon_i^*\in\cL C(\hat\tau_{i1},\ldots,\hat\tau_{ip}), \quad i=1,2.
\eeq
Conversely,  if $\Upsilon_i^*, \Upsilon$ satisfy \re{quaeqp}-\re{el1-+}, we take $\Upsilon_i=\Upsilon_i^*\Upsilon$ to get   \re{tpup+} by \re{tpup+-}. Note that
\re{el1-+} ensures that $\Upsilon$ commutes with $\hat\tau_i$ and $E_{\mathbf{\hat\Lambda_i}}$. Then \re{LUht}, or equivalently
\re{quaeq} (i.e. \re{quaeqp}) as $\Upsilon$ commutes with $E_{\mathbf{\hat\Lambda}_i}$,
gives us \re{el1++}.  By \rl{clas}, \re{el1++}-\re{tpup+} are equivalent to \re{UtUt+}. %
\end{proof}
%

\pr{caseI} Let $\{\tau_{ij}\}$, $\{\tilde \tau_{ij}\}$, $\Phi_i$, and $\tilde\Phi_i$ be as in \rpa{ideal6}.
Decompose $\Phi_i=\Phi_{i1}\circ\Phi_{i0}^{-1}$   with $\Phi_{i1}\in\cL C^{\mathsf c}(\hat\tau_{i1}, \ldots, \hat\tau_{1p})$
and $\Phi_{i0}\in\cL C (\hat\tau_{i1}, \ldots, \hat\tau_{1p})$, and decompose $\tilde\Phi_i$ analogously.
Then $\{\{\tau_{1j}\}, \{\tau_{2j}\}\}$ and $\{\{\tilde\tau_{1j}\}, \{\tilde\tau_{2j}\}\}$ are equivalent under a mapping
that is tangent to the identity  if and only if $\Phi_{i1}=\tilde\Phi_{i1}$ for $i=1,2$.  Assume further that
 $\tau_{2j}=\rho\tau_{1j}\rho$ and $\tilde\tau_{2j}=\rho\tilde\tau_{1j}\rho$.
Then two families
are equivalent under a mapping that is tangent to the identity and commutes with $\rho$ if and only if $\Phi_{11}=\tilde\Phi_{11}$.  
\end{prop}
\begin{proof} When restricting to changes of coordinates that are tangent to the identity, we have $\Upsilon=I$ in \re{UtUt+}.
Also \re{quaeqp} holds trivially as $\nu_i$ is the identity.
By the uniqueness
of the decomposition $\Phi_i=\Phi_{i1}\Phi_{i0}^{-1}$, \re{tpup+-} becomes $\Phi_{i1}=\tilde\Phi_{i1}$.
\end{proof}

We consider a general case   without restriction on coordinate changes.

\le{rbnu}Let $\boldsymbol{\Upsilon}=\diag(\mathbf a,\mathbf a)$ with $\mathbf a\in\{-1,1\}^p$. Let $\mathbf B$ be a nonsingular $p\times p$ matrix and let $\nu$ be a permutation of $\{1,\dots, p\}$. Then
\eq{Usbs}
\Upsilon^{-1}\circ   B_*\circ Z_j\circ
  B_*^{-1}\circ \Upsilon=    B_*\circ Z_{\nu(j)}\circ
  B_*^{-1}, \quad 1\leq j\leq p
  \eeq
if and only if
\eq{genB}
\mathbf B=(\diag\mathbf a)^{-1} \mathbf B(\diag_{\nu}\mathbf d).
\eeq
In particular, if   $\mathbf B$ is an upper or lower triangular matrix, then $\nu=I$ and $\mathbf d=\mathbf a$.
\ele
\begin{proof}Let $\tilde{\mathbf Z}_j=\diag(1,\dots, -1,\dots, 1)$ be the matrix where $-1$ at the $j$-th place.   Set $\mathbf C:=\mathbf B^{-1}\diag\mathbf a\, \mathbf B$ and   $\mathbf C=(c_{ij})$. In  $2\times 2$ block matrices, we see that \re{Usbs} is equivalent to $\mathbf C\tilde{\mathbf Z}_{\nu(j)}=\tilde{\mathbf Z}_{j}\mathbf C$, i.e.
$$
-c_{i\nu(j)}=c_{i\nu(j)}, \quad \quad i\neq j.
$$
Therefore, $\mathbf C=\diag_{\nu}\mathbf d$ with $d_j=c_{j\nu(j)}$, by \re{dfndiag}.  \end{proof}
 We will
assume that $M$ is a higher order perturbation of non-resonant product quadric. Let us recall $\hat \sigma$
be given by \re{hsixi} and define $\hat\tau_{ij}$ as follows:
\gan\label{7convnfsi}
\hat\sigma:\begin{cases}\xi'_j = \hat M_j(\xi\eta)\xi_j\\ \eta'_j=\hat M_j^{-1}(\xi\eta)\eta_j,\end{cases}\quad
\hat \tau_{ij}:\begin{cases}\xi'_j  =\hat \Lambda_{ij} (\xi\eta)\eta_j\\ \eta'_j = \hat\Lambda_{ij}^{-1}(\xi\eta)\xi_j\\ \xi'_k = \xi_k\\
 \eta'_k = \eta_k,\quad k\neq j\end{cases}
\end{gather*}
with $\hat\Lambda_{2j}=\hat\Lambda_{1j}^{-1}$ and $\hat M_j=\hat\Lambda_{1j}^2$.
Let $\hat\tau_i=\hat\tau_{i1}\cdots\hat\tau_{1p}$.  Recall that $E_{\boldsymbol{\hat\Lambda}_i}$ in \re{bise}.
\begin{prop}\label{ideal6+}
   Let $\{\tau_{11},\dots, \tau_{1p},\rho\}$ be the family of involutions with $\rho$ be given by \rea{57rhoz}. Suppose that the linear parts of $\tau_{1j}$ are given by \rea{ltij} and associated $\sigma$ is non-resonant, while the associated matrix $\mathbf B$ for $\{T_{1j}\}$ satisfies the non-degeneracy condition that \rea{genB} holds only for $\nu=I$.
Let $\hat\sigma$ be the formal normal form $\hat\sigma$ of the $\sigma$ associated to $M$ that is
given by  \rea{hsixi} in which  $\log\hat M$ is 
in the 
 formal normal form  \rea{hfj1}. Let $\hat\tau_{1j}$ be given by \rea{tPtPhat} and $\hat\tau_{2j}=\rho\hat\tau_{1j}\rho$.
  In suitable formal coordinates
  the   involutions $\tau_{ij}$ 
  have the form
  \eq{tPtP}
  \tau_{1j}=\Psi\hat\tau_{ij} \Psi^{-1},\quad \tau_{2j}=\rho \tau_{1j}\rho, \quad
  \quad \Psi\in\cL C(\hat\tau_1)\cap C^c(\hat\tau_{11}, \ldots, \hat\tau_{1p}), \quad \Psi'(0)=\mathbf I.
   \eeq
Moreover,  if $\tilde\tau_{11}, \dots, \tilde\tau_{1p}$ have the form \rea{tPtP} in which   $\Psi$
is replaced by $\tilde\Psi$. Then there exists a formal mapping $R$ commuting with $\rho$ and transforms
the family
 $\{\tilde\tau_{11},\dots,\tilde\tau_{1p}\}$ into
$\{ \tau_{11},\dots, \tau_{1p}\}$  if and only if  $R$ is an $R_\e$ defined by \rea{upsi}
and \eq{redequiv}
\tilde\Psi=R_\e^{-1}\Psi R_\e,  \quad R_\e\rho=\rho R_\e.
\eeq
In particular, $\{\tau_{11},\dots,\tau_{1p},\rho\}$ is formally equivalent to $\{\hat\tau_{11},\ldots,\hat\tau_{1p},\rho\}$
if and only if $\Psi$ in \rea{tPtP} is the identity map.
 \end{prop}
\begin{proof} We apply \rp{ideal6}.
We need to refine the equivalence relation \re{quaeqp}-\re{el1-+}.  First we know that $\re{el1-+}$
means that $\Upsilon=R_\e$ and it commutes with $\rho$.   It remains to refine \re{tpup+-}.
 We have $\Phi_2=\rho\Phi_1\rho$.     By assumption, we know that $\nu_1$ in \re{quaeqp} must be the identity.
Then    $\Phi_1\in\cL C^{\mathsf c}(\hat\tau_{11}, \ldots, \hat\tau_{1p})$ implies that
    $\Upsilon^{-1}\Phi_1\Upsilon\in \cL C^{\mathsf c}(\hat\tau_{11}, \ldots, \hat\tau_{1p})$;  indeed by \re{genB} we have
  $$
 \Upsilon { E}_{\mathbf{\Lambda}_i}\circ({ B}_i)_*={ E}_{\mathbf{\Lambda}_i}\circ\Upsilon\circ ({ B}_i)_* ={ E}_{\mathbf{\Lambda}_i}\circ ({ B}_i)_*\circ D,\quad\tilde{\mathbf D}= \diag(\diag\mathbf a,\diag \mathbf d). $$
 Note that $\psi_0=(U,V)$ is in ${\cL C}_2(Z_1,\ldots, Z_p)$ if and only if
  $$
 U(\xi,\eta)= \tilde U(\xi,\eta_1^2,\dots, \eta_p^2), \quad V_j(\xi,\eta)=\eta_j\tilde V_j(\xi,\eta_1^2,\dots, \eta_p^2).
 $$
 Let $ \psi_1=(U,V)$ be in ${\cL C}^{\mathsf c}_2(Z_1,\ldots, Z_p)$, i.e.
 $$
 U(\xi,\eta)=\sum_i\eta_i\tilde U_i(\xi,\eta_1^2,\dots,\eta_i^2), \quad
 V_j(\xi,\eta)=V^*_j(\xi,\eta)+\eta_j\sum_i\eta_i\tilde V_i(\xi,\eta_1^2,\dots,\eta_i^2),
 $$
 where $V_j^*(\xi,\eta)$ is independent of $\eta_j$.   Since $\mathbf D$ is diagonal, then $D \psi_1D^{-1}$ is in ${\cL C}^{\mathsf c}_2(Z_1,\ldots, Z_p)$.    This shows that conjugation by  $\Upsilon$ preserves ${\cL C}^{\mathsf c}(\hat\tau_{11},\dots,\hat\tau_{1p})$.
 Also $\Upsilon$ commutes with each $\hat\tau_{1j}$.  Hence, it preserves ${\cL C}(\hat\tau_{11},\dots,\hat\tau_{1p})$. By the uniqueness of decomposition,  \re{tpup+-} becomes
\eq{}
\nonumber
\tilde\Phi_{11}=\Upsilon^{-1}\Phi_{11}\Upsilon, \quad \tilde\Phi_{10}^{-1}=\Upsilon^{-1} \Phi_{10}^{-1}\Upsilon^*_1 
\Upsilon.
\eeq
The second equation defines $\Upsilon_1^*$ that is in $\cL C(\hat\tau_{11}, \ldots, \hat\tau_{1p})$
as $\Upsilon, \Phi_{10}, \tilde\Phi_{10}$ are in the centralizer.
Rename $\Phi_{11},\tilde\Phi_{11}$ by $\Psi,\tilde\Psi$.
This shows that the equivalence relation is reduced to  \re{redequiv}.  \end{proof}

We now derive the following formal normal form.
\begin{thm}\label{mnorm} Let $M$ be a real analytic submanifold   that is a higher order perturbation of a non-resonant
product quadric.
Assume that the formal normal form $\hat\sigma$ of the $\sigma$ associated to $M$ is
given by  \rea{hsixi} in which  $\log\hat M$ is tangent to the identity and in the 
formal normal form  \rea{hfj1}. Let $ E_{\boldsymbol{\hat\Lambda_1}}$ be defined by \rea{t1jd+l}.
Then $M$
is formally equivalent to a formal submanifold in the $(z_1, \ldots, z_{2p})$-space defined by
\begin{equation}\label{tMzp}
\nonumber
 \tilde M\colon   z_{p+j}=(\la_j^{-1}U_j(\xi,\eta)-V_j(\xi,\eta))^2, \quad1\leq j\leq p, \end{equation}
where  $(U,V)=E_{\boldsymbol{\hat\Lambda_1}(0)} E_{\boldsymbol{\hat\Lambda}_1}^{-1}\Psi^{-1}$,
 $\Psi$ is tangent to the identity and 
  in $\cL C(\hat\tau_1)\cap\cL C^{\mathsf c}(\hat\tau_{11},\ldots, \hat\tau_{1p})$, defined in \rla{cnnl}, and
$\xi,\eta$ are solutions to
$$
z_j=U_j(\xi,\eta)+\la_jV_j(\xi,\eta), \quad \ov z_j=\ov{U_j\circ\rho(\xi,\eta)}+\ov\la_j\ov{V_j\circ\rho(\xi,\eta)},
\quad 1\leq j\leq p.$$
Furthermore,
the $\Psi$ is uniquely determined up to   conjugacy  $R_\e\Psi R_\e^{-1}$ by an
 involution  $R_\e\colon \xi_j\to\e_j\xi_j,
\eta_j\to\e_j\eta_j$ for $1\leq j\leq p$   that commutes with $\rho$, i.e. $\e_{s+s_*}=\e_s$.
The formal holomorphic automorphism group of $\hat M$ consists of
 involutions of the form
$$
L_\e\colon z_j\to\e_j z_j, \quad  z_{p+j}\to z_{p+j}, \quad 1\leq j\leq p
$$
with $\e$ satisfying  $R_\e\Psi=\Psi R_\e$   and
$\e_{s+s_*}=\e_s$.   If the $\sigma$   associated to  $M$ is holomorphically equivalent to  a Poincar\'e-Dulac normal
form, then  $\tilde M$ can be achieved by a holomorphic transformation too.
\end{thm}
\begin{proof}  We fist choose linear coordinates so that the linear parts of $\{\tau_{11},\dots, \tau_{1p},\rho\}$ are in the
normal form in
\rla{t1t2sigrho}.
We apply \rp{ideal6+} and assume that $\tau_{ij}$ are already in the normal form. The rest of proof is essentially in \rp{mmtp}
and we will be brief. Write $T_{1j}=E_{\boldsymbol{\hat\Lambda}_1(0)} \circ Z_j\circ
 E_{\boldsymbol{\hat\Lambda}_1(0)}^{-1}$. Let $\psi=(U,V)$ with $U,V$ being given in the theorem.  We obtain
 $$
\tau_{1j}=\Psi\hat\tau_{1j} \Psi^{-1}= \psi^{-1} T_{1j}\psi, \quad 1\leq j\leq p.
 $$
 Let $f_j=\xi_j+\la_j\eta_j$ and $h_j=(\la_j\xi_j-\eta_j)^2$.  The invariant functions of $\{T_{11}, \ldots, T_{1p}\}$ are generated by
 $f_1,\ldots, f_p,   h_1,\ldots, h_p$. This shows that the invariant functions of $\{\tau_{11}, \ldots, \tau_{1p}\}$ are generated by
 $f_1\circ\psi, \ldots, f_p\circ\psi, h_1\circ \psi, \ldots, h_p\circ\psi$.
 Set $g:=\ov{f\circ\psi\circ\rho}$.
 We can verify that $\phi=(f\circ\psi,g)$ 
  is biholomorphic. Now
 $
 \phi\rho\phi^{-1}=\rho_0.
 $
 Let $ M$ 
  be defined by
 $$
 z_{p+j}=E_j(z',\ov z'), \quad 1\leq j\leq p,
 $$
 where $ E_j= h_j\circ\phi^{-1}$.  Then  $E_j\circ\phi$ and $z_j\circ\phi=f_j$
 are invariant by $\{\tau_{1k}\}$. This shows that \{$\phi\tau_{ij}\phi^{-1}\}$ has the same
 invariant functions as deck transformations of $\pi_1$
 of the complexification $\cL M$ of $M$.   By Lemma 2.5 in \cite{GS15},
 $\{\phi\tau_{1j}\phi^{-1}\}$ agrees with the unique set of generators for
  the deck transformations of $\pi_1$.
 Then $ M$ is a realization of $\{\tau_{11}, \ldots, \tau_{1p},\rho\}$.

 Finally, we identity the formal automorphisms of $ M$, which fix the origin.  For such an automorphism $F$ on ${\cc}^n$, define $\tilde F(z',w')=(F(z', E(z',w')), \ov F(w',\ov E'(w',z'))$ on  $\cL M$. Then $\phi^{-1}\tilde F\phi$ preserves $\{\tau_{11},\dots, \tau_{1p},\rho\}$. By \rp{ideal6+}, $\phi^{-1}\tilde F\phi=R_\e$,  $R_\e\rho=\rho R_\e$, and
 $R_\e\Psi=\Psi R_\e$.  Given   \re{upsi}, we write $R_\e=(L_\e',L_\e')$.  In view of $(U,V)=E_{\boldsymbol{\hat\Lambda_1}(0)} E_{\boldsymbol{\hat\Lambda}_1}^{-1}\Psi^{-1}$, we obtain that $L'_\e U=UR_\e$ and $L'_\e V=VR_\e$.
 Since $z_j=U_j(\xi,\eta)+\la_jV_j(\xi,\eta)$ and $z_{p+j}=(\la_j^{-1} U_j(\xi,\eta)-V_j(\xi,\eta))^2$, then $z'\circ\tilde F=L_\e'z'$ and $z''\circ\tilde F=z''$ as functions in $(z',w')$. This shows that $z'\circ F=L_e'z'$ and $z''\circ F=z''$ as functions in $(z',z'')$. Therefore, $F=L_\e$.
  \end{proof}

\begin{rem} Let $b$ be on the unit circle with $0\leq\arg b<\pi$. Let
$$
\mathbf B=\begin{pmatrix}
1   & b\\
 \tilde b&1
\end{pmatrix}, \quad |\tilde b|\leq 1, \quad b\tilde b\neq1.
$$
One can check  that  \re{genB} admits  a solution $\nu\neq I$ if and only if $\tilde b=-\ov b$.   \end{rem}

\bibliographystyle{alpha}
\def\cprime{$'$}

\end{document}